\documentclass[12pt]{amsart}
\usepackage{amssymb,amsmath,amsthm,}  
\usepackage{graphicx}

\usepackage{color}
\definecolor{Blue}{rgb}{0.3,0.3,0.9}

\newtheorem{thm}{Theorem}[section]
\newtheorem{cor}[thm]{Corollary}
\newtheorem{lem}[thm]{Lemma}
\newtheorem{prop}[thm]{Proposition}

\newtheorem{problem}[thm]{Problem}
\newtheorem{basicproblem}[thm]{Basic problem}

\newtheorem{que}[thm]{Question}

\theoremstyle{definition}
\newtheorem{defn!}[thm]{Definition}
\newtheorem{nota!}[thm]{Notation}
\newtheorem{exe!}[thm]{Example}
\newtheorem{exes!}[thm]{Examples}
\newtheorem{rem!}[thm]{Remark}
\newtheorem{rems!}[thm]{Remarks}
\newtheorem{remind!}[thm]{Reminder}
\newtheorem{que!}[thm]{Question}

\newcommand{\Z}{\mathbf{Z}}

\newcommand{\N}{\mathbf{N}}
\newcommand{\R}{\mathbf{R}}
\newcommand{\C}{\mathbf{C}}
\newcommand{\Q}{\mathbf{Q}}
\newcommand{\un}{\textnormal{un}}

\title[Finitely presented covers]
{Amenable groups without finitely presented amenable covers
}

\author[M.G.\ Benli, R.\ Grigorchuk, and P.\ de la Harpe]
{Mustafa G\"okhan
Benli, Rostislav Grigorchuk \\ and  Pierre de la Harpe
\\
with a post-scriptum by Ralph Strebel}

\address{M.B.\ and R.G.:  Department of Mathematics \\
Mailstop 3368 \\ 
\newline
Texas A{\&}M University \\ College Station \\
TX 77843--3368 \\ USA.}
\email{mbenli@math.tamu.edu \\ grigorch@math.tamu.edu}

\address{P.H.:  Section de math\'ematiques \\
Universit\'e de Gen\`eve \\
C.P.~64 \\
\newline
CH--1211 Gen\`eve~4 \\ Suisse.}
\email{Pierre.delaHarpe@unige.ch}

\address{R.S.: 
Universit\'e de Fribourg \\
D\'epartement de math\'ematiques \\
\newline
Chemin du Mus\'ee 23 \\
CH--1700 Fribourg \\ Suisse.}
\email{ralph.strebel@unifr.ch}

\date{March 20, 2013, and April 27 for the Post Scriptum}


\thanks{We are grateful to the Swiss National Science Foundation
for supporting a visit of the first author to Geneva.
The first author is grateful to the department of mathematics, University of  Geneva, 
for their hospitality during his stay.
The first and the second authors are partially supported 
by NSF grant DMS -- 1207699.
}


\begin{document}

\begin{abstract} 
The goal of this article is to study results and examples concerning 
finitely presented covers of finitely generated amenable groups.
We collect examples of groups $G$ with the following properties:
(i) $G$ is finitely generated,
(ii) $G$ is amenable, e.g.\  of intermediate growth,
(iii) any finitely presented group with a quotient isomorphic to $G$
contains non-abelian free subgroups,
or the stronger 
(iii') any finitely presented group with a quotient isomorphic to $G$
is large.
\end{abstract}

\maketitle

\footnotesize
After submission of the final version of this paper to the
\emph{Bulletin of Mathematical Sciences} \hskip.1cm
\verb+http://www.springer.com/birkhauser/mathematics/journal/13373+~,
we have received a note of Ralph Strebel on an construction
of Yves de Cornulier, providing a strenghtening of Corollary \ref{BNSrem}.
This note, slightly modified, is added as a post-scriptum to the present file.
\normalsize

\section{\textbf{Introduction}}
\label{sectionIntro}

\subsection{Motivation}
\label{Motivation}
In the study of finiteness conditions on groups, 
the following type of question is natural:

\begin{que}
\label{initialquestion}
Given a Property ($\mathcal P$) of groups, 
does any finitely generated group with~($\mathcal P$) 
have a finitely presented cover with ($\mathcal P$) ?
\end{que}

A \textbf{cover} of a group $G$ 
is a group $E$ given together with an epimorphism $E \twoheadrightarrow G$.
A \emph{cover} fits into an \emph{extension}
$1  \rightarrow  
N \rightarrow 
E   \twoheadrightarrow 
G \rightarrow 
1$.
\par

The answer can be positive for trivial reasons,
for example when Property ($\mathcal P$) holds for free groups
(such as exponential growth)
or when Property ($\mathcal P$) implies finite presentation
(such as nilpotency, or polynomial growth).
The answer is also positive when
($\mathcal P$) is Kazhdan's Property ($\mathcal T$),
by a non-trivial result of Shalom \cite{Shal--00}.
\par

Here, we concentrate on Question \ref{initialquestion} for amenability,
a case in which the answer is negative.
Indeed, the goal of this article is to study examples and results concerning 
\emph{finitely generated amenable groups
that do not have finitely presented amenable covers.}
\par

Our motivation is better to understand
the class of \emph{finitely presented amenable groups},
and related groups.
To describe the situation, it is convenient to introduce
the class $\mathcal A \mathcal G$ of amenable groups,
and the subclass $\mathcal E \mathcal G$ of elementary amenable groups,
which is easier to understand
(Appendix \ref{AppendixDGrowthAmenability} is a reminder
on these classes, and on growth).
\par

On the one hand, many examples are known in $\mathcal E \mathcal G$ 
of finitely presented soluble groups,
including the standard examples described in the exposition \cite{Stre--84},
the Baumslag-Remeslennikov group,
which is a finitely presented metabelian group 
with derived group free abelian of infinite rank 
(references in our Remark \ref{RemarkOnHall}.c),
and the Kharlampovich group, 
which is a finitely presented soluble group with unsoluble word problem
(references in the proof of Proposition \ref{HallAbelsEtc}).
There are also finitely presented groups in $\mathcal E \mathcal G$
which are not virtually soluble (Example \ref{Houghton}).

In the complement $\mathcal A \mathcal G \smallsetminus \mathcal E \mathcal G$,
a reasonable number of \emph{finitely generated} examples are known, 
in particular the \emph{groups of intermediate growth}.
Also known are \emph{finitely presented} examples,
all closely related to self-similar groups.
But not enough examples are known to allow one to guess at a general picture;
this is illustrated by the following old and basic problem in the theory of growth of groups
(see, e.g., 
\cite{Kour--06}, Problems 4.5.b,  9.8, and 10.11):

\begin{basicproblem}
\label{basicproblem}
Does there exist a finitely presented group of intermediate growth?
\end{basicproblem}


The answer does not seem to be at hand.
Here is a possibly easier problem:

\begin{problem}
\label{problextintgr}
Does there exist a finitely generated group of intermediate growth 
that is the quotient of a finitely presented group
without non-abelian free subgroups?
or the quotient of a finitely presented amenable group?
\end{problem}

Finitely generated elementary amenable groups
are never of intermediate growth \cite{Chou--80},
so that Problems \ref{basicproblem} and \ref{problextintgr} 
do not arise in $\mathcal E \mathcal G$.
Problem \ref{problextintgr} involves \emph{two} questions,
because there exist finitely presented groups that are non-amenable
and do not contain non-abelian free subgroups \cite{OlSa--02}.
\par

Strictly speaking, the solution 
of the analogue of Problem \ref{basicproblem} is known for the class
$\mathcal A \mathcal G \smallsetminus \mathcal E \mathcal G$.
Indeed, one knows a few sporadic examples of finitely presented groups
in this class,
such as the HNN-extension $\widehat{\mathfrak G}$ of the Grigorchuk group \cite{Grig--98},
and the HNN-extension of the Brunner-Sidki-Vieira group
that appears in \cite[Proposition 6]{GrZ--02b}
(the latter group is amenable by \cite{BaKN--10} 
and not elementary amenable by \cite{GrZ--02b}).
The finitely presented HNN-extension of the basilica group 
that appears in \cite[Theorem 1.7]{GrZ--02b}
is not even in the class
$\mathcal A \mathcal G \smallsetminus \mathcal S \mathcal G$,
defined in Appendix \ref{AppendixDGrowthAmenability};
see also \cite[Theorem 12]{BaVi--05}.
But, despite these examples, 
very little is known about finitely presented groups in
$\mathcal A \mathcal G \smallsetminus \mathcal E \mathcal G$,
and it would be interesting to find methods providing new examples;
a priori, one could hope and try finitely presented covers,
but our paper shows that this does not look very promising.
\par

Before stating the next problem, 
here is a definition:
the \textbf{elementary amenable radical} 
$\operatorname{Rad}_{\operatorname{ea}}(G)$ of a group $G$
is its unique maximal normal elementary amenable subgroup.
Note that $\operatorname{Rad}_{\operatorname{ea}}(G)$ is contained in the
amenable radical of $G$, that appears (but for its name) 
in \cite[Lemma 1 of Section 4]{Day--57}.

\begin{problem}
\label{AddedOnMay21}
Is there a finitely presented group $G$
in $\mathcal A \mathcal G \smallsetminus \mathcal E \mathcal G$,
with $\operatorname{Rad}_{\operatorname{ea}}(G) = \{1\}$, 
that has infinitely presented\footnote{A
finitely generated group is ``infinitely presented'' if it is not finitely presented.}  
quotients?
or that has uncountably many pairwise non-isomorphic quotients?
\end{problem}

\subsection{First examples}
When ($\mathcal P$) is amenability,
we know two approaches to Question \ref{initialquestion}:
One is based on the theory of $\Sigma$-invariants, 
developed in a series of papers by Bieri, Strebel, and Neumann
(see Appendix \ref{AppendixCBNS}). 
The other involves self-similar groups and approximation methods
(see Sections \ref{sectionSelfsimilar} and \ref{sectionMark+Chab}). 
\par

The first result we quote is due to Bieri and Strebel
\cite[Theorem 5.5 and Corollary 5.6]{BiSt--80}.
On the one hand, we reformulate it
in a slightly weaker version, assuming that $E$ is finitely presented
(instead of assuming that $E$ has Property FP$_2$, as in the original paper);
on the other hand, we formulate it for \emph{virtually} metabelian groups,
because this follows immediately from the case of metabelian groups.

\begin{thm}[Bieri-Strebel]
\label{ThBieriStrebel}
Let $G$ be a virtually metabelian group 
that is finitely generated and infinitely presented. 
\par

Any finitely presented cover $E$ of $G$ 
has non-abelian free subgroups.
In particular, $E$ is non-soluble, indeed non-amenable.
\end{thm}

The proof of Theorem \ref{ThBieriStrebel} can now be understood
in a much simpler way than in \cite{BiSt--80},
as we indicate in Corollary \ref{BNScorollary}. 
The reason is that we can use better invariants,
namely those defined in \cite{BiNS--87} 
and their later reformulations
(see Strebel's exposition in progress \cite{Stre}).
\par

Examples of metabelian groups 
that are finitely generated and infinitely presented include 
matrix groups like
$
\left(
\begin{array}{cc}
  \big(\frac{\ell}{m} \big)^{\Z} & \Z \left[ \frac{1}{\ell m} \right] \\
  0   & 1 \\
\end{array} 
\right ) 
$
where $\ell, m \ge 2$ are coprime integers, 
wreath products $A \wr \Z$
where $A \ne \{1\}$ is a finitely generated abelian group,
and free metabelian groups $\operatorname{FSol}(k,2) := F_k / [[F_k, F_k],[F_k, F_k]]$,
where $F_k$ is the free group on $k \ge 2$ generators.
\par

Appendix \ref{AppendixAMetabelian} includes a reminder 
on metabelian groups.
Appendix~\ref{AppendixBWreath} is a reminder 
on wreath products and lamplighter groups.
Appendix~\ref{AppendixCBNS} is a reminder
on Bieri-Neumann-Strebel invariants.
Proposition \ref{AwrZ}, Corollary \ref{freeSol},  Proposition \ref{BS+GrMa},
Corollary \ref{BNScorollary}, and Corollary \ref{BNSrem} give
examples of finitely generated groups 
of which all finitely presented covers have non-abelian free subgroups.

\vskip.2cm

We denote by $\mathfrak G$ the \emph{Grigorchuk group}\footnote{Two
of the authors insist that we mention this terminology, often used today.
The third author submits.},
introduced in \cite{Grig--80};
see Example \ref{Ex1stGrig}.
Recall here that $\mathfrak G$ is finitely generated, 
indeed generated by four involutions
traditionally written $a,b,c,d$.
This group has many remarkable properties, including that
of being an infinite $2$-group of intermediate growth;
in particular it is amenable.
The group $\mathfrak G$ has a presentation, due to Lysenok,
with four generators and infinitely many relators.
Appropriate finite subsets of these relators naturally define 
a sequence $(\mathfrak G_n)_{n \ge 0}$ of four-generated finitely presented groups
converging to $\mathfrak G$ (Definition \ref{defGn}).
\par

There is a reminder
on convergence of groups in Section \ref{sectionMark+Chab}.

\begin{thm}[\cite{GrHa--01}]
\label{GrHa+i}
Any finitely presented cover of the group $\mathfrak G$ 
is large.
\end{thm}

Recall that a group is \textbf{large} if it contains a subgroup of finite index
that has non-abelian free quotients. 
Note that large groups have non-abelian free subgroups.
\par

In view of Proposition \ref{main} below,
Theorem \ref{GrHa+i} is a straightforward consequence of the main result of \cite{GrHa--01};
see also Theorem \ref{FirstThmOnSSGps} and Example \ref{Ex1stGrig}.
More precisely, it was shown in \cite{GrHa--01}
that each $\mathfrak G_n$ is virtually 
a direct product of finitely generated non-abelian free groups;
this has been sharpened in \cite{BaCo--06},
and can be further improved:

\begin{thm}[Section \ref{sectionGrigGroup}]
\label{GrHa+ii}
Let $\mathfrak G$ and $\mathfrak G_n$ be as above.
For each $n \ge 0$, 
the group $\mathfrak G_n$ has a normal subgroup $H_n$ of index $2^{2^{n+1}+2}$ 
that is isomorphic to the direct product of $2^n$ free groups of rank $3$. 
\end{thm}

Another way of sharpening Theorem \ref{GrHa+i} is given in Proposition \ref{GrigNotFP2},
where the condition of finite presentability for the cover of $\mathfrak G$
is replaced by the weaker Condition FP$_2$.

\vskip.2cm

We denote by $\mathfrak B$ the Basilica group.
Recall here that $\mathfrak B$ can be generated by two elements,
and is an amenable torsion-free group of exponential growth.
See Example \ref{ExBasilica} and Appendix \ref{AppendixDGrowthAmenability} 
for some other properties of $\mathfrak B$.

\begin{thm}[Erschler]
\label{Erschler}
Any finitely presented cover of the Basilica group $\mathfrak B$ 
has non-abelian free subgroups.
\end{thm}

Given an invertible automaton $(A, \tau)$ over a finite alphabet $X$, 
Erschler introduces a notion of 
``automatically presented group $G^*(A, \tau)$ generated by the automaton''
(see \cite{Ersc--07} for details).
She shows that, if $G^*(A, \tau)$ is not virtually abelian, 
then any finitely presented cover of it has non-abelian free subgroups.
For the Basilica automaton, the group $G^*(A, \tau)$ coincides with $\mathfrak B$,
and Theorem \ref{Erschler} follows.
\par

As noted in \cite{Ersc--07},
these arguments do not apply to the group $\mathfrak G$ of Theorem \ref{GrHa+i};
this is due to the fact that, for $(A,\tau)$ the automaton of $\mathfrak G$,
the cover $G^*(A, \tau) \twoheadrightarrow \mathfrak G$
has a non-trivial kernel.
\par

In our setting, for $G = \mathfrak B$,
the universal contracting cover $G_0$ of Definition \ref{DefG0} is free of rank $2$
(see Example \ref{ExBasilica}), and Theorem \ref{Erschler} 
follows from our Theorem \ref{FirstThmOnSSGps}.
Indeed, our argument shows that 
\begin{center}
\emph{any finitely presented cover of $\mathfrak B$ is large.}
\end{center}

\vskip.2cm

\textbf{Notation:} In this paper, the symbols $\mathfrak G$ and $\mathfrak B$
will be used \emph{only} for the two groups introduced above.
In Examples \ref{IMGz^2+i} to \ref{Hanoi},
we introduce four other groups, with their usual notation in this subject:
the iterated monodromy group of $z^2+i$ denoted by $\mathfrak J$, 
the Gupta-Sidki group $\mathfrak G \mathfrak S$,  
the Fabrykowski-Gupta group $\mathfrak F \mathfrak G$, 
and the Hanoi Towers group $\mathfrak H$.

\vskip.2cm

We show in Section \ref{sectionSelfsimilar} that
Theorem \ref{FirstThmOnSSGps} implies
Theorems \ref{GrHa+i} on $\mathfrak G$, 
\ref{Erschler} on $\mathfrak B$,
and \ref{thmetc} on
$\mathfrak J$, $\mathfrak G \mathfrak S$, 
$\mathfrak F \mathfrak G$, and $\mathfrak H$.
We state now a shorthand version of \ref{FirstThmOnSSGps}.
Notation and technical terms are defined in Section  \ref{sectionSelfsimilar}.

\begin{thm}
\label{mainsection2}
Let $G$ be an infinite finitely generated self-similar group.
Assume that $G$ is contracting faithful self-replicating.
Let $G_0$ denote a standard contracting cover of $G$, as in Definition \ref{DefG0Tilde}.
\par

If $G_0$ has non-abelian free subgroups, then so does any finitely presented cover of $G$.
\par

If $G_0$ is large, any finitely presented cover of $G$ is large.
\end{thm}

\subsection{Infinitely more examples}
The group $\mathfrak G$ has uncountably many\footnote{Here
 and elsewhere, ``uncountably many'' groups means 
``uncountably many \emph{pairwise non-isomorphic}'' groups.}  
relatives $G_\omega$
introduced in \cite{Gri--84a}.
Each of these groups is generated by a set $S_\omega$ of four involutions.
They are parameterized by the space $\Omega = \{0,1,2\}^{\N}$.
We have $\mathfrak G = G_\omega$
when $\omega$ is the $3$-periodic sequence $012012012\cdots$.
Let $\Omega_+$ denote the complement in $\Omega$ 
of the space of eventually constant sequences.
Section \ref{sectionOmega} contains 
a description of the family $(G_\omega)_{\omega \in \Omega_+}$.

\begin{thm}
\label{ThGri84}
For $\omega \in \Omega_+$, the group $G_\omega$
is of intermediate growth,
and any finitely presented cover of $G_\omega$ is large.
\end{thm}

Note also the following straightforward consequence 
of Theorems \ref{GrHa+i}, \ref{Erschler}, and \ref{ThGri84}:

\begin{cor}
\label{corollaryBE}
Let $H$ be a finitely generated cover of one of
$\mathfrak G$ (as in Theorem \ref{GrHa+i}),
$\mathfrak B$ (as in Theorem \ref{Erschler}),
or $G_\omega$, $\omega \in \Omega_+$ (as in Theorem \ref{ThGri84}).
Any finitely presented cover of $H$ is large.
\end{cor}

There are several interesting classes of groups 
that qualify to be the $H$ of Corollary \ref{corollaryBE}:

(i) 
The uncountably many groups of \cite{Ersc--04},
which are finitely generated, of intermediate growth, 
and not residually finite, 
each one being a \emph{central} cover of $\mathfrak G$.

(ii) 
The groups of \cite{BaEr--a},
which are finitely generated groups of intermediate growth,
with exactly known growth functions,
each one being a cover of $\mathfrak G$.

(iii)
Permutational wreath products of the form 
$A \wr_X G_\omega$,
where $A \ne \{1\}$ is a finite group
and $G_\omega$ is as in Theorem \ref{ThGri84}  \cite{BaEr--b}.

\vskip.2cm

There is an uncountable family of finitely generated amenable simple groups,
which are topological full groups of minimal homeomorphisms of the Cantor space
\cite{JuMo}.
None of these groups is finitely presented 
(see \cite[Theorem 5.7]{Matu--06}, as well as \cite{GrMe}).
In our context, it is natural to formulate:

\begin{problem}
\label{ProblemOnJuMo}
Let $G$ be one of the finitely generated amenable simple groups
that appears in \cite{JuMo}.
Does $G$ have an amenable finitely presented cover?
\end{problem}

\begin{rem!}  
We state two kinds of results concerning appropriate covers:
some establish that the covers are large (see Theorem \ref{GrHa+i}),
others, weaker, that they contain non-abelian free groups (see Theorem \ref{ThBieriStrebel}).
Strong statements do not hold in all cases.
\par

For example, let
$G = \operatorname{Met}(k,\ell) \simeq \Z \left[ \frac{1}{\ell m} \right] \rtimes_{\ell / m} \Z$,
for two coprime integers $k,\ell \ge 2$ (Definition \ref{defmet(l,m)}).
Then $G$ satisfies the hypothesis of Theorem \ref{ThBieriStrebel}.
The Baumslag-Solitar group $\operatorname{BS}(\ell,m)$ 
is a finitely presented cover of $G$; 
it is known that $\operatorname{BS}(\ell,m)$ has non-abelian free subgroups,
but is not large (Proposition \ref{BSbasique}.iv).

\end{rem!}

\subsection{A dual to Question 1.1}
 \label{DualQuestion1.1}
 Given a Property ($\mathcal P$) of groups, it is standard to ask
whether any countable group with ($\mathcal P$)
is a \emph{subgroup} of some finitely generated group with ($\mathcal P$).
For example, the answer is  known to be positive
if ($\mathcal P$) is solubility \cite{NeNe--59}, or amenability,
or elementary amenability \cite[Corollary 1.3]{OlOs}.
An earlier result of \cite{HiNN--49} also gives a positive answer to a similar question:
for a given integer $n$, if $G$ is a group that has a presentation
with a countable number (possibly infinite) of generators and $n$ relators,
then $G$ can be embedded
in a finitely presented group with $2$ generators and $n$ relators.
But the answer is negative if ($\mathcal P$) is nilpotency,
because any subgroup of a finitely generated nilpotent group is finitely generated.
It is also negative if ($\mathcal P$) is metabelianity:
indeed, for any prime $p$, the abelian quasi-cyclic group $\Z(p^\infty)$
cannot be embedded in a finitely generated metabelian group
\cite[Lemma 5.3]{NeNe--59}.

As a digression from our main theme,
and since recursively presented groups are mentioned in Appendix A,
we formulate one more question, which is in some sense dual to Question 1.1.

\begin{que}
\label{sgfp}
Given a Property ($\mathcal P$) of groups,
is any finitely generated recursively presented group with ($\mathcal P$)
a subgroup of some finitely presented group with ($\mathcal P$)?
\end{que}

The answer to Question \ref{sgfp} is known to be positive
if ($\mathcal P$) is metabel\-ianity (Proposition \ref{MetabelianWellBehaved}), 
or solubility of the word problem (\cite{Clap--67},
and also \cite[Theorem 2.8]{Mill--89}).
The answer is not known if ($\mathcal P$) is amenability \cite[Problem 1.7]{OlOs}.

\subsection{Plan of the paper}  
Section \ref{sectionSelfsimilar}:
Non-abelian free subgroups of finitely presented covers
of contracting self-similar groups.
Proofs of Theorems \ref{mainsection2} and \ref{FirstThmOnSSGps};
proofs of Theorems
\ref{GrHa+i}, \ref{Erschler}, and their analogues for
$\mathfrak J$, $\mathfrak G \mathfrak S$, 
$\mathfrak F \mathfrak G$, and $\mathfrak H$.

Section \ref{sectionMark+Chab}:
Marked groups and the Chabauty topology, a reminder.

Section \ref{sectionOmega}:
The analogue of Theorem \ref{GrHa+i}
for the family $\left( G_\omega \right)_{\omega \in \Omega}$ of \cite{Gri--84a},
Theorem \ref{ThGri84} and its proof.

Section \ref{sectionGrigGroup}:
The group of intermediate growth $\mathfrak G$, and the proof of
Theorem \ref{GrHa+ii}; this is a quantitative sharpening of Theorem \ref{GrHa+i}.
Proposition \ref{GrigNotFP2} on FP$_2$-covers of $\mathfrak G$.

Appendix \ref{AppendixAMetabelian}:
On soluble groups, metabelian groups, and finite presentations.

Appendix \ref{AppendixBWreath}:
On wreath products and lamplighter groups.

Appendix \ref{AppendixCBNS}:
On Bieri-Neumann-Strebel invariants.

Appendix \ref{AppendixDGrowthAmenability}: 
On growth and amenability.

Post-scriptum: \emph{A construction of large groups, due to Yves de Cornulier,}
 by Ralph Strebel.

\subsection{Acknowledgements}
We are most grateful to Yves de Cornulier and Ralph Strebel,
for generous comments on preliminary versions of this paper,
as well as to Laurent Bartholdi, Gilbert Baumslag, Marc Burger, 
Tullio Ceccherini-Silberstein, Anna Erschler, Sergei Ivanov,
Olga Kharlampovich, Volodymyr Nekrashevych, 
Alexander Ol'shanskii, Vitali\u{i} Roman'kov, and Mark Sapir,  
for useful discussions and mail exchanges.

\section{\textbf{Non-abelian free subgroups of finitely presented covers
of contracting self-similar groups}}
\label{sectionSelfsimilar}

Let $G,H$ be two groups and $X$ a $H$-set;
here, this means that $H$ acts on $X$ \emph{from the right}.
The corresponding \textbf{permutational wreath product} 
is the semi-direct product 
\begin{equation*}
G \wr_X H \, := \, G^{(X)} \rtimes H .
\end{equation*}
We denote by $G^{(X)}$ the group of functions
$(g_x)_{x \in X} : X \longrightarrow G, x \longmapsto g_x$ with finite support,
and we consider the action \emph{from the left} of $H$ on $G^{(X)}$,
for which the action of $h$ on $(g_x)_{x \in X}$ is $(g_{xh})_{x \in X}$.
Hence the product of two elements in $G \wr_X H$ is given by
\begin{equation*}
\left( (g_x)_{x \in X}, h \right)  \left( (g'_x)_{x \in X}, h' \right)
\, = \, 
\left( (g_xg'_{xh})_{x \in X}, hh' \right) .
\end{equation*}
In case $G$ acts \emph{from the right} on some set $W$, 
the group $G \wr_X H$ acts naturally \emph{from the right} on $W \times X$, by
\begin{equation*}
(w,y) \left( (g_x)_{x \in X}, h\right) \ = \, (wg_y, yh) .
\end{equation*}
In expositions of wreath products and self-similar groups,
choices of which groups actions are from the left and which from the right
vary from one paper to the other, 
but a (sometimes hidden) mixture of left actions and right actions
seems unavoidable.
\par

If $G,G'$ are two groups and $\alpha: G \longrightarrow G'$ a homomorphism,
we have a natural homomorphism
\begin{equation}
\label{WreathHomo}
\alpha \wr_X 1_H \, : \, 
G \wr_X H \longrightarrow G' \wr_X H , \hskip.2cm
\left( (g_x)_{x \in X}, h \right) \longmapsto \left( (\alpha(g_x))_{x \in X}, h \right) ,
\end{equation}
where $1_H$ stands for the identity automorphism of $H$.
\par

If $X$ is clear from the context, we write ``$\wr$'' for ``$\wr_X$''.
In particular, with $X = \{0, 1, \hdots, d-1 \}$ and $S_d$ the full symmetric group of $X$,
we write $G \wr S_d$ for $G \wr_X S_d$.
Also, we write $1_d$ for $1_{S_d}$, and
\begin{equation*}
((g_x)_{x \in X}, \tau) = (g_0, \hdots, g_{d-1}, \tau) \hskip.5cm \text{with} \hskip.2cm
g_0, \hdots, g_{d-1} \in G \hskip.2cm \text{and} \hskip.2cm \tau \in S_d
\end{equation*}
for a typical element of $G \wr S_d = G^{X} \rtimes S_d$.
\par

The \textbf{iterated wreath products} with $S_d$ are defined inductively, 
for each integer $n \ge 0$, by
\begin{equation*}
G \wr^n S_d \, = \, \left\{
\aligned
G \hskip2cm  &\text{for} \hskip.2cm n=0
\\
\left( G \wr^{n-1} S_d \right) \wr S_d \hskip.2cm &\text{for} \hskip.2cm n \ge 1 .
\endaligned
\right.
\end{equation*}
We have the following \textbf{associativity of permutational wreath products}:
for a $H$-set $X$ and a $K$-set $Y$, 
the canonical mapping
\begin{equation*}
\left\{
\aligned
(G \wr_X H) \wr_Y K 
\hskip1cm &\longrightarrow \hskip1cm 
G \wr_{X \times Y} (H \wr_Y K)
\\
\left( \big( (g_{x,y})_{x \in X} , h_y \big)_{y \in Y}, k \right) 
\hskip.2cm &\longmapsto \hskip.2cm
\left( \big( g_{x,y} \big)_{(x,y) \in X \times Y} , \big( (h_y)_{y \in Y},  k \big) \right)
\endaligned
\right.
\end{equation*}
is an isomorphism of groups
(this is standard, see e.g. \cite[Chapter 1, Theorem 3.2]{Meld--95}).
In particular, for $n \ge 1$, we have
\begin{equation*}
G \wr^n S_d \, = \, 
\left( G \wr^{n-1} S_d \right) \wr S_d \, = \,
\left( G \wr S_d \right) \wr^{n-1} S_d \, = \, G^{X^n} \rtimes S_d^{(n)} ,
\end{equation*}
where 
$S_d^{(n)} = \left( \cdots (S_d \wr S_d) \wr \cdots \right) \wr S_d
= S_d^{(n-1)} \wr S_d$
is the appropriate permutation group of $X^n$, acting here \emph{from the right}.
We write
\begin{equation*}
\left( (g_v)_{v \in X^n} , \tau \right) \hskip.5cm \text{with} \hskip.2cm
g_v \in G \hskip.2cm \text{for all} \hskip.2cm v \in X^n
\hskip.2cm \text{and} \hskip.2cm \tau \in S_d^{(n)}
\end{equation*}
for a typical element of $G \wr^n S_d$.

\begin{defn!}
\label{SsGroup}
Let $G$ be a group and $d \ge 2$ an integer.
A \textbf{self-similar structure of degree $d$} on $G$
is a homomorphism
\begin{equation}
\label{GivenPhi}
\Phi \, : \, G \longrightarrow G \wr S_d  .
\end{equation}
A \textbf{self-similar group} is such a pair $(G,\Phi)$;
when $\Phi$ is clear from the context, 
we write also ``$G$ is a self-similar group''.
\end{defn!}

If $(G,\Phi)$ be a self-similar group, 
the construction (\ref{WreathHomo}) gives rise to a sequence of homomorphisms
\begin{equation}
\label{EqPhin}
\Phi_n \, : \, G \overset{\Phi_{n-1}}{\longrightarrow} 
G \wr^{n-1} S_d \overset{ \Phi \wr 1_{d^{n-1}} }{\longrightarrow} 
(G \wr S_d) \wr^{n-1}  S_d = G \wr^n S_d
\end{equation}
for $n \ge 2$; we write $\Phi_0 = \operatorname{id}_G$ and $\Phi_1 = \Phi$.
Note that, if $\Phi$ is injective, so is $\Phi_n$ 
for all $n \ge 0$.
It is routine to check that $\Phi_{m+n}$ is the composition
\begin{equation}
\label{m+n=mn} 
\Phi_{m+n}
\, : \,
G 
\overset{ \Phi_n }{\longrightarrow} 
G \wr^n S_d 
\overset{ \Phi_m \wr 1_{d^n} }{\longrightarrow}
\left( G \wr^m S_d \right) \wr^n S_d = G \wr^{m+n} S_d
\end{equation}
for all $m,n \ge 0$.
\par

The composition of $\Phi_n$ and the quotient map 
$G \wr^n S_d \longrightarrow S_d^{(n)}$
is a homomorphism
\begin{equation}
\label{DefOfTau}
G \longrightarrow S_d^{(n)}, \hskip.2cm g \longmapsto \tau_g^{(n)} .
\end{equation}
Thus, introducing the $v$-coordinates of $\Phi_n(\cdot )$, we have
\begin{equation*}
\Phi_n (g) \, = \, \left( (g_v)_{v \in X^n}, \tau_g^{(n)} \right) \, \in \,
G \wr_{X^n} S_d^{(n)} = G \wr^n S_d 
\end{equation*}
for all $g \in G$.
Note that
\begin{equation}
\label{RecTauN}
\tau_g^{(n)} \, = \, \big( \big( \tau_{g_x}^{(n-1)} \big)_{x \in X} , \tau_g^{(1)} \big)
\, \in \, S_d^{(n)}
\end{equation}
for all $g \in G$ and $n \ge 1$.
\par

Let $X^* = \bigsqcup_{n \ge 0} X^n$ 
be the free monoid over $X = \{0, 1, \hdots, d-1 \}$. 
The \textbf{length} of $v \in X^*$ is the integer $n = \vert v \vert$
such that $v \in X^n$.
The \textbf{$d$-regular rooted tree} is the tree with vertex set $X^*$, 
and with edges connecting pairs of vertices of the form
$(x_1\cdots x_n , x_1\cdots x_n x_{n+1})$, with $n \ge 0$ 
and $x_1, \hdots, x_{n+1} \in X$; 
abusively, we denote this tree also by $X^*$.
The homomorphisms of (\ref{DefOfTau}) define an action 
\emph{from the right} of $G$ on the tree $X^*$.
\par

For $n \ge 1$ and $v \in X^n$, we define the \textbf{stabilizer of $v$}
to be the subgroup
\begin{equation}
\label{DefOfStab}
\operatorname{Stab}_G (v) \, = \, 
\left\{ g \in G \mid  v \tau_g^{(n)} = v \right\} ,
\end{equation}
and we have a homomorphism
\begin{equation}
\label{DefOfPhiw}
\Phi_v \, : \, \operatorname{Stab}_G (v)   
\longrightarrow G , \hskip.2cm g \longmapsto g_v
\end{equation}
where $g_v = \Phi_v(g)$ is the $v$-coordinate of $\Phi_n(g)$.
%
%

\begin{lem}
\label{SectionsOfProducts}
With the notation above,
\begin{equation*}
g_{vw} \, = \ (g_v)_w, \hskip.2cm 
(gh)_v \, =  \, g_v h_{v \tau_g^{(n)} } , \hskip.2cm \text{and} \hskip.2cm
(h^{-1})_v \, = \,  \left( h_{v \tau_{h^{-1}}^{(n)}} \right)^{-1} ,
\end{equation*}
for all $g,h \in G$, $n \ge 1$, $v \in X^n$,  and $w \in X^*$.
\end{lem}

\begin{proof}
To illustrate the fact that $G$ acts on $X^*$ \emph{from the right}, 
we spell out the proof of the second equality,
writing $vg$ for $v \tau^{(n)}_g$.
We have on the one hand
\begin{equation*}
\big( (vw) g \big) h \ = \, \big( (vg) (w g_v)\big) h \, = \, (v gh) (w g_v h_{vg})
\end{equation*}
 and on the other hand
\begin{equation*}
(vw) (gh) \ = \,  (v gh) \big(w (gh)_v) \big) .
\end{equation*}
Hence $(gh)_v = g_v h_{vg}$.
\end{proof}

\begin{defn!}
\label{faithful}
A self-similar group $(G,\Phi)$
is \textbf{faithful} if its action on the tree $X^*$ described above is faithful;
this implies that the homomorphism $\Phi$ is injective 
(but the converse does not hold\footnote{Set
$H = \Phi^{-1}(G^X)$. The kernel of the action of $G$ on $X^*$
is the largest normal subgroup $N$ of $G$ that is contained in $H$
and such that $\Phi(N) \subset N^X$.}).
\par

A self-similar group $(G,\Phi)$
is \textbf{contracting} if there is a finite subset $\mathcal M \subset G$
with the following property:
\begin{itemize}
\item[]
for all $g \in G$, there exists an integer $k \ge 0$ 
\item[]
such that
$g_v \in \mathcal M$ for all $v \in X^*$ with $\vert v \vert \ge k$.
\end{itemize}
The smallest such $\mathcal M$, namely
\begin{equation*}
\mathcal N \, := \, \bigcup_{g \in G} \bigcap_{k \ge 0}
\{ g \in G \mid \exists \hskip.1cm h \in G, \hskip.1cm \ell \ge k, \hskip.1cm  v \in X^\ell \hskip.2cm \text{with} \hskip.2cm h_v = g \}
\end{equation*} 
is called the \textbf{nucleus} of $(G,\Phi)$.
\par

A self-similar group $(G,\Phi)$
is \textbf{self-replicating}\footnote{Or
\textbf{recurrent}, or \textbf{fractal}, 
as in \cite[Definition 2.8.1]{Nekr--05}. See \cite{BaGN--03} for relations of such groups
with fractal sets.} 
if, for all $g \in G$ and $x \in X$, 
there exists $h \in \operatorname{Stab}_G (x)$ such that $h_x = g$,
namely if, for all $x \in X$,
the homomorphism $\Phi_x$ of (\ref{DefOfPhiw}) is onto.
When this is so, it is easy to check by induction on the level that,
for all $g \in G$, $n \ge 1$, and $v \in X^n$,
there exists $h \in \operatorname{Stab}_G (v)$ 
such that $\Phi_v(h) = g$, 
namely $\Phi_v$ is onto.
\end{defn!}

Observe that, by definition, we have $1 \in \mathcal N$.
Moreover for $g \in G$, we have $g \in \mathcal N$
if and only if $g^{-1} \in \mathcal N$, by Lemma \ref{SectionsOfProducts}.

The following proposition records basic facts 
about the nucleus of a contracting group.

\begin{prop}
\label{PropOnN}
Let $(G, \Phi)$ be a contracting self-similar group with nucleus $\mathcal N$, as above.
\par
(i)
For $g \in \mathcal N$ and $x \in X$, we have $g_x \in \mathcal N$.
\par
(ii)
If $(G, \Phi)$ is self-replicating and $G$ is finitely generated,
then $\mathcal N$ generates $G$.
\end{prop}

\begin{proof}
For $g \in \mathcal N$, there exist $h \in G$, $k \ge 0$,  and $v \in X^k$ 
such that $h_v = g$ and $h_w \in \mathcal N$ for all $w \in X^*$ with $\vert w \vert \ge k$
(otherwise, $\mathcal N$ would not be minimal).
Hence $g_x = (h_v)_x =  h_{vx}  \in \mathcal N$ for all $x \in X$.
This proves (i).
\par

For (ii), we paraphrase \cite[Lemma 2.11.3]{Nekr--05}.
Denote by $\langle \mathcal N \rangle$ the subgroup of $G$
generated by $\mathcal N$.
Let $S$ be a symmetric finite generating set of $G$.
For all $s \in S$, there exists $k_s \ge 0$ such that $s_v \in \mathcal N$
for all $v \in X^*$ with $\vert v \vert \ge k_s$.
Set $k = \max \{k_s \mid s \in S \}$.
\par

Let $g \in G$ and $v \in X^*$ with $\vert v \vert \ge k$.
There exist $s_1 \hdots, s_m \in S$ with $g = s_1 \cdots s_m$,
so that
\begin{equation*}
\aligned
g_v \ &= \, 
(s_1)_v (s_2 \cdots s_m)_{vs_1} \, = \, \cdots 
\\
\, &= \,
(s_1)_v (s_2)_{vs_1} (s_3)_{vs_1s_2} \cdots (s_m)_{vs_1 \cdots s_{m-1}} \, \in \,
\langle N \rangle ,
\endaligned
\end{equation*}
where the last inclusion follows from
$\vert v \vert = \vert vs_1 \vert = \cdots \vert vs_1 \cdots s_{m-1} \vert \ge k$.
In particular,  the image of $\Phi_v$,
as defined in (\ref{DefOfPhiw}), 
is contained in $\langle \mathcal N \rangle$.
\par

If $(G, \Phi)$ is self replicating, then $\Phi_v$ is onto for all $v \in X^*$
with $\vert v \vert \ge 1$. 
The conclusion follows.
\end{proof}

The next proposition was inspired to us by \cite[Lemma 2.13.4]{Nekr--05}.
We need some notation and a definition; 
our exposition borrows from \cite{Bart}.

\begin{defn!}
\label{DefG0}
Let $(G,\Phi)$ be a self-replicating contracting self-similar group, 
with nucleus $\mathcal N = \{n_1, \hdots, n_\ell\}$.
Let $S = \{s_1, \hdots, s_\ell \}$ be a finite set given with a bijection
$s_j \leftrightarrow n_j$ with $\mathcal N$.
Let $R$ be the set of relators in the letters of $S$ of one the forms
\begin{equation*}
\begin{array}{cccc}
s_i &= 1 \hskip.5cm &\text{if} \hskip.2cm n_i &= 1 \in G,
\\
s_i s_j &= 1 \hskip.5cm &\text{if} \hskip.2cm n_i n_j &= 1 \in G,
\\
s_i s_j s_k &= 1 \hskip.5cm &\text{if} \hskip.2cm n_i n_j n_k &= 1 \in G.
\end{array}
\end{equation*}
Note that these relators are of length at most $3$;
they are indexed by a subset of $\mathcal N \sqcup \mathcal N^2 \sqcup \mathcal N^3$.
\par

Assume furthermore that $G$ is finitely generated.
The   \textbf{universal contracting cover} of $G$ 
is the finitely presented group $G_0^\un$ defined by the presentation
with $S$ as set of generators and $R$ as set of relators.
The assignment $\pi^\un (s_i) = n_i$ extends to 
a group homomorphism 
\begin{equation}
\label{DefPi}
\pi^\un \, : \, G_0^\un = \langle S \mid R \rangle \longrightarrow G ,
\end{equation}
because $\pi^\un (r) = 1$ for any $r \in R$.

Note that $\pi^\un$ is onto, by  Proposition \ref{PropOnN}.
We define finally
\begin{equation}
\label{DefPiHat}
\widehat \pi^\un = \pi^\un \wr 1_d \, : \, G_0^\un \wr S_d \longrightarrow G \wr S_d .
\end{equation}
\end{defn!}

\begin{rem!}
\label{InExamplesDeleteFromS}
In particular examples, and for simplicity, it is often convenient to delete from $S$
the generator corresponding to $1 \in \mathcal N$,
to delete  $s_k$ if 
there exist $i,j \in \{1, \hdots, \ell\}$ with $n_k = n_i n_j$,
and to delete one generator of every pair corresponding to $\{n, n^{-1} \} \subset \mathcal N$.
For example, in Example \ref{ExBasilica},
we have $\mathcal N = \{1, a^{\pm 1}, b^{\pm 1}, c^{\pm 1} \}$ with $7$ elements,
and $c = a^{-1}b$, 
but $S = \{a,b\}$ with $2$ elements.
\par
Note however that, in Example \ref{Ex1stGrig}, 
we keep $d$ in the generating set $\{a,b,c,d\}$ of $G_0$, even though $d=bc$.
\end{rem!}

\begin{prop}
\label{CoveringSSGroup}
Let $(G,\Phi)$ be a self-replicating contracting self-similar group of degree $d$,
with nucleus $\mathcal N$.
Assume that $G$ is finitely generated.
Let $G_0^\un = \langle S \mid R \rangle$ and $\pi^\un : G_0^\un \twoheadrightarrow G$ 
be the universal contracting cover and the projection of Definition \ref{DefG0}.
\par

Then there exists a homomorphism
\begin{equation*}
\varphi_1^\un \, : \, G_0^\un \longrightarrow G_0^\un \wr S_d
\end{equation*}
such that the self-similar group $(G_0^\un, \varphi_1^\un)$ is 
contracting, with nucleus $S$.
Moreover the diagram 
\begin{equation}
\label{DiagramInf}
\begin{array}{ccc}
 G_0^\un         & \overset{\varphi_1^\un}{\longrightarrow}  & G_0^\un \wr S_d 
 \\
 &&
 \\
 \pi^\un \downarrow  &                        & \downarrow \widehat \pi^\un
 \\
 &&
 \\
 G           & \overset{\Phi}{\longrightarrow}     & G\wr S_d
\end{array}
\end{equation}
commutes.
\end{prop}

\begin{proof}
\emph{Step 1,  definition of $\varphi_1^\un$.}
Denote by $\ell$ the order of $\mathcal N$, 
and write $\mathcal N = \{n_1, \hdots, n_\ell\}$, as above.
Let $i \in \{1, \hdots, \ell\}$. 
By Proposition \ref{PropOnN}, 
there exist $i_0, \hdots, i_{d-1} \in \{1, \hdots, \ell\}$ and $\tau_i \in S_d$ such that
\begin{equation*}
\Phi(n_i) \, = \, (n_{i_0}, \hdots, n_{i_{d-1}} , \tau_i ) .
\end{equation*}
We set 
\begin{equation*}
\varphi_1^\un(s_i) \, = \, (s_{i_0}, \hdots, s_{i_{d-1}} , \tau_i ) \, \in \, G_0^\un \wr S_d ,
\end{equation*}
and we claim that this extends to 
a group homomorphism $\varphi_1^\un$ as in (\ref{DiagramInf}).
\par

Consider a relator as in Definition \ref{DefG0}, 
say $s_is_js_k = 1$ (shorter relators are dealt with similarly);
hence $n_i n_j n_k = 1 \in G$.
Choose $x \in X$; recall that $X$ stands for $\{0, \hdots, d-1\}$.
There exist $p,q,r \in \{1, \hdots, \ell\}$ 
and $\tau_p, \tau_q, \tau_r \in S_d$
such that the $x$-coordinate and the last coordinate of 
$\Phi( n_i n_j n_k)$
can be written as
\begin{equation*}
\left( n_i n_j n_k  \right)_x \, = \, n_p n_q n_r 
\hskip.2cm \text{and} \hskip.2cm 
\tau_{ n_i n_j n_k } \, = \, \tau_{p}  \tau_{q} \tau_{r} .
\end{equation*}
Since $n_i n_j n_k = 1 \in G$, we have 
\begin{equation*}
n_p n_q n_r \, = \, 1 \in G \hskip.2cm \forall x \in X
\hskip.2cm \text{and} \hskip.2cm 
\tau_{p}  \tau_{q}  \tau_{r} \, = \,  1 \in S_d .
\end{equation*}
Hence
$\varphi_1^\un(s_i) \varphi_1^\un(s_j) \varphi_1^\un(s_k) = 1 \in G_0^\un$.
The claim is proven.

\vskip.2cm

\emph{Step 2: $(G_0^\un, \varphi_1^\un)$ is a contracting group with nucleus $S$.}
For any word $w$ in the letters of $S$, we have to show that there exists
a vertex $v \in X^*$ such that $(w)_v \in S$.
By induction on the word length, and by Lemma \ref{SectionsOfProducts},
it is enough to show this for a word of length $2$.
\par

Let $s_i, s_j \in S$ and $v \in X^*$ be such that $(n_i n_j)_v \in \mathcal N$,
say $(n_i n_j)_v = n_k$.
We have 
\begin{equation*}
(n_i)_v (n_j)_{v \tau_{n_i}^{ ( \vert v \vert ) } } \, =  \, n_k 
\hskip.5cm \text{in} \hskip.2cm G ,
\end{equation*}
which is a relator of length at most $3$.
Hence the corresponding relator $(s_is_j)_v = s_k$ holds in $S$.
\par

It follows that $S$ is the nucleus of the group $(G, \varphi_1^\un)$.

\vskip.2cm

\emph{Step 3, commutativity of the diagram.}
This can be checked on the set $S$ of generators of $G_0^\un$.
\end{proof}

The universal contracting cover $(G_0^\un, \varphi_1^\un)$ of $(G, \Phi)$
is uniquely defined by $(G,\Phi)$, and contracting.
But we believe it need not be self-replicating
(even though we do not know of any specific example).
In all cases, $(G_0^\un, \varphi_1^\un)$ has quotients by finite sets of relations
that are self-replicating contracting covers of $(G,\Phi)$,
as described in the Definition \ref{DefG0Tilde} and Proposition \ref{CoveringSSGroupTilde}.
Note however that these quotients are no more uniquely defined by $(G, \Phi)$,
since choices are involved.
\par

In each of Examples \ref{Ex1stGrig} to \ref{Hanoi} below,
the universal contracting cover is self-replicating.

\begin{defn!}
 \label{DefG0Tilde} 
Let $(G,\Phi)$ be a self-replicating contracting self-similar group, 
with nucleus $\mathcal N = \{n_1, \hdots, n_\ell\}$;
assume that $G$ is finitely generated.
Let $S = \{s_1, \hdots, s_\ell \}$ be in bijection with $\mathcal N$,
and $(G_0^\un, \varphi_1^\un)$ the universal contracting cover of $(G,\Phi)$, 
as in Definition \ref{DefG0}. Let $\pi^\un : G_0^\un \longrightarrow G$ be as in (\ref{DefPi}).
\par

Let $x \in X$ and  $n_i \in \mathcal N$.
Since the pair $(G,\Phi)$ is self-replicating, 
there exists\footnote{Note that a choice is involved here.} 
$g(x,n_i) \in \operatorname{Stab}_G(x)$
such that $(g(x,n_i))_x = n_i$.
Since $\pi^\un$ is onto, 
there exists $h(x,n_i) \in G_0^\un$
such that $\pi^\un (h(x,n_i)) = g(x,n_i)$;
moreover, since $\widehat \pi^\un$ is the identity 
on the permutations of the wreath product,
we have $h(x,n_i) \in \operatorname{Stab}_{G_0^\un}(x)$.
By commutativity of  Diagram (\ref{DiagramInf}), we have $\pi^\un ((h(x,n_i)) _x)=n_i$.
Set $w(x,n_i) = (h(x,n_i))_x s_i^{-1}$;  then
$w(x,n_i)$ belongs to the kernel of $\pi^\un$.
\par

Again, by commutativity of (\ref{DiagramInf}), 
we have $(w(x,n_i))_v \in \ker (\pi^\un)$ for all $v \in X^*$.
Since $(G_0^\un, \varphi_1^\un)$ is  contracting, the subset 
\begin{equation*}
E(x,n_i) \, = \,  \{ g \in G_0^\un \mid g = (w(x,n_i))_v \hskip.2cm \text{for some} \hskip.2cm v \in X^* \}
\end{equation*}
of $G_0^\un$ is finite.
Define
\begin{equation*}
E \, = \, \bigcup_{x \in X,  n \in \mathcal N} E(x,n) 
\hskip.5cm \text{and} \hskip.5cm H = \langle\langle E \rangle\rangle \subset G_0^\un ,
\end{equation*}
where $\langle\langle E \rangle\rangle$ denote the \emph{normal} subgroup of $G_0^\un$
generated by $E$.
\par

A \textbf{standard contracting cover} of $G$ is a quotient group of the form 
$G_0 = G_0^\un /H$, with $H$ as above;
the image of $S$ in $G$ is a generating set, that we denote again (abusively) by $S$.
Note that $E$ is a finite subset of $G_0^\un$, and consequently that 
$G_0$ is a finitely presented group.
\par

The epimorphism $\pi^\un$ factors through a homomorphism $\pi : G_0 \longrightarrow G$,
because $E$ is a subset of $\ker \pi^\un$.
It follows from the definition that the restriction of $\pi$ to the generating set $S$ of $G_0$ is injective.
\end{defn!}

The following proposition is 
the analogue of proposition \ref{CoveringSSGroup} for $G_0$.

\begin{prop}
\label{CoveringSSGroupTilde}
Let $(G,\Phi)$ be a self-replicating contracting self-similar group of degree $d$,
with nucleus $\mathcal N$.
Assume that $G$ is finitely generated.
Let $G_0$ and $\pi : G_0 \twoheadrightarrow G$ 
be a standard contracting cover of $(G,\Phi)$ and its projection to $G$,  
as in Definition \ref{DefG0Tilde}. 
\par
Then there exists a homomorphism
\begin{equation*}
\varphi_1 \, : \, G_0 \longrightarrow G_0 \wr S_d
\end{equation*}
such that the self-similar group $(G_0, \varphi_1)$ is 
contracting and self-replicating with nucleus $S$.
Moreover the diagram 
\begin{equation}
\label{DiagramInfTilde}
\begin{array}{ccc}
 G_0         & \overset{\varphi_1}{\longrightarrow}  & G_0\wr S_d 
 \\
 &&
 \\
 \pi\downarrow  &                        & \downarrow \widehat \pi
 \\
 &&
 \\
 G           & \overset{\Phi}{\longrightarrow}     & G\wr S_d
\end{array}
\end{equation}
commutes.
\end{prop}

\begin{proof}
By construction of the set $E$, for any element $g\in E$ and any $x\in X$, 
we have $g_x \in E$. Hence the homomorphism 
$ \varphi_1^\un \, : \,  G_0^\un \longrightarrow  G_0^\un \wr S_d$
induces a homomorphism $\varphi_1 : G_0 \longrightarrow G_0 \wr S_d$.
Since  $(G_0^\un ,\varphi_1^\un)$ is contracting with nucleus $S$,
the self-similar group $(G_0, \varphi_1)$ is contracting with nucleus $S$.
\par

Let $x \in X$, and $n_i \in \mathcal N$.
We continue with the notation of Definition \ref{DefG0Tilde}.
By construction of $G_0$, 
the relation $h(x,n_i) = s_i$ holds in $G_0$;
moreover $h(x,n_i)$ is an element of $\operatorname{Stab}_{G_0}(x)$.  
This shows that the pair $(G_0, \varphi_1)$ is  self-replicating.
\par

The commutativity of diagram (\ref{DiagramInfTilde}) 
can be checked on the generators of  $G_0$.
\end{proof}

From here to Corollary \ref{GNotfp}, 
we keep the same notation as in Definition \ref{DefG0Tilde} and Proposition \ref{CoveringSSGroupTilde}
for $G_0$, $\pi$, and $\varphi_1$, in relation with a given
contracting self-replicating self-similar group $(G,\Phi)$,
with $G$ finitely generated.
 
\begin{defn!}
\label{DefGnEtc}
For an integer $n \ge 0$, define 
\begin{itemize}
\item[(i)]
the homomorphism
$\varphi_n : G_0 \longrightarrow G_0 \wr^n S_d$ as in (\ref{EqPhin}),
\item[(ii)]
its kernel $N_n = \ker (\varphi_n)$ and the quotient $G_n = G_0 / N_n$,
\item[(iii)]
the homomorphism
\begin{equation}
\label{DefPiHatN}
\widehat \pi_n = \pi \wr 1_{d^n} \, : \,
G_0 \wr^n S_d \longrightarrow G \wr^n S_d
\end{equation}
as in (\ref{WreathHomo}); 
note that $\widehat \pi_1$ is the $\widehat \pi$ of (\ref{DefPiHat}).
\end{itemize}
We have $\Phi_n \pi = \widehat \pi_n \varphi_n$,
i.e.\ the diagram
\begin{equation}
\label{DiagramN}
\begin{array}{ccc}
   G_0         & \xrightarrow{\varphi_n}  & G_0 \wr^n S_d 
   \\
   &&
   \\
 \pi\downarrow  &                        & \downarrow\widehat\pi_n
  \\
  &&
  \\
 G           & \xrightarrow{\Phi_n}     & G \wr^n S_d
\end{array}
\end{equation}
commutes.
Observe that 
$N_0 \subset \cdots \subset N_n  \subset N_{n+1} \subset \cdots$
and define
\begin{equation*}
N = \bigcup_{n=0}^{\infty} N_n .
\end{equation*}
\end{defn!}

\begin{rem!}
\label{RestrictionPiOnS}
As noted in Definition \ref{DefG0Tilde}, 
the restriction of $\pi$ to $S$ is injective.
More generally, in Definition \ref{DefGnEtc},
the restriction of $\widehat \pi_n$ to the subset
$(S^{X^n} , 1)$ of $G_0 \wr^n S_d = G_0^{X^n} \rtimes S_d^{(n)}$
is injective.
\end{rem!}

\begin{lem}
\label{N=KerPi}
Let $(G, \Phi)$ be a self-similar group;
assume that $G$ is finitely generated
and that $(G,\Phi)$ is faithful contracting self-replicating.
With the notation above, we have
\begin{equation*}
N \, = \, \ker \pi , \hskip.2cm \text{namely} \hskip.2cm G_0/N = G ,
\end{equation*}
so that 
\begin{equation*}
\lim_{n \to \infty} G_n \, = \, G 
\end{equation*}
in the space of marked groups on $\vert S \vert$ generators
(in the sense of Section \ref{sectionMark+Chab}).
\end{lem}

\begin{proof}
Let $g \in N$. Let $n \ge 1$ be such that $g \in \ker (\varphi_n)$.
Then $\Phi_n \pi (g) = \widehat \pi_n \varphi_n (g) = 1$, 
hence $g \in \ker (\pi)$ by the faithfulness assumption.
\par

Conversely, let $k \in \ker (\pi)$. 
On the one hand,
since $(G_0, \varphi)$ is contracting, there exists $n \ge 0$
such that $k_v \in S$ for all $v \in X^n$.
On the other hand,
$\pi(k) = 1$ implies
$ \pi(k_v) = 1$ for all $v \in X^n$;
moreover, the $S_d^{(n)}$-coordinate of $\varphi_n(k)$ is $1$,
by commutativity of Diagram (\ref{DiagramN}).
Hence, by Remark (\ref{RestrictionPiOnS}),
we have $k_v = 1$ for all $v \in X^n$,
and therefore $k \in N_n = \ker (\varphi_n)$, a fortiori $k \in N$.
\end{proof}

\begin{lem}
\label{backwardsalongG_n's}
In the situation of the previous lemma,  for all $n \ge 1$, we have
\begin{equation*}
N_n \, = \,  \varphi_1^{-1}(N_{n-1}^d)
\end{equation*}
so that $\varphi_1 : G_0 \longrightarrow G_0 \wr S_d$
induces a homomorphism
\begin{equation*}
\psi_n \, : \, \left\{
\aligned
G_n \, &\longrightarrow \, \hskip1cm G_{n-1} \wr S_d
\\
gN_n \, &\longmapsto \, 
\left( 
\big( \varphi_1(g)_x N_{n-1} \big)_{x \in X} , \tau_g^{(1)} \right) .
\endaligned
\right.
\end{equation*}
Moreover $\psi_n$ is injective.
\end{lem}

\begin{proof}
For $g \in G$, write
\begin{equation}
\label{ProofLemma1.11}
\varphi(g) \, = \, \big( (g_x)_{x \in X}, \tau_g^{(1)} \big)
\hskip.5cm \text{and} \hskip.5cm
\varphi_n(g) \, = \, \big( (g_v)_{v \in X^n}, \tau_g^{(n)} \big)  .
\end{equation}
\par

Assume first that $g \in N_n$.
Thus $(g_x)_{v'} = 1$ and $\tau_{g_x}^{(n-1)} = 1$
for all $x \in X$ and $v' \in X^{n-1}$.
This can be written
\begin{equation*}
\varphi_{n-1}(g_x) \, = \, 
\big( ((g_x)_{v'})_{v' \in X^{n-1}}, \tau_{g_x}^{(n-1)} \big) \, = \, 1 
\hskip.5cm \forall x \in X ,
\end{equation*}
namely $g_x \in N_{n-1} \hskip.2cm \forall x \in X$.
We have checked that $\varphi_1(N_n) \subset N_{n-1}^d$,
and $N_n \subset \varphi_1^{-1}(N_{n-1}^d)$ follows.
\par

Assume now that $g \in \varphi_1^{-1} (N_{n-1}^d)$,
namely that $(g_x)_{v'} = 1$ and $\tau_{g_x}^{(n-1)} = 1$
for all $x \in X$ and $v' \in X^{n-1}$.
This can be written $g_v = 1$ for all $v \in X^n$ and $\tau_g^{(n)} = 1$,
namely $g \in N_n$.
Hence $\varphi_1^{-1}(N_{n-1}^d) \subset N_n$.
\end{proof}

The next theorem is a detailed version of Theorem \ref{mainsection2}.

\begin{thm}
\label{FirstThmOnSSGps}
Let $(G, \Phi)$ be a self-similar group;
assume that $G$ is finitely generated
and that $(G, \Phi)$ is faithful contracting self-replicating.
Let $G_0$ be a standard contracting cover,
as in Definition \ref{DefG0Tilde}.
\par

Assume that $G_0$ contains non-abelian free subgroups.
Then, for each $n \ge 0$, the group $G_n$ of Definition \ref{DefGnEtc}
contains non-abelian free subgroups.
More generally, every finitely presented cover of $G$
contains non-abelian free subgroups.
\par

Assume moreover that $G_0$ is large.
Then every finitely presented cover of $G$ is large.
\end{thm}

\begin{proof}

Let $S^{(0)}_d$ be the subgroup of $S_d$ of permutations fixing the letter $x=0$.
For $n \ge 1$, let $H_n$ be the finite index  subgroup of $G_n$ defined by
\begin{equation*}
H_n \, = \,  \psi_n^{-1}(G_{n-1}^{ \{0, 1, \hdots, d-1 \} } \rtimes S^{(0)}_d) ,
\end{equation*}
where $\psi_n$ is as in lemma \ref{backwardsalongG_n's}.
Projection onto the first coordinate (i.e. the coordinate $x=0$)
 \begin{equation*}
p^{(0)}_n \, : G_{n}^{ \{0, 1, \hdots, d-1 \} } \rtimes S^{(0)}_d \longrightarrow  G_{n}
\end{equation*}
defined by 
\begin{equation*}
p^{(0)}_n((g_x N_{n} \big)_{x \in X} , \tau) \, = \, g_0N_{n}
\end{equation*} 
is a group homomorphism. It turns out that the composition
\begin{equation*}
q^{(0)}_n\,:H_n \overset{\psi_n}{\longrightarrow}\psi_n(H_n)\overset{p^{(0)}_{n-1}}{\longrightarrow}
G_{n-1} 
\end{equation*}
defines a group homomorphism from $H_n$ to $G_{n-1}$.

Given a generator $sN_{n-1}$ of $G_{n-1}$ (where $s$ is a generator of $G_0$), 
using the self-replicating property of $(G_0,\varphi_1)$
let $h\in St_{G_0}(0)$ be such that $\varphi_1(h)_0=s$. It turns out that $q^{(0)}(hN_n)=sN_{n-1}$
which shows that $q^{(0)}_n$ is onto $G_{n-1}$. The conclusion is that for each $n\geq 1$,
$G_n$ contains a finite index subgroup $H_n$ which maps onto $G_{n-1}$.

Therefore, if $G_0$ contains
non-abelian free subgroups (respectively is large), by induction on $n$, each $G_n$ will contain
non-abelian free subgroups (respectively will be large). Then by lemma \ref{N=KerPi}
and corollary \ref{main2} below,
every finitely presented cover of $G$ will contain non-abelian free subgroups 
(respectively will be large).
\end{proof}

%
%

\begin{cor}
\label{GNotfp}
Let $G$ be as in Theorem \ref{FirstThmOnSSGps}.
If $G_0$ contains non-abelian free subgroups, 
then $G$ is infinitely presented.
\end{cor}

\begin{proof}
Since $G_0$ does contain non-abelian free subgroups, by assumption,
and $G$ does not, by \cite[Theorem 4.2]{Nekr--10},
$G$ cannot be finitely presented, by the previous theorem.
\end{proof}

Consider integers $d, \ell \ge 1$, and a system of relations
\begin{equation}
\label{SsGenerators}
\left\{ \aligned
s_1 &\, = \, \left( (s_1)_0, \hdots, (s_1)_{d-1} \right) \tau_1
\\
\cdot\cdot & \cdots\cdots\cdots\cdots
\\
s_\ell &\, = \left( (s_\ell)_0, \hdots, (s_\ell)_{d-1} \right) \tau_\ell
\endaligned \right.
\end{equation}
with 
\begin{equation*}
\aligned
(s_j)_x \, \in \,  \{s_1, \hdots, s_\ell \} 
\hskip.2cm &\text{and} \hskip.2cm
\tau_j \, \in \,  S_d 
\\
\hskip.2cm \text{for} \hskip.2cm 
    j \in \{1, \hdots, \ell\}
\hskip.2cm &\text{and} \hskip.2cm    
    x  \in   X = \{0, \hdots, d-1 \} .
\endaligned
\end{equation*}
By induction on $n \ge 0$, the system (\ref{SsGenerators})
defines a set $S$ of automorphisms of the $d$-regular rooted tree $X^*$,  
again denoted by $s_1, \hdots, s_\ell$.
These generate a group $G = \langle S \rangle$
of automorphisms of the tree $X^*$,
and (\ref{SsGenerators}) define a self-similarity structure $\Phi$ on $G$;
thus $G$ is a self-similar group as in Definition \ref{SsGroup},
and moreover $G$ is faithful.
The system (\ref{SsGenerators}) is often denoted by $\Phi$ again.
\par

We review below some classical examples of self-similar groups defined this way.

\begin{exe!}
\label{Ex1stGrig}
The $4$-generated group
$\mathfrak G = \langle a,b,c,d \rangle$ of 
Theorem \ref{GrHa+i} and \cite{Grig--80, Grig--83} 
is a self-similar group of degree $2$, 
with $\Phi : \mathfrak G \longrightarrow \mathfrak G \wr C_2$ defined by
\begin{equation*}
\Phi(a) = (1,1)\tau, \hskip.2cm
\Phi(b) = (a,c), \hskip.2cm
\Phi(c) = (a,d), \hskip.2cm
\Phi(d) = (1,b) .
\end{equation*}
Here, $C_2 = \{1, \tau \}$ denotes the cyclic group of order $2$
(written $S_2$ in Definition \ref{SsGroup}).
The self-similar group $(\mathfrak G, \Phi)$ is 
faithful, contracting, and self-replicating.
The group $\mathfrak G$ is of intermediate growth.
\par

The nucleus is 
\begin{equation*}
\mathcal N \ = \, \{ 1, a, b, c, d \} .
\end{equation*}
The universal contracting cover of Definition \ref{DefG0}
has the presentation
\begin{equation*}
G_0 \, = \, \langle a,b,c,d \mid 
a^2, b^2, c^2, d^2, bcd \rangle \, \simeq \,
C_2 \ast V ,
\end{equation*}
where $C_2$ is now the group $\{1,a\}$ 
and $V$ the Klein Vierergruppe $\{1, b, c, d\}$,
isomorphic to $C_2 \times C_2$.
The sign $\simeq$ indicates an isomorphism of groups. 
It can easily be checked that this cover is self-replicating,
so that a cover as in Definition \ref{DefG0Tilde} is not needed here.
\par

Proofs of these facts, and of other properties of $\mathfrak G$,
can be found in \cite{Grig--80}, \cite[Chapter VIII]{Harp--00},
or \cite[Section 1.6]{Nekr--05}, 
to quote \emph{some} of the existing expositions only;
see also our Section \ref{sectionGrigGroup}.
The group $\mathfrak G$ can be viewed as the IMG 
of an orbifold version of the \emph{tent map} 
$T : \mathopen[0,1\mathclose] \longrightarrow \mathopen[0,1\mathclose]$, defined by
$T(x) = 2x$  for $x \le 1/2$
and $T(x) = 2-2x$ for $x \ge 1/2$
\cite[Section 5.3]{Nekr--11}.
Here and below, ``IMG'' stands for 
``\textbf{Iterated Monodromy Group}'' (see \cite{Nekr--05}).
\par

Observe that $C_2 \ast V$ is virtually a non-abelian free group,
because it is a free product of finite groups, distinct from $C_2 \ast C_2$
(see for example \cite[Proposition 4 in Number I.1.3, Page 16]{Serr--77}).
It contains a free subgroup $F_3$ of index $8$;
one easy way to check this involves computing 
virtual Euler-Poincar\'e characteristics, as in \cite[Section 1.8]{Serr--71}:
if 
\begin{equation*}
F_x \, = \,  \ker (p : C_2 \ast V \overset{\text{canonical}}{\longrightarrow} C_2 \times V) ,
\end{equation*}
where $F_x$ stands for the free group of rank $x$, then
\begin{equation*}
\chi (C_2 \ast V) \, = \, \frac{1}{2} + \frac{1}{4} - 1 \, = \,
\frac{1}{ [C_2 \ast V : \ker p] } \chi (F_x) \, = \, \frac{1}{8} (1-x) ,
\end{equation*}
and therefore $x=3$.
\par

Hence  Theorem \ref{GrHa+i}
is a particular case of Theorem \ref{FirstThmOnSSGps}.
\end{exe!}

\begin{exe!}
\label{ExBasilica}
The $2$-generated \textbf{Basilica group} $\mathfrak B = \langle a,b \rangle$
is a self-similar group of degree $2$, 
with homomorphism $\Phi : \mathfrak B \longrightarrow \mathfrak B \wr C_2$
defined by
\begin{equation*}
\Phi(a) = (b,1)\tau, \hskip.2cm
\Phi(b) = (a,1). 
\end{equation*}
The self-similar group $(\mathfrak B, \Phi)$ is 
faithful, contracting, and self-replicating.
The group $\mathfrak B$ is of exponential growth.
\par

The nucleus is 
\begin{equation*}
\mathcal N \ = \, \{ 1, a^{\pm 1},  b^{\pm 1}, c^{\pm 1} \}, 
\hskip.2cm \text{where} \hskip.2cm
c = a^{-1} b .
\end{equation*}
The universal contracting cover 
has the presentation
\begin{equation*}
G_0 \, = \, \langle a, b \mid  \emptyset \rangle \, \simeq \, F_2 .
\end{equation*}
It is self-replicating.
\par

The group $\mathfrak B$ has been introduced in \cite{GrZ--02a, GrZ--02b}.
The name ``Basilica'' was given by Mandelbrot to the Julia set
of the quadratic transformation $z \longmapsto z^2-1$ of the complex plane,
in honour of the \emph{Basilica Cattedrale Patriarcale di San Marco,}
and its reflection in Venetian waters \cite[Page 254]{Mand--80}.
The group $\mathfrak B$ was identified as $\operatorname{IMG}(z^2-1)$ 
in \cite[Theorem 5.8]{BaGN--03}\footnote{As acknowledged in \cite{BaGN--03},
part of the credit for this is due to Richard Pink.},
and the group was named ``Basilica'' in
\cite{BaVi--05},  \cite{Kaim--05}, and \cite{Nekr--05}.
\par

Our notation for $\Phi(a)$ and $\Phi(b)$ 
is essentially that of \cite[Page 208]{Nekr--05};
the roles of $a$ and $b$ are exchanged in \cite{GrZ--02a}.
\par

Again, Theorem \ref{Erschler}
is a particular case of Theorem \ref{FirstThmOnSSGps}.
\par
Incidentally, since $\mathfrak B$ is amenable 
(references in Appendix \ref{AppendixDGrowthAmenability}),
Theorem \ref{Erschler} shows
that $\mathfrak B$ is \emph{not} finitely presented.
Since we could not find references for a direct proof the latter statement in the literature,
let us allude to two other simple ways to show that $\mathfrak B$ is not finitely presented.
One, suggested by Julia Bartsch (private communication),
uses the infinite presentation of $\mathfrak B$ given in \cite{GrZ--02b}
and obtained together with Laurent Bartholdi;
then a nice argument concludes that this presentation is minimal
(erasing any of its relators would change the group).
The other uses the contracting property of the Basilica group established in \cite{GrZ--02a}
and follows the idea indicated in \cite{Grig--80} for $\mathfrak G$.
\end{exe!}

\begin{exe!}
\label{IMGz^2+i}
The \textbf{IMG of $z^2+i$} 
\begin{equation*}
\mathfrak J \, = \, \operatorname{IMG}(z^2+i) \, = \, \langle a,b,c \rangle
\end{equation*}
is defined by
\begin{equation*}
\Phi(a) = (1,1)\tau, \hskip.2cm
\Phi(b) = (a,c) , \hskip.2cm
\Phi(c) = (b,1) . 
\end{equation*}
It was studied in detail in \cite{GrSS--07},
and was shown to be of intermediate growth in \cite{BuPe--06}.
\par

Its nucleus is
\begin{equation*}
\mathcal N \ = \, \{ 1, a, b, c \} ,
\end{equation*}
and the only non-trivial relators of length $\le 3$ among elements of $\mathcal N$
are $a^2 = b^2 = c^2 = 1$ (this can best be checked with the GAP 
package
(\verb+http://www.gap-system.org/Packages/automgrp.html+).
Thus the universal contracting cover
\begin{equation*}
G_0 \, = \, \langle a,b,c \mid a^2, b^2, c^2 \rangle
\, \simeq \, C_2 \ast C_2 \ast C_2
\end{equation*}
has a free subgroup of finite index
(indeed a subgroup $F_{5}$ of index $8$).
It is self-replicating.
\end{exe!}

\begin{exe!}
\label{GuptaSidki}
The \textbf{Gupta-Sidki group} 
$\mathfrak G \mathfrak S = \langle a,b \rangle$ 
is the $2$-gene\-rated group of automorphisms of the ternary rooted tree defined by
\begin{equation*}
\Phi(a) \ = \, (1,1,1)\tau , \hskip.2cm \Phi(b) \, = \, (a, a^{-1},b) ,
\end{equation*}
where $\tau \in S_3$ is the cyclic permutation $(0,1,2)$.
It is the infinite $3$-group introduced in \cite{GuSi--83};
it is just infinite \cite[Proposition 8.3]{BaGr--02};
it can be viewed as an IMG \cite[Section 4.5]{Nekr--11}.
\par

Its nucleus is
\begin{equation*}
\mathcal N \ = \, \{ 1, a, a^{-1}, b, b^{-1} \} .
\end{equation*}
The universal contracting cover is
\begin{equation*}
G_0 \, = \, \langle a,b \mid a^3, b^3 \rangle
\, \simeq \, C_3 \ast C_3 ,
\end{equation*}
and contains a free subgroup $F_4$ of index $9$.
It is self-replicating.
\par
The growth type of $\mathfrak G \mathfrak S$ is not known.

\end{exe!}

\begin{exe!}
\label{FabrykowskiGupta}
The \textbf{Fabrykowski-Gupta group} $\mathfrak F \mathfrak G  = \langle a,b \rangle$ 
is the $2$-generated group of automorphisms of the ternary rooted tree defined by
\begin{equation*}
\Phi(a) \ = \, (1,1,1)\tau , \hskip.2cm \Phi(b) \, = \, (a, 1, b) ,
\end{equation*}
with $\tau$ as in Example \ref{GuptaSidki}
\cite{FaGu--85, FaGu--91}.
It is of intermediate growth 
(see the original papers, 
and an exposition with improved estimates of growth in \cite{BaPo--09}),
it is just infinite \cite[Proposition 6.2]{BaGr--02},
and it is the IMG of the cubic polynomial
$z^3(-\frac{3}{2} + i \frac{ \sqrt 3 }{2}) + 1$ \cite[Section 5.4]{Nekr--11}.
\par

As in the previous example, the nucleus is
\begin{equation*}
\mathcal N \ = \, \{ 1, a, a^{-1}, b, b^{-1} \} .
\end{equation*}
The universal contracting cover is
\begin{equation*}
G_0 \, = \, \langle a,b \mid a^3, b^3 \rangle
\, \simeq \, C_3 \ast C_3 .
\end{equation*}
It is self-replicating.
\end{exe!}     

\begin{exe!}
\label{Hanoi}
The ternary \textbf{Hanoi Towers group} $\mathfrak H = \langle a,b,c \rangle$
is the $3$-generated group of automorphisms of the ternary rooted tree defined by
\begin{equation*}
\Phi (a) \, = \, (a,1,1) \tau_{1,2} , \hskip.2cm
\Phi(b) \, = \, (1,b,1) \tau_{0,2} , \hskip.2cm
\Phi(c) \, = \, (1,1,c) \tau_{0,1}
\end{equation*}
where $\tau_{1,2}$ is the transposition of $S_3$
exchanging $1$ and $2$, and similarly for $\tau_{0,2}, \tau_{0,1}$.
It was introduced in \cite{GrSu--06} 
as a model for the well-known Hanoi Towers problem;
it is known to be of exponential growth 
(\cite[Subsection 6.1]{GrSu--07} and \cite{Grig--06}),
and isomorphic to $\operatorname{IMG}(z^2 - \frac{16}{27 z})$
\cite[Example 8]{GrSu--07}.
\par

The nucleus is
\begin{equation*}
\mathcal N \, = \, \{1,a,b,c\} .
\end{equation*}
The universal contracting cover is
\begin{equation*}
G_0 \, = \, \langle a,b,c \mid a^2, b^2, c^2 \rangle
\, \simeq \, C_2 \ast C_2 \ast C_2 .
\end{equation*}
It is self-replicating.
\end{exe!}

\begin{thm} 
\label{thmetc}
Any finitely presented cover of one of the groups
$\mathfrak J$, $\mathfrak G \mathfrak S$, $\mathfrak F \mathfrak G$,
$\mathfrak H$ of the four previous examples is large.
\end{thm}

This is a straightforward consequence of Theorem \ref{FirstThmOnSSGps}.
In Section \ref{sectionOmega},
we will show how to modify \ref{FirstThmOnSSGps}
to cover uncountably many examples.

\begin{rem!}
\label{fewquotients}
Groups of interest here are often known to have rather few quotients,
of special kinds. Let us illustrate this as follows.

\vskip.2cm

(i) 
A group is \textbf{just infinite} if all its proper quotients are finite.
The group $\mathfrak G$ is just infinite.
More generally, with the notation of section \ref{sectionOmega}, 
the group $G_\omega$ is just infinite for all $\omega \in \Omega_0$
(as we repeat below in Proposition \ref{Grig84}.ii).

\vskip.2cm

(ii) 
Without recalling here the technical definitions, 
let us mention the following property of a finitely generated group $G$
assumed to be branch, or even weakly branch:
for any normal subgroup $N \ne \{1\}$ of $G$, 
there exists an integer $n \ge 1$ such that $N$ contains
the derived group of the rigid stabilizer $\operatorname{Rist}_G(n)$;
this follows from the proof of \cite[Theorem 4]{Grig--00}.
\par
As a consequence, if $G$ is branch, then any proper quotient of $G$
is virtually abelian.
\par
In particular, any proper quotient of one of
the groups $\mathfrak J$, 
$\mathfrak G \mathfrak S$,
$\mathfrak F \mathfrak G$, and $\mathfrak H$,
is virtually abelian.
\par

(This does \emph{not} apply to $\mathfrak B$,
which is weakly branch but \emph{not} branch group.
This applies to $\mathfrak G$, but it is of little interest in this case
since the property of (i) is strictly stronger.)

\vskip.2cm

(iii)
It is shown in \cite{Grig--98} that $\mathfrak G$ has 
a finitely presented HNN-extension $\widehat{\mathfrak G}$
which is in
$\mathcal S \mathcal G \smallsetminus \mathcal E \mathcal G$.
Any proper quotient of $\widehat{\mathfrak G}$ 
is metabelian and virtually abelian \cite[Theorem 2.3]{SaWi--02}.

\vskip.2cm

(iv)
The Basilica group $\mathfrak B$ is \textbf{just non-soluble},
which means that all its proper quotients are soluble
\cite[Proposition 6]{GrZ--02a}.

\vskip.2cm

(v)
Recall however that there exist groups of intermediate growth 
with uncountably many quotients: see \cite{Gri--84b}
and Definition \ref{LambdaUniversal}.
\end{rem!}

\section{\textbf{Marked groups and the Chabauty topology}}
\label{sectionMark+Chab}
For $k$ a positive integer, let $F_k$ denote the free group of rank $k$,
given together with an ordered free basis $(s_1, \hdots, s_k)$ of generators.
A \textbf{marked group of rank $k$} is a pair $(G,S)$
where $G$ is a group and $S$ an ordered set of $k$ generators
(for distinct $s,t \in S$,  equalities $s=1$ and $s=t \in G$ are allowed).
To such a pair corresponds a free cover
$\pi_G : F_k \twoheadrightarrow G$,
with $\pi_G(s_j)$ being the $j$th generator of $S$ ($1 \le j \le k$).
We denote by $\mathcal M_k$ the set of marked groups on $k$ generators,
identified here with the set of normal subgroups of $F_k$
via the bijection $(G,S) \longleftrightarrow \ker \pi_G$.
\par

The idea to furnish a space of (sub)groups with a topology
goes back at least to Chabauty \cite{Chab--50}, 
and has been revisited on many occasions, among others by
Bourbaki \cite[chapitre VIII, $\S$~5]{Bour--63},
Gromov \cite[final remarks]{Grom--81},
one of us \cite{Gri--84a}, 
Stepin \cite{Step--96}, 
Champetier \cite{Cham--00},
Champetier-Guirardel \cite{ChGu--05},
and Ceccherini-Silberstein $\&$ Coornaert \cite{CeCo--10}.
\par

The \textbf{Chabauty topology} on $\mathcal M_k$,
also called the \textbf{Cayley topology}, 
is that defined by the basis\footnote{There
is an equivalent definition in terms of the subbasis
\begin{equation*}
\mathcal O_{K,V} \, = \, 
\left\{ N \triangleleft F_k \hskip.2cm : \hskip.2cm 
N \cap K = \emptyset \hskip.2cm \text{and} \hskip.2cm N \cap V \ne \emptyset \right\} ,
\end{equation*}
indexed by pairs $(K,V)$ where $K$ is a finite subset of $F_k$
and $V$ a subset of $F_k$.
With $K$ compact and $V$ open,
it has the advantage to carry over to the space of closed subgroups
of a locally compact group $G$.
}
\begin{equation}
\label{DefChabauty}
\mathcal O_{K, K'} \, = \,
\left\{ N \triangleleft F_k \hskip.2cm : \hskip.2cm  
N \cap K \ = \emptyset \hskip.2cm \text{and} \hskip.2cm K' \subset N \right\} ,
\end{equation}
with $K,K'$ finite subsets in $F_k$.
This topology makes $\mathcal M_k$ a totally disconnected compact space.
It is also completely metrisable, as we now recall.
For two subsets $A,A'$ in $F_k$, let $v(A,A')$ denote the largest integer $n$
such that $A \cap B_S^{F_k}(n) = A' \cap B_S^{F_k}(n)$,
where $S = (s_1, \hdots, s_k)$ in $F_k$ is as above,
and where balls $B_S^{F_k}(n)$ are 
as in Appendix \ref{AppendixDGrowthAmenability}.
Set $d(A,A') = \exp (-v(A,A'))$.
Then $d$ is a metric (indeed an ultrametric) and makes
the set $2^{F_k}$ of subsets of $F_k$ a totally discontinuous compact metric space,
in which the space $\mathcal M_k$ of marked groups on $k$ generators
(namely the space of normal subgroups of $F_k$) is closed.
The topology induced by $d$ on $\mathcal M_k$ 
coincides with that defined by (\ref{DefChabauty}).
\par

Here is an elementary and basic fact about this topology.
The earliest written reference we know for it is
\cite[Lemma 1.3 and Lemma 1]{CoGP--07}.

\begin{prop}
\label{localinM}
Let $k,\ell$ be two positive integers.
Let $(G,S) \in \mathcal M_k$ and $(G,T) \in \mathcal M_\ell$
be two marked groups with the same underlying group.
Then there exist neighbourhoods $U \subset \mathcal M_k$ of $(G,S)$
and $V \subset \mathcal M_\ell$ of $(G,T)$
that are homeomorphic.
\par
In loose words, local properties of $(G,S)$ are properties of $G$ itself.
\end{prop}

This proposition justifies the following definitions:
a property ($\mathcal P$) of finitely generated groups is 
\textbf{open} [respectively \textbf{closed]} if,
for any positive integer $k$, the subset of $\mathcal M_k$
of marked groups $(G,S)$ such that $G$ has Property ($\mathcal P$)
is open [respectively closed].
A finitely generated group $G$ is \textbf{isolated}
if, for any ordered generating set $S = (s_1, \hdots, s_k)$ of $G$,
the point $(G,S)$ is isolated in $\mathcal M_k$.
We collect a few examples as follows:

\begin{prop} 
\label{propertiesChab}
For $k \ge 2$,
in the space $\mathcal M_k$ of marked groups of rank $k$:
\begin{itemize}
\item[(i)]
``Being abelian'' is both an open and a closed property;
more generally, for $d \ge 1$,
``being nilpotent of nilpotent class at most~$d$''
is both open and closed.
``Being nilpotent'' is open and non-closed.
\item[(ii)]
``Being soluble of solubility class at most k'' is closed and non-open.
``Being soluble'' is neither open nor closed .
\item[(iii)]
``Being finite'' and ``having torsion'' are open and  non-closed.
\item[(iv)]
If $(G,S) \in \mathcal M_k$ is a marking of a finitely presented group $G$,
there exists a neighbourhood of $(G,S)$ in $\mathcal M_k$
containing only marked quotients of $(G,S)$.
\item[(v)]
A necessary condition for $(G,S)$ to be an isolated point in $\mathcal M_k$
is that $G$ is finitely presented. 
Finite groups and finitely presented simple groups are isolated.
\item[(vi)]
There exists an isolated group that is $3$-soluble and non-Hopfian;
the group $\widehat{\mathfrak G}$ 
mentioned in Remark \ref{fewquotients}.iii is isolated.
\item[(vii)]
Amenability is neither open nor closed.
\item[(viii)]
Kazhdan Property (T) is open in $\mathcal M_k$.
\item[(ix)]
Serre Property (FA) is not open in $\mathcal M_k$.
\end{itemize}
\end{prop}

\begin{proof}[On the proof]
Claims (i) to (v) are elementary;
most of them appear explicitly in
\cite[Section 2.6 and Lemma 2.3]{ChGu--05}.
For (i), note moreover that ``being nilpotent'' is open by (iv),
because nilpotent groups are finitely presented.
For (ii), note that ``being soluble'' is non-open,
because metabelian groups like $\Z \wr \Z$
are limits of non-soluble groups (see Example \ref{AwrZ}, say).
\par
``Being nilpotent'', ``being soluble'', ``being finite'', ``being amenable''
and ``having torsion'' are non-closed properties, 
because non-abelian free groups
are residually finite $p$-groups, for any prime $p$
(due to \cite{Hall--50}, see also \cite{Vale--93}).
\par
For (v), 
observe that a finitely generated infinitely presented group $G$
is always a limit of finitely presented groups $G_n$; more precisely
\begin{equation*}
\aligned
& G \, = \,  
\langle s_1, \hdots, s_k \mid (r_i)_{i \ge 1} \rangle
\, = \, \lim_{n \to \infty} G_n
\\
& \hskip.5cm \text{with} \hskip.2cm
G_n \, = \, 
\langle s_1, \hdots, s_k \mid r_1, \hdots, r_n  \rangle .
\endaligned
\end{equation*}
Necessary \emph{and sufficient} conditions for isolated points
are known in terms of the existence of ``finite discriminating subsets''; 
we refer to  \cite[Proposition 2]{CoGP--07}; see also \cite[Theorem 2.1]{Grig--05}.
The class of isolated groups contains considerably more groups
than the finite groups and the finitely presented simple groups \cite{CoGP--07}.
\par
The first part of Claim (vi) is \cite[Proposition 10]{CoGP--07};
the second part is implicit in \cite{SaWi--02}, 
and explicit in \cite[Proposition 5.18]{CoGP--07}).
``Being amenable'' is non-open, again because
$\Z \wr \Z$ is a limit of groups with non-abelian free subgroups
(Example \ref{AwrZ}).
Claim (viii) is a result of \cite{Shal--00}, and (ix) of \cite{DuMi}.
\end{proof}

The Chabauty topology on $\mathcal M_k$ plays an important role
in connection with many group properties including
``Property LEF'' and ``soficity''.
The two latter properties define subspaces in $\mathcal M_k$ that are closed
\cite[Propositions 7.3.7 and 7.5.13]{CeCo--10}.
\par

Note the contraposition of (v):
for $(G,S) \in \mathcal M_k$ with $G$ infinitely presented,
there exists a sequence $\big( (G_n, S_n) \big)_{n \ge 1}$ 
of pairwise distinct points in $\mathcal M_k$
such that $\lim_{n \to \infty} (G_n, S_n) = (G,S)$.
\par

The simplest examples of converging sequences in $\mathcal M_k$
are of the following kind. Let
\begin{equation*}
N_1 \, \subset \, \cdots 
\, \subset \, N_n  \, \subset \, N_{n+1}  \, \subset \, \cdots \, 
\subset \, N := \bigcup_{n \ge 1} N_n
\end{equation*}
be a nested sequence of normal subgroups in $F_k$.
Let $S_0$ be a free basis of $F_k$.
Denote by $p_n : F_k \longrightarrow G_n := F_k/N_n$ ($n \ge 1$)
and $p : F_k \longrightarrow G := F_k/N$ the canonical projections.
Set $S_n = p_n(S_0)$ and $S = p(S_0)$.
Then  $\big( (G_n,S_n) \big)_{n \ge 1}$ is a sequence in $\mathcal M_k$
converging to $(G,S)$.
In this case we often suppress the emphasis on the generating sets
and write simply that 
the sequence $(G_n)_{n \ge 1}$ converges to $G$ in $\mathcal M_k$.
\par

Converging sequences in $\mathcal M_k$ need not be of this special kind,
with $G$ a quotient of $G_n$ for all $n$ large enough.
See below, Proposition \ref{Grig84}.vi.

\vskip.2cm

The following observation about $\mathcal M_k$ and covers, 
basic for us, is well-known; see e.g. \cite[Lemma 3.1]{CoKa--11}.
We provide a proof for the convenience of the reader.

\begin{prop}
\label{main}
Let $\big( (G_n,S_n) \big)_{n \ge 1}$ be a converging sequence in $\mathcal M_k$;
set $(G,S) = \lim_{n \to \infty} (G_n,S_n)$.
Let $\Gamma$ be a finitely presented group; assume there exists
a cover $\pi : \Gamma \twoheadrightarrow G$.
\par
Then $\Gamma$ is a cover of $G_n$ for $n$ large enough.
\end{prop}

\noindent
\emph{Note.} In case $G$ itself is finitely presented, 
this lemma is an immediate consequence of Proposition \ref{propertiesChab}.iv.

\begin{proof}
Denote as above by $(s_1, \hdots, s_k )$ an ordered free basis of $F_k$.
Let $p_n : F_k \twoheadrightarrow G_n$
and $p : F_k \twoheadrightarrow G$
be the free covers corresponding to $(G_n, S_n)$ and $(G,S)$ respectively.
Set $N_n = \ker (p_n)$ and $N = \ker (p)$.
Let $(t_1, \hdots, t_\ell)$ an ordered generating set of $\Gamma$.
Consider the free group $F_\ell$ on an ordered basis
$U = (u_1, \hdots, u_\ell)$ and the free cover 
$q : F_\ell \twoheadrightarrow \Gamma$
defined by $q(u_j) = t_j$ for $j = 1, \hdots, \ell$.
\par

Since $\Gamma$ is finitely presented,
there exists a finite subset $R \subset F_\ell$ 
of words $v_1, \hdots, v_m$ in the letters of $U \cup U^{-1}$
such that $\ker (q)$ is the normal subgroup of $F_\ell$ generated by $R$,
namely such that $\langle U \mid R \rangle$ is a presentation of $\Gamma$.
For $j  \in \{1, \hdots, \ell \}$, choose a word $w_j$ 
in the letters $p(s_1), \hdots, p(s_k)$ and their inverses
such that $\pi(t_j) = w_j$.
Let $\tilde w_j$ be the word in $\{s_1, s_1^{-1}, \hdots, s_k, s_k^{-1} \}$
obtained by substitution of $s_i^{\pm 1}$ in place of $p(s_i)^{\pm 1}$; 
observe that $p(\tilde w_j) = w_j = \pi(t_j)$.
Consider the homomorphism
\begin{equation*}
h \,  : \,  F_\ell \longrightarrow F_k
\hskip.5cm \text{defined by} \hskip.5cm
h(u_j) \, = \, \tilde w_j \hskip.5cm (1 \le j \le \ell ).
\end{equation*}
Then $ph(u_j) = p(\tilde w_j) = w_j = \pi(t_j) = \pi q (u_j)$ for all $j$,
so that $p h = \pi q$, and therefore $h(R) \subset N$.
\par

The last inclusion means that the open subset
\begin{equation*}
\mathcal O' \, := \, 
\left\{ M  \triangleleft F_k \hskip.2cm : \hskip.2cm h(R) \subset M \right\}
\, = \,
\bigcap_{i=1}^m \mathcal O_{\emptyset, \{h(r_i)\}}
\end{equation*}
is a neighbourhood of $N$ in $\mathcal M_k$.
Hence, for $n$ large enough, we have $N_n \in \mathcal O'$
and therefore $h(R) \subset N_n$.
\par

Denote by $\langle\langle T \rangle\rangle$ the \emph{normal} subgroup
of a group $H$ generated by a subset $T \subset H$. Let
\begin{equation*}
h_1 \,  : \, \Gamma = F_\ell  / \langle\langle R \rangle\rangle \longrightarrow
F_k / \langle\langle h(R) \rangle\rangle 
\end{equation*}
be the cover induced by $h$, and 
\begin{equation*}
h_2 \,  : \, F_k / \langle\langle h(R) \rangle\rangle 
\longrightarrow F_k / N_n \, = \, G_n
\end{equation*}
that defined by the inclusion $\langle\langle h(R) \rangle\rangle \subset N_n$
(for $n \gg 1$).
The composition $h_2h_1$ is a cover $\Gamma \twoheadrightarrow G_n$,
and this concludes the proof.
\end{proof}

An immediate consequence of the previous proposition is the following corollary,
of very frequent use in our work.

\begin{cor}
\label{main2}
Consider the three following 
group properties:
\begin{itemize}
\item[(NA)]
non-amenability,
\item[(Fr)]
containing non-abelian free groups,
\item[(La)]
being large.
\end{itemize}
Let $k \ge 2$ and
 $\big( (G_n,S_n) \big)_{n \ge 1}$ be a converging sequence in $\mathcal M_k$,
with limit $(G,S)$.
\par

If, for all $n$ large enough, $G_n$ has one of the three properties above,
then any finitely presented cover of $G$ has the same property.
\end{cor}

\begin{proof}
The point is that a group that has a quotient with one of the properties (NA), (Fr), (La) 
has itself the same property.
\end{proof}

\section{\textbf{The analogue of Theorem \ref{GrHa+i}
for the family $\left( G_\omega \right)_{\omega \in \Omega}$ of \cite{Gri--84a}}}
\label{sectionOmega}

Let $\Omega$ be the Cantor space $\{0,1,2\}^{\N}$ 
of all sequences of $0$'s, $1$'s and $2$'s, with the product topology.
Denote by $\Omega_-$ the countable subspace of eventually constant sequences,
by $\Omega_+$ its complement,
and by $\Omega_0$ the subspace of sequences with infinitely many
occurrences of each of $0, 1, 2$; thus
\begin{equation*}
\Omega_0 \, \subset \, \Omega_+
\, \subset \ \Omega \, = \, \Omega_+ \sqcup \Omega_- .
\end{equation*}
We denote by $\sigma$ the shift on $\Omega$,
defined by $(\sigma(\omega))_n = \omega_{n+1}$ for all $n \ge 1$.
\par

We will recall the construction of \cite{Gri--84a},
which is a generalisation of that of Section \ref{sectionSelfsimilar}.
It associates with each point $\omega \in \Omega$
a marked group $(G_\omega,S_\omega)$ 
with $S_\omega$ consisting of $4$ generators of order $2$;
for example, $\mathfrak G = G_{\overline{012}}$,
where $\overline{012}$ stands for the $3$-periodic sequence $012012012\cdots$.
In this section, set
\begin{equation*}
X \, = \, \{0,1\} 
\end{equation*}
and identify $X^*$ with the $2$-regular rooted tree.
We proceed to define for all $\omega \in \Omega$ a marked group 
$(G_\omega. S_\omega) \in \mathcal M_4$ of automorphisms of $X^*$.

\begin{defn!}
\label{DefGOmega}
The flip $a \in \operatorname{Aut}(X^*)$ is defined by
\begin{equation*}
a(0v) \, = \, 1v \hskip.2cm \text{and} \hskip.2cm a(1v) \, = \, 0v
\hskip.2cm \text{for all} \hskip.2cm v \in X^*.
\end{equation*}
Set
\begin{equation*}
\begin{array}{ccc}
a_{\beta(0)} = a \hskip.2cm &  a_{\beta(1)} = a \hskip.2cm & a_{\beta(2)} = 1 \phantom{.}
\\
a_{\gamma(0)} = a \hskip.2cm &  a_{\gamma(1)} = 1 \hskip.2cm & a_{\gamma(2)} = a \phantom{.}
\\
a_{\delta(0)} = 1 \hskip.2cm &  a_{\delta(1)} = a \hskip.2cm & a_{\delta(2)} = a .
\end{array}
\end{equation*}
Define for each $\omega = (\omega_n)_{n \ge 1} \in \Omega$
a set $S_\omega = \{a, b_\omega, c_\omega, d_\omega\}$ 
of four automorphisms of $X^*$ by
\begin{equation*}
\aligned
b_\omega \, &= \,  \left( a_{\beta(\omega_1)} , b_{\sigma(\omega)} \right) 
\\
c_\omega \, &= \,  \left( a_{\gamma(\omega_1)} , c_{\sigma(\omega)} \right) 
\\
d_\omega \, &= \,  \left( a_{\delta(\omega_1)} , d_{\sigma(\omega)} \right).
\endaligned
\end{equation*}
It is easy to check that
\begin{equation}
\label{Relations0InGomega}
\aligned
ac_\omega a \, &= \, \left( b_{\sigma (\omega)}, a_{\beta(\omega_1)} \right)
\\
ad_\omega a \, &= \, \left( c_{\sigma (\omega)}, a_{\gamma(\omega_1)} \right)
\\
ab_\omega a \, &= \, \left( d_{\sigma (\omega)}, a_{\delta(\omega_1)} \right)
\\
a^2 \, = \, b_\omega^2 \, &= \, c_\omega^2 \, = \, d_\omega^2 \, = \,
b_\omega c_\omega d_\omega = 1 .
\endaligned
\end{equation}
We define the group
\begin{equation*}
G_\omega \, = \ \langle a, b_\omega, c_\omega, d_\omega \rangle
\end{equation*}
generated by $S_\omega$; it is a subgroup of  $\operatorname{Aut}(X^*)$.
It follows from the last line of
(\ref{Relations0InGomega}) that any element of $G_\omega$
can be written as 
\begin{equation}
\label{ElementsInGomega}
(\ast) a \ast a \ast \cdots a (\ast)
\end{equation}
with $\ast \in \{ b_\omega, c_\omega, d_\omega \}$, 
$(\ast) \in \{1,  b_\omega, c_\omega, d_\omega \}$,
and $n \ge 0$ occurrences of $a$.
\end{defn!}

Observe that any permutation $\tau$
of $\{0,1,2\}$ induces 
a permutation of $\Omega$, again denoted by $\tau$;
the groups $G_{\tau(\omega)}$ and  $G_\omega$ are isomorphic.
\par

In \cite[Section 6]{Gri--84a}, there is moreover a modified construction
providing a marked group $(\widetilde G_\omega, \widetilde S_\omega)$;
we refer to the original paper.
Note that (v) below holds for the modified groups, 
but not for the groups $G_\omega$.

\begin{prop}
\label{Grig84}
Let $\Omega = \{0,1,2\}^{\N}$.
For $\omega \in \Omega$, 
let $G_\omega$ and $\widetilde G_\omega$ be as above.
\begin{itemize}
\item[(i)]
For $\omega \in \Omega$, the groups
$G_\omega$ and $\widetilde G_\omega$ are both infinite,
and $\widetilde G_\omega$ is infinitely presented.
\item[(ii)]
For $\omega \in \Omega_+$, the marked groups
$(G_\omega, S_\omega)$ and $(\widetilde G_\omega, \widetilde S_\omega)$
are isomorphic;
the group $G_\omega$
is of intermediate growth.
\par
For $\omega \in \Omega_0$, the group $G_\omega$
is an infinite $2$-group, and is just infinite.\footnote{Let
$\Omega_1$ be the subset of $\Omega_+$
of sequences containing infinitely many occurrences
of two of $0,1,2$, and finitely many occurrences of the third,
so that $\Omega_+ = \Omega_0 \sqcup \Omega_1$.
For $\omega \in \Omega_1$, the group $G_\omega$
is not a $2$-group, indeed it has elements of infinite order.}
\item[(iii)]
For $\omega \in \Omega_-$, 
the group $G_\omega$ is
virtually free abelian, 
and consequently finitely presented of polynomial growth,
while the group $\widetilde G_\omega$ is virtually metabelian
and of exponential growth.
\item[(iv)]
For $\omega, \omega' \in \Omega_+$,
the groups $G_{\omega}$ and $G_{\omega'}$ are isomorphic
if and only if $\omega' = \tau(\omega)$ for some permutation $\tau$
of $\{0,1,2\}$.
\item[(v)]
The mapping $\Omega \longrightarrow \mathcal M_4$,
$\omega \longmapsto (\widetilde G_\omega, \widetilde S_\omega)$ 
is a homeomorphism onto its image.
\item[(vi)]
For a converging sequence $(\omega_{(n)})_{n \ge 1}$ 
of points in $\Omega_+$ with a limit $\omega$ in $\Omega_+$, 
we have $\lim_{n \to \infty} (G_{\omega_{(n)}}, S_{\omega_{(n)}})
= (G_\omega, S_\omega)$ in $\mathcal M_4$.
If, moreover, $\omega_n \in \Omega_0$ for all $n \ge 1$ and $\omega \notin \Omega_0$,
then, for all $n \ge 1$, 
the group $G_\omega$ is \emph{not} a quotient of $G_{\omega_{(n)}}$.
\end{itemize}
\end{prop}

\begin{proof}[On the proof]
Most of this is proved in \cite{Gri--84a}; more precisely:
\par

(i) $G_\omega$ is infinite \cite[Theorem 2.1]{Gri--84a}
and $\widetilde G_\omega$ is infinitely presented \cite[Theorem 6.2]{Gri--84a}.
\par  

(ii) For $\omega \in \Omega_+$, we have
$(G_\omega, S_\omega) = (\widetilde G_\omega, \widetilde S_\omega)$
\cite[observation just before Theorem 6.1]{Gri--84a},
and $G_\omega$ is  of intermediate growth
\cite[Corollary 3.2]{Gri--84a}.
For $\omega \in \Omega_0$,
$G_\omega$ is an infinite $2$-group that is just infinite
\cite[Theorems 2.1 and  8.1]{Gri--84a}.
\par

(iii) $G_\omega$ is virtually free abelian 
\cite[Theorem 2.1.(3)]{Gri--84a},
while $\widetilde G_\omega$ is virtually
metabelian and of exponential growth
\cite[Theorem 6.1]{Gri--84a}.
\par

About (iv), see \cite[Theorem 2.10.13]{Nekr--05}.
A weaker statement is proved in \cite[Section 5]{Gri--84a}. 
\par

For (v), see \cite[Proposition 6.2]{Gri--84a}.
\par

For (vi), given any $n \ge 1$, 
note that $G_\omega$ is neither isomorphic to $G_{\omega_{(n)}}$, by (iv),
nor a non-trivial quotient of $G_{\omega_{(n)}}$, by (ii).
\end{proof}

For the main result of this section (Theorem \ref{Muntyan}),
we will need an analogue in the present context 
of the homomorphisms (\ref{GivenPhi}) and (\ref{EqPhin}) 
of Section~\ref{sectionSelfsimilar}.
Recall that we have a natural \emph{isomorphism}
\begin{equation*}
\Phi_X \, : \, 
\operatorname{Aut}(X^*) \, \overset{\simeq}{\longrightarrow} \, 
\operatorname{Aut}(X^*) \wr S_2 .
\end{equation*}
We keep the notation of Definition \ref{DefGOmega}.

\begin{defn!}
\label{DefPhiOmega}
Let $\omega \in \Omega$. The restriction to $G_\omega$
of the isomorphism $\Phi_X$ provides an injective homomorphism
\begin{equation*}
\Phi_\omega^{(1)} =
\Phi_\omega \, : \, G_\omega \longrightarrow G_{\sigma(\omega)} \wr S_2 .
\end{equation*}
On the generators, we have
\begin{equation*}
\aligned
\Phi_{\omega} (a) \, &= \, (1,1)\tau
\\
\Phi_{\omega} (b_\omega) \, &= \, (a_{\beta(\omega_2)} , b_{\sigma(\omega)})
\\
\Phi_{\omega} (c_\omega) \, &= \, (a_{\gamma(\omega_2)} , c_{\sigma(\omega)})
\\
\Phi_{\omega} (d_\omega) \, &= \, (a_{\delta(\omega_2)} , d_{\sigma(\omega)})
\endaligned
\end{equation*}
(recall that $S_2 = \{1, \tau\}$).
The sequence of homomorphisms
$\big( \Phi_\omega^{(n)} \big)_{n \ge 1}$
is defined inductively by
\begin{equation*}
\Phi_\omega^{(n)} \, : \,
G_\omega 
\overset{ \Phi_{\omega\phantom{ {}^n }}^{(n-1)} }{\longrightarrow}
G_{\sigma^{n-1}(\omega)} \wr^{n-1} S_2
\overset{ \Phi_{\sigma^{n-1}(\omega)}^{(1)} \wr 1_{d^{n-1}} }{\longrightarrow}
G_{\sigma^{n}(\omega)} \wr^{n} S_2 .
\end{equation*}
\end{defn!}

\begin{lem}[contraction in $G_\omega$]
\label{ContractionInGomega}
Let $\omega \in \Omega$. We keep the notation above.
\par

(i) 
For each $n \ge 1$, the homomorphism $\Phi_\omega^{(n)}$ is injective.
\par

(ii)
For all $g \in G_\omega$, there exists an integer $n \ge 1$ such that
\begin{equation*}
\Phi_\omega^{(n)}(g) = \big( (g_v)_{v \in X^{n}} , \tau_g^{(n)} \big)
\end{equation*}
with $g_v \in \{1, a, b_{\sigma^n(\omega)}, c_{\sigma^n(\omega)}, d_{\sigma^n(\omega)}\}
\hskip.1cm \forall \hskip.1cm v \in X^n$ 
and $\tau_g^{(n)} \in S_2$.
\end{lem}

\begin{proof}
By induction on the length of $g$, in the sense of (\ref{ElementsInGomega}).
\end{proof}

In Theorem \ref{ThGri84} of the Introduction,
the claim on intermediate growth is a repetition of part of Proposition \ref{Grig84},
and the claim on covers is the theorem below.

\begin{thm}
\label{Muntyan}
For $\omega \in \Omega_+$, 
any finitely presented cover  of $G_\omega$ is large.
\end{thm}

\begin{rem!}
(1)
Let $\omega \in \Omega_-$.
Any finitely presented cover  of the infinitely presented group $\widetilde G_\omega$
contains non-abelian free groups, by Theorem \ref{ThBieriStrebel}.
As recorded in Proposition \ref{Grig84}.iii, the group $G_\omega$
is virtually free abelian, and finitely presented.
For example, if $\omega$ is the constant sequence $000\cdots$,
then $G_\omega$ is the infinite dihedral group.
\par
(2)
If we replace ``is large'' by ``contains non-abelian free subgroups'' in Theorem \ref{Muntyan},
the resulting statement has a short proof. More precisely:
\begin{itemize}
\item[]
\emph{For any $\omega \in \Omega$, 
any finitely presented cover of $\widetilde G_\omega$
has non-abelian free subgroups.}
\end{itemize}
Indeed, 
let $(\omega_n)_{n \ge 1}$ be a sequence of eventually constant sequences 
converging to $\omega$ in $\Omega$.
Then $\widetilde G_{\omega_n}$ is virtually metabelian and infinitely presented for all $n \ge 1$
(Claims (i) and (iii) in Proposition \ref{Grig84}),
and $(\widetilde G_{\omega_n})_{n \ge 1}$ converges to $\widetilde G_\omega$
(Proposition \ref{Grig84}.v). 
Let $E$ be a finitely presented cover of $\widetilde G_\omega$. 
Then $E$ is a cover of $G_{\omega_n}$ for $n$ large enough
(Proposition \ref{main}).
Hence Bieri-Strebel Theorem \ref{ThBieriStrebel} shows that $E$ contains non-abelian free subgroups.

\end{rem!}

From now on, we assume that 
\begin{equation*}
\omega \, \in \, \Omega_+ .
\end{equation*}
Our strategy for the proof of Theorem \ref{Muntyan}
is to adapt to the present context the steps of Section \ref{sectionSelfsimilar}.

The following definition should be compared 
with Definition \ref{DefGnEtc}.
Note however that $G_0$ has not quite the same meaning here and there.

\begin{defn!}
Set again 
\begin{equation*}
G_0 \, = \,  \langle a,b,c,d \mid a^2, b^2, c^2, d^2, bcd \rangle
\, \simeq C_2 \ast V ,
\end{equation*}
as in Example \ref{Ex1stGrig}.
Observe that any element of $G_0$
can be written as 
\begin{equation}
\label{ElementsInG0}
(\ast) a \ast a \ast \cdots a (\ast)
\end{equation}
with $\ast \in \{ b, c, d \}$, 
$(\ast) \in \{1,  b, c, d \}$,
and $n \ge 0$ occurrences of $a$
(compare with Equation (\ref{ElementsInGomega})).
\par

For $i \in \{0,1,2\}$, set
\begin{equation*}
\varphi_i(a) = (1,1)\tau \hskip.2cm \text{for all} \hskip.2cm i \in \{0,1,2\}
\end{equation*} 
and
\begin{equation*}
\begin{array}{cccccccc}
\varphi_0(b) &= (a,b) && \varphi_1(b) &= (a,b) && \varphi_2(b) &= (1,b) \phantom{.}
\\
\varphi_0(c) &= (a,c) && \varphi_1(c) &= (1,c) && \varphi_2(c) &= (a,c) \phantom{.}
\\
\varphi_0(d) &= (1,d) &&  \varphi_1(d) &= (a,d) && \varphi_2(d) &= (a,d) .
\end{array}
\end{equation*}
It is easy to check that these formulas define homomorphisms
\begin{equation*}
\varphi_i : G_0 \longrightarrow G_0 \wr S_2  \hskip.5cm (i = 0,1,2).
\end{equation*} 
\par

Set $\varphi_{\omega}^{(1)} = \varphi_{\omega_1}$
and define, inductively for $n \ge 2$, homomorphisms
\begin{equation*}
\varphi_{\omega}^{(n)} \,  : \,  G_0 
\overset{ \varphi_{\omega}^{(n-1)} }{\longrightarrow}
G_0 \wr^{n-1} S_2
\overset{ \varphi_{\omega_n} \wr 1_{2^n} }{\longrightarrow}
G_0 \wr^n S_2 .
\end{equation*}
For $n \ge 1$, set
\begin{equation*}
N_{\omega,n} \, = \, \ker ( \varphi_\omega^{(n)} )
\hskip.2cm \text{and} \hskip.2cm
G_{\omega,n} \, = \, G_0 / N_{\omega,n} .
\end{equation*}
\par

We have natural homomorphisms
\begin{equation*}
\aligned
\pi_\omega \, : \, G_0 \, & \longrightarrow \, G_\omega,
\\
\widehat \pi_\omega = \widehat \pi_{\omega,1} \, : \, 
G_0 \wr S_2  \, &  \longrightarrow \, G_{\sigma(\omega)} \wr S_2,
\\
\widehat \pi_{\omega,n} \, : \, 
G_0 \wr^n S_2 \, & \longrightarrow \, G_{\sigma^n(\omega)} \wr^n S_2 .
\endaligned
\end{equation*}
(Compare with (\ref{DefPi}), (\ref{DefPiHat}), and  (\ref{DefPiHatN}),
but note that $\widehat \pi_{\omega,1} = \pi_\omega \wr 1_2$ \emph{does not hold} here.)
\end{defn!}

The next lemma is about diagrams 
analogous to (\ref{DiagramInf}) and (\ref{DiagramN}). 
Its proof uses an argument similar to one in the proof of Proposition \ref{CoveringSSGroup},
and will be omitted.

\begin{lem}
\label{DiagramsOmega}
The diagram
\begin{equation}
\label{DiagramNomega}
\begin{array}{ccc}
   G_0     &\overset{\varphi_\omega^{(n)}}{\longrightarrow}  
                & G_0 \wr^n S_2 
\\
&&
\\
  \pi_\omega \downarrow  &          & \downarrow\widehat\pi_{\omega,n}
\\
&&
\\
   G        & \overset{\Phi_\omega^{(n)}}{\longrightarrow}  
                & G_{\sigma^n(\omega)} \wr^n S_2
\end{array}
\end{equation}
commutes for each $n \ge 1$.
\end{lem}

The next lemma is analogous to Step 2 in the proof of Proposition \ref{CoveringSSGroup}.

\begin{lem}[contraction in $G_0$] 
\label{ContractionInGzero}
For all $k \in G_0$, there exists an integer $n \ge 1$ such that
\begin{equation*}
\varphi_\omega^{(n)}(k) = \big( (k_v)_{v \in X^{n}} , \tau_k^{(n)} \big)
\end{equation*}
with $k_v \in \{1, a, b, c, d \} \hskip.1cm \forall \hskip.1cm v \in X^n$
and $\tau_k^{(n)} \in S_2$.
\end{lem}

\begin{proof}[Proof:]
by induction on the length of $k$, in the sense of (\ref{ElementsInG0}).
\end{proof}

Define now 
\begin{equation*}
N_\omega \, = \,  \bigcup_{n \ge 1} N_{\omega,n} 
\end{equation*}
(compare with Definition \ref{DefGnEtc}).
The two following lemmas are appropriate modifications of 
Lemmas \ref{N=KerPi} and \ref{backwardsalongG_n's};
we repeat the proof for the first one, and not for the second one.

\begin{lem}
\label{kernelsGomega}
We have
\begin{equation*}
N_\omega \, = \,  
\ker \left( \pi_\omega :   G_0 \longrightarrow G_\omega \right) ,
\hskip.2cm \text{namely} \hskip.2cm G_\omega \simeq G_0/N_\omega ,
\end{equation*}
so that
\begin{equation*}
\lim_{n \to \infty} G_{\omega,n} \, = \, G_\omega 
\end{equation*}
in the space of marked groups on $4$ generators.
\end{lem}

\begin{proof}
Let $g \in N$. Let $n \ge 1$ be such that $g \in \ker (\varphi_\omega^{(n)})$.
Since $\Phi_\omega^{(n)} \pi_\omega (g) = 
\widehat \pi_{\omega,n} \varphi_\omega^{(n)} (g)$,
we have $\pi_\omega(g) = 1$ by Lemma  \ref{ContractionInGomega}.i.
\par

Conversely, let $k \in G_0$. 
There exists $n \ge 0$ such that
$\left( \varphi_\omega^{(n)}(k) \right)_v \in \{1,a,b,c,d\}$ for all $v \in X^n$,
by Lemma \ref{ContractionInGzero}.
Assume that $k \in \ker (\pi_\omega)$.
Then $\widehat \pi_{\omega,n} (\varphi_\omega^{(n)} (k)) = 1$.
As $\widehat \pi_{\omega,n}$ is injective ``on generators''
(in a sense similar to that of Remark \ref{RestrictionPiOnS}),
we have $\varphi_\omega^{(n)}(k) = 1$,
and therefore $k \in N_{\omega,n} \subset N_\omega$.
\par

(Note that the hypothesis ``$\omega \in \Omega_+$'' is necessary
for the previous argument. 
If $\omega_n$ were eventually constant,
one of $b_{\sigma^n(\omega)}, c_{\sigma^n(\omega)}, d_{\sigma^n(\omega)}$ 
would be the identity of $G_{\sigma^n(\omega)}$ for $n$ large enough).
\end{proof}

\begin{lem}
\label{onemoreeffortcompanion}
In the situation of the previous lemma, we have for all $n \ge 1$
\begin{equation*}
\varphi_\omega^{(1)} (N_{\omega,n}) \subset N_{\omega, n-1}^2 \subset G_0 \wr S_2
\hskip.2cm \text{and} \hskip.2cm
\left( \varphi_\omega^{(1)}\right) ^{-1}(N_{\omega, n-1}^2) \subset N_{\omega,n} .
\end{equation*}
It follows that $\varphi_\omega^{(1)} : G_0 \longrightarrow G_0 \wr S_2$ 
induces a homomorphism
\begin{equation*}
\psi_\omega^{(n)} \, : \, \left\{
\aligned
G_{\omega,n} \, &\longrightarrow \, \hskip1cm G_{\omega, n-1} \wr S_2
\\
gN_{\omega,n} \, &\longmapsto \, 
\left( 
\big( (\varphi_{\omega_n}(g))_x N_{\omega, n-1} \big)_{x \in X} , \tau_g^{(1)} \right)
\endaligned
\right.
\end{equation*}
which is injective.
\end{lem}

\begin{prop}
For each $\omega \in \Omega_+$ and $n \ge 0$, 
the group $G_{\omega,n}$ is large.
\end{prop}

\begin{proof}
The group $G_0 = C_2 \ast V$ has a free subgroup of finite index,
indeed a subgroup isomorphic to $F_3$ of index $8$.
For $n \ge 1$, because of the previous lemma and
as in the proof of Theorem \ref{FirstThmOnSSGps},
there exists a subgroup of index $2$ in $G_{\omega,n}$
and a homomorphism from this subgroup onto $G_{\omega, n-1}$.
It follows by induction on $n$ that $G_{\omega,n}$ is large.
\end{proof}

\begin{proof}[End of proof of Theorem \ref{Muntyan}]
Since $G_{\omega,n}$ is large for $n$ large enough, 
it follows from Lemma \ref{kernelsGomega} and Corollary \ref{main2}
that any cover of $G_\omega$ is large.
\end{proof}

\begin{defn!}
\label{LambdaUniversal}
For $\omega \in \Omega$, let $M_\omega$ denote
the kernel of the defining cover $F_4 \twoheadrightarrow G_\omega$;
in other terms, $M_{\omega}$ is the inverse image of $N_{\omega}$
by the epimorphism $F_4 \twoheadrightarrow G_0$
mapping the four generators of $F_4$ onto $a,b,c,d \in G_0$.
For a subset $\Psi$ of $\Omega$, the \textbf{$\Psi$-universal group}
is the group
\begin{equation*}
\mathcal U_\Psi \, = \, F_4 / \bigcap_{\omega \in \Psi} M_\omega .
\end{equation*}
For example, $\mathcal U_\emptyset = \{1\}$, 
and $\mathcal U_{\{\omega\}} = G_\omega$ for all $\omega \in \Omega$.
\end{defn!}

The terminology is justified by cases
that have appeared in the literature,  with $\Psi$ large.
For example, let $\Lambda$ denote the subset of $\Omega_0$
of sequences that are concatenations of blocks $012, 120, 201$.
Then $\mathcal U_\Lambda$ has uncountably many quotients
(a consequence of Proposition \ref{Grig84}.iv);
it has intermediate growth, and therefore is amenable
(established in \cite[Theorem 9.7]{Grig}).
\par

Suppose that $\Psi$ contains some $\omega \in \Omega_+$.
Then any cover of $\mathcal U_\Psi$ is a cover of $G_\omega$.
Theorem \ref{Muntyan} implies: 

\begin{cor}
\label{Lambdauniversal}
For any $\Psi \subset \Omega$ such that $\Psi \cap \Omega_+ \ne \emptyset$,
the $\Psi$-universal group $\mathcal U_\Psi$ is infinitely presented,
and any finitely presented cover of it is large.
\end{cor}                                                                                                                                                                                                                

In particular, this corollary solves 
the first part of Problem 9.5 in \cite{Grig--05}, 
by showing that $\mathcal U_\Omega$ is infinitely presented.

\section{\textbf{The group of intermediate growth $\mathfrak G$}}
\label{sectionGrigGroup}

Let $\mathfrak G$ be the self-similar group of degree $2$ 
of Example \ref{Ex1stGrig}. 
On the one hand, $\mathfrak G$ is a group of the family studied in the previous section:
$\mathfrak G = G_{ \overline{012} }$;
thus Theorem \ref{Muntyan} ``contains'' Theorem \ref{GrHa+i}.
On the other hand, in this particular case, 
we can describe much more precisely 
a sequence of finitely presented covers converging to $\mathfrak G$, 
and this is the subject of the present section.
Note however that, even if the cover $\mathfrak G_{-1}$ below
is the same as $G_0$ in Example \ref{Ex1stGrig}, 
the sequence $(\mathfrak G_n)_{n \ge 0}$ \emph{is not}
the particular case for $\mathfrak G$ of the sequence $(G_n)_{n \ge 1}$ (even shifted)
of Section \ref{sectionSelfsimilar}.
\par

Immediately after its discovery 
it was observed that $\mathfrak G$ is infinitely presented. 
Then, Lysenok found a presentation that we recall below.
\par

Set
\begin{equation*}
\mathfrak G_{-1} \, = \, \left\langle a,b,c,d \mid
a^2 = b^2 = c^2 = d^2 = bcd = 1 \right\rangle \, \simeq \,
C_2 \ast V ,
\end{equation*}
and denote by $S$ the system of four involutions $\{a,b,c,d\}$
generating $\mathfrak G_{-1}$.
Elements in $\mathfrak G_{-1}$ are in natural bijection with ``reduced words'' of the form
\begin{equation*}
t_0 a t_1 a \cdots a t_{k-1} a t_k
\end{equation*}
with $k \ge 0$, $t_1, \hdots, t_{k-1} \in \{b,c,d\}$,
and $t_0, t_k \in \{ \emptyset, b,c,d \}$.
Throughout the remainder of this section,
we use the same symbol to denote an element of $\mathfrak G_{-1}$
and its image in any quotient of $\mathfrak G_{-1}$, in particular in $\mathfrak G$;
thus, $S = \{a,b,c,d\}$ denotes a set of generators
in $\mathfrak G_{-1}$ \emph{and} in any quotient of $\mathfrak G_{-1}$.
\par

The substitution $\sigma$ defined by
\begin{equation*}
\sigma(a) = aca , \hskip.2cm \sigma(b) = d, \hskip.2cm
\sigma(c) = b , \hskip.2cm \sigma(d) = c
\end{equation*}
extends to reduced words,
for example $\sigma(abac) = acadacab$, and the resulting map
\begin{equation*}
\sigma \, : \, \mathfrak G_{-1} \longrightarrow \mathfrak G_{-1}
\end{equation*}
is a group endomorphism. Define 
\begin{equation*}
\begin{array}{ccc}
u_0 = (ad)^4        \hskip.2cm &  u_n = \sigma^n(u_0) & \forall n \ge 0
\\
v_0 = (adacac)^4 \hskip.2cm &  v_n = \sigma^n(v_0) & \forall n \ge 0
\end{array}
\end{equation*}

\begin{thm}[\cite{Lyse--85}]
\label{Lysenok}
The group $\mathfrak G$ has a presentation
\begin{equation*}
\left\langle a,b,c,d \mid
a^2 = b^2 = c^2 = d^2 = bcd = 1 , \hskip.2cm
u_n = v_n = 1 \hskip.2cm \forall n \ge 0
\right\rangle .
\end{equation*}
\end{thm}

\noindent \emph{Note.}
It is moreover known that this presentation is minimal \cite{Grig--99}.
Lysenok's presentation if the prototype of what is now called
an \emph{L-presentation} \cite{Bart--03}.

\begin{defn!}
\label{defGn}
For $n \ge 0$, define a pair $(\mathfrak G_n,S) \in \mathcal M_4$ by
\begin{equation*}
\aligned
&\mathfrak G_n \, = \, 
\left\langle a,b,c,d \hskip.1cm \bigg\vert \hskip.1cm 
\aligned
&a^2 = b^2 = c^2 = d^2 = bcd = 1
\\
&u_0 = \cdots = u_n = v_0 = \cdots = v_{n-1} = 1
\endaligned
\right\rangle 
\\
&S \, = \, \{a,b,c,d\} \, \subset \, \mathfrak G_n .
\endaligned
\end{equation*}
\end{defn!}

Observe that $\lim_{n \to \infty} (\mathfrak G_n,S) = (\mathfrak G, S)$
in $\mathcal M_4$, and that there are natural surjections
$\mathfrak G_{-1} \twoheadrightarrow \mathfrak G_n \twoheadrightarrow \mathfrak G$
for all $n \ge 0$.

\begin{thm}
\label{improved}
For each $n \ge 0$, 
the group $\mathfrak G_n$ has a normal subgroup $H_n$ of index $2^{2^{n+1}+2}$ 
which is isomorphic to the direct product of $2^n$ free groups of rank $3$. 
\end{thm}

\begin{rem!}
(i)
A weaker result was first established in \cite{GrHa--01}:
For each $n \ge 0$, $\mathfrak G_n$ contains a subgroup of finite index 
isomorphic to the direct product
of $2^n$ copies of finitely generated non-abelian free groups.
This by itself implies that any finitely presented cover of $\mathfrak G$
contains non-abelian free subgroups.

\vskip.2cm

(ii)
The result of \cite{GrHa--01} was improved in \cite{BaCo--06}:
For each $n \ge 0$, the group $\mathfrak G_n$ has a normal subgroup
$H_n$ of index $2^{\alpha_n}$, where $\alpha_n \leq (11\cdot 4^n+1)/3$, 
and $H_n$ is a subgroup of index $2^{\beta_n}$ in a finite direct
product of $2^n$ non-abelian free groups of rank $3$, 
where $\beta_n \leq (11\cdot 4^n-8)/3-2^n$.

\vskip.2cm

(iii)
Our proof of Theorem \ref{improved} is split in several lemmas,
until \ref{endofproof}.
\end{rem!}

If $x, \hdots, y$ are elements of a group $H$,
we denote by $\langle x, \hdots , y \rangle_H$ the subgroup of $H$ 
they generate, 
and by $\langle\langle x, \hdots, y \rangle\rangle_H$ the \emph{normal} subgroup of $H$
they generate. Define first
\begin{equation*}
\aligned
B_0 \, &= \, \langle\langle b \rangle\rangle_{\mathfrak G_0},
\\
\Xi_0 \, &= \, \langle b,c,d,aba,aca,ada \rangle_{\mathfrak G_0},
\\
D_0 \, &= \,  \langle a, d \rangle_{\mathfrak G_0},
\\
D_0^{\operatorname{diag}} \, &= \,  \langle (a,d), (d,a) \rangle_{\mathfrak G_0}.
\endaligned
\end{equation*}
It is easy to check that $D_0^{\operatorname{diag}} \cap (B_0 \times B_0) = \{1\}$,
and that $D_0^{\operatorname{diag}}$ normalizes $B_0 \times B_0$.
The assignment
\begin{equation*}
\begin{array}{ccccccc}
b & \mapsto (a,c) & & & aba & \mapsto & (c,a) \\
c & \mapsto (a,d) & & & aba & \mapsto & (d,a) \\
d & \mapsto (1,b) & & & aba & \mapsto & (b,1) 
\end{array}
\end{equation*}
extends to a group homomorphism 
$\psi_0 : \Xi_0 \longrightarrow \mathfrak G_0 \times \mathfrak G_0$
\cite[Proposition 1]{GrHa--01}.
For each $n \ge 0$, define now
\begin{equation*}
\aligned
N_n \, &= \, 
\langle\langle u_0, \hdots, u_n, v_0, \hdots, v_{n-1} \rangle\rangle_{\mathfrak G_0} ;
\hskip.2cm \text{observe that} \hskip.2cm N_n \subset \Xi_0 ;
\\
\mathfrak G_n \, &= \, \mathfrak G_0 / N_n \hskip.2cm \text{and} \hskip.2cm
\pi_n : \mathfrak G_0 \twoheadrightarrow \mathfrak G_n \hskip.2cm \text{the canonical projection};
\\
B_n \, &= \,  \langle\langle b \rangle\rangle_{\mathfrak G_n} \, =  \, \pi_n(B_0);
\\
\Xi_n \, &= \,  \langle b,c,d,aba,aca,ada \rangle_{\mathfrak G_n} \, = \, \pi_n(\Xi_0) ;
\\
D_n^{\operatorname{diag}} \, &= \,  
\langle (a,d), (d,a) \rangle_{\mathfrak G_n \times \mathfrak G_n} ;
\\
\sigma_n \, &: \, \mathfrak G_{n-1} \longrightarrow \mathfrak G_n, \hskip.2cm
gN_{n-1} \longmapsto \sigma(g)N_n 
\hskip.5cm (\text{for $n \ge 1$ only}).
\endaligned
\end{equation*}
For the definition of the homomorphism $\sigma_n$,
note that  $\sigma(N_{n-1}) \subset N_n$.

\begin{lem}[\cite{GrHa--01}, Lemma 3]
\label{lemmaB_0}
Let $B_0$ denote the \emph{normal} subgroup of $\mathfrak G_0$ generated by $b$. Then:
\begin{itemize}
\item[(i)]
$B_0$ is of index $8$ in $\mathfrak G_0$;
\item[(ii)]
$B_0$ is generated by the four elements
\item[]
$\xi_1 := b$, $\xi_2 := aba$, $\xi_3 :=  dabad$, $\xi_4 = adabada$;
\item[(iii)]
$B_0$ has the presentation
$\langle \xi_1,\xi_2,\xi_3,\xi_4 \mid
\xi_1^2=\xi_2^2=\xi_3^2=\xi_4^2=1 \rangle$;
\item[(iv)]
$B_0$ contains $N_n$ for all $n \ge 1$.
\end{itemize}
\end{lem}

\begin{lem}[\cite{GrHa--01}, mostly Proposition 10]
\label{inj}
(i) The kernel and the image of the homomorphism $\psi_0$ are given by
\begin{equation*}
\aligned
\ker(\psi_0) \, &= \, \langle\langle u_1, v_0 \rangle\rangle_{\Xi_0} ,
\\
\operatorname{Im}(\psi_0) \, &= \, (B_0 \times B_0) \rtimes D_n^{\operatorname{diag}}
\hskip.2cm \text{of index $8$ in} \hskip.2cm \mathfrak G_0 \times \mathfrak G_0 .
\endaligned
\end{equation*}
\par

(ii)
For $n \ge 1$, the homomorphism $\psi_0$ induces an isomorphism
\begin{equation*}
\psi_n \, : \, \Xi_n \overset{\simeq}{\longrightarrow} 
(B_{n-1} \times B_{n-1}) \rtimes  D_{n-1}^{\operatorname{diag}}
\, <_8 \, \mathfrak G_{n-1} \times \mathfrak G_{n-1}
\end{equation*}
where $<_8$ indicates that
the left-hand side is a subgroup of index 8 in the right-hand side.
\end{lem}

Set $K_0 = \langle\langle (ab)^2 \rangle\rangle_{\mathfrak G_0}$;
observe that $K_0 \subset B_0$.

\begin{lem}
\label{induction_0}
(i)
The subgroup $K_0$ is of index 2 in $B_0$. It is generated by 
\begin{equation*}
t=(ab)^2\quad v=(bada)^2 \quad w=(abad)^2
\end{equation*}
Moreover $K_0$ contains $N_n$ for $n \ge 1$. 
\par

(ii)
The group $K_0$ is a free group of rank 3.
\end{lem}

\begin{proof}
(i) This follows from \cite[Page 230]{Harp--00}.
Since $B_0$ contains 
$N_n$ and each $u_n,v_n$ is a fourth power, 
necessarily $N_n$ is contained in $K_0$.
\par

For (ii), see \cite[Proposition 4]{BaCo--06},
where the proof uses Kurosh's theorem. 
Alternatively one can use the Reidemeister-Schreier method 
to find a presentation for $K_0$ and see that it is indeed free of rank 3.
\end{proof}

\begin{lem}
\label{bes}
If $g$ is an element of $B_{n-1}$ then 
\begin{equation*}
\psi_n(\sigma_n(g)) \, = \ (1,g) 
\hskip.5cm \text{and} \hskip.5cm
\psi_n(a\sigma_n(g)a) \, = \, (g,1) .
\end{equation*}
\end{lem}

\begin{proof}
For the generators of $B_{n-1}$ that are images of 
those of Lemma \ref{lemmaB_0} for $B_0$,
we have
\begin{equation*}
\aligned
\psi_n(\sigma_n(b)) &= \psi_n(d) = (1,b) ,
\\
\psi_n(\sigma_n(aba)) &= \psi_n(acadaca)=(d^2,aba)=(1,aba) ,
\\
\psi_n(\sigma_n(dabad)) &= \psi_n(cacadacac)=(ad^2a,dabad)=(1,dabad) ,
\\
\psi_n(\sigma_n(adabada)) &= \psi_n(acacacadacacaca)=(1,adabada) ,
\endaligned
\end{equation*}
and this shows the first equality.
The second follows because, if $\psi_n(h) = (h_0, h_1)$,
then $\psi_n(aha) = (h_1, h_0)$.
\end{proof}

Let $K_n = K_0/N_n$. It is a normal subgroup of $\mathfrak G_n$ contained in $B_n$.

\begin{lem}
\label{yedi}
Let $n \ge 1$.
\par

(i) 
We have  $\sigma_n(K_{n-1})  \subset   K_n  \subset  B_n$.
\par

(ii)
If $H_{n-1}$ is a subgroup of $K_{n-1}$, then 
$\psi_n^{-1}(H_{n-1} \times H_{n-1}) \subset K_n$.
\end{lem}

\begin{proof}
(i)
Let $t, v, w$ be now the canonical images in $K_n$
of the elements of $K_0$ denoted by the same symbols in Lemma \ref{induction_0}.
On the one hand, we have $\psi_n(\sigma_n(t)) = (1,t)$ by Lemma \ref{bes}.
On the other hand, we have
\begin{equation*}
\psi_n(w) = \psi_n(aba) \psi_n(d) \psi_n(aba) \psi_n(d) = (cc,abab) = (1,t) 
\end{equation*}
by the definitions of $\psi_n$ and $w$.
Hence $\sigma_n(t)=w \in K_n$ by Lemma \ref{inj}.ii.
\par

Let $g_1 \in \mathfrak G_{n-1}$.
From the definition of $\psi_n$, we see that
the composition $\Xi_n \longrightarrow \mathfrak G_{n-1}$
of $\psi_n$ with a projection onto one of the factors is onto.
Hence there exists $g \in \Xi_n$ and $g_0 \in \mathfrak G_{n-1}$
such that $\psi_n(g) = (g_0,g_1)$.
We have as above\footnote{Remember that $t^h = h^{-1}th$.}
$\psi_n(\sigma_n(t^{g_1})) = (1, t^{g_1})$ and
\begin{equation*}
\psi_n(w^g)  =  \psi_n(w)^{\psi_n(g)} = (1,t)^{\psi_n(g)} = (1,t^{g_1}) ,
\end{equation*}
and therefore $\sigma_n(t^{g_1}) = w^g$.
Since $K_n$ is a normal subgroup of $\mathfrak G_n$ containing $w$,
we have $\sigma_n(t^{g_1}) \in K_n$ for all $g_1 \in \mathfrak G_{n-1}$.
The inclusion  $\sigma_n(K_{n-1}) \subset K_n$ follows,
because $K_{n-1}$ is generated by $t$ as a normal subgroup of $\mathfrak G_{n-1}$.

(ii)
Let $(h_0,h_1) \in H_{n-1} \times H_{n-1}$. We have
\begin{equation*}
\psi_n^{-1}(h_0,h_1) \, = \,
a \sigma_n(h_0)a \hskip.1cm \sigma_n(h_1)
\end{equation*}
by Lemma \ref{bes},
and the right-hand side is in $K_n$ by (i).
\end{proof}

Set $H_0 = K_0$. For $n \ge 1$, define inductively
\begin{equation*}
H_n \, = \, \psi_n^{-1}(H_{n-1} \times H_{n-1}) .
\end{equation*}
The definition makes sense by Lemma \ref{yedi}.ii.
The following lemma finishes the proof of Theorem \ref{improved}.

\begin{lem}
\label{endofproof}
Let $n \ge 0$, and the notation be as above.
\par

(i) 
$H_n$ is a normal subgroup of $\mathfrak G_n$ contained in $K_n$. 
\par

(ii) 
The group $H_n$ is a direct product of $2^n$ free groups of rank $3$.

(iii)
Its index is given by $[\mathfrak G_n : H_n] = 2^{(2^{n+1}+2)}$.
\end{lem}

\begin{proof}
For $n=0$, the three claims follow 
from Lemmas \ref{lemmaB_0} and \ref{induction_0}.
We suppose now that $n \ge 1$ and that the lemma holds for $n-1$.
\par

(i) The group $H_n$ is clearly normal in $\Xi_n$, by Lemma \ref{inj}.ii. 
To show that $H_n$ is normal in $\mathfrak G_n$, it suffices to check
that $aH_na \subset H_n$,
because $\mathfrak G_n$ is generated by $\Xi_n$ (of index $2$ in $\mathfrak G_n$) and $a$.
Let $h \in H_n$. Let $h_0, h_1 \in H_{n-1}$ be defined by
$\psi_n(h) = (h_0, h_1)$.
Then $\psi_n(aha) = (h_1, h_0) \in H_{n-1} \times H_{n-1}$,
and therefore $aha \in H_n$.
\par

(ii) This is a straightforward consequence of the isomorphism
$H_n \simeq H_{n-1} \times H_{n-1}$, 
see again Lemma \ref{inj}.
\par

(iii)
By the induction hypothesis, we have
\begin{equation*}
\aligned
&[ (B_{n-1} \times B_{n-1}) \rtimes D_{n-1}^{\operatorname{diag}} : 
H_{n-1} \times H_{n-1} ]
\\
& \hskip2cm  \, = \,
\frac{ [\mathfrak G_{n-1} \times \mathfrak G_{n-1} : H_{n-1} \times H_{n-1} ] }
       { [\mathfrak G_{n-1} \times \mathfrak G_{n-1} : 
               (B_{n-1} \times B_{n-1}) \rtimes D_{n-1}^{\operatorname{diag}}] }
               \\
& \hskip2cm  \, = \,
\frac{ 2^{2^n+2} \times 2^{2^n+1} }{2^3} \, = \, 2^{2^{n+1}+1}.
\endaligned
\end{equation*}
Thus the commutative diagram
\begin{equation*}
\begin{array}{cccc}
  & \mathfrak G_n &   & \mathfrak G_{n-1}\times \mathfrak G_{n-1}\\
 &\arrowvert & & \arrowvert  \\
 & 2 & &2^3 \\
 & \arrowvert & & \arrowvert \\
\psi_n : & \Xi_n & \xrightarrow{\cong}& (B_{n-1} \times B_{n-1}) \rtimes D_{n-1}^{diag}  \\
 & \arrowvert & & \arrowvert \\
& 2^{(2^{n+1}+1)}& & 2^{(2^{n+1}+1)}\\
& \arrowvert  & & \arrowvert  \\
 & H_n & \longrightarrow & H_{n-1}  \times H_{n-1} 
\end{array}
\end{equation*}
shows that $H_n$ has index $2^{(2^{n+1}+2)}$ in $\mathfrak G_n$.
\end{proof}

The proof of Theorem \ref{improved} is now complete.

\vskip.2cm

Recall that a group $G$ is called  of type FP$_n$ 
if the trivial $\Z [G]$-module $\Z$ has a projective resolution,
namely if there exists an exact sequence
\begin{equation*}
\cdots \, \longrightarrow \, P_{j+1} \, \longrightarrow \,  
P_j \, \longrightarrow \,  \cdots 
\, \longrightarrow \, P_1\, \longrightarrow \, P_0 \, \longrightarrow \, \Z 
\end{equation*}
whee the $P_j$ 's are projective $\Z [G]$-modules for all $j \ge 0$, 
and finitely generated projective $\Z [G]$-modules for all $j \le n$.
It is known that 
\begin{itemize}
\item[(i)]
a group is of type FP$_1$ if and only if it is finitely generated;
\item[(ii)]
finitely presented groups are of type FP$_2$;
\item[(iii)]
Condition FP$_2$ is \emph{strictly} weaker than finite presentability;
\item[(iv)]
a group is of type FP$_2$ if and only if it is the quotient of some finitely presented group
by a \emph{perfect} normal subgroup.
\end{itemize}
For (i) and (ii), see for example the notes 
in which ``type FP$_n$'' was first defined \cite{Bier--76};
see \cite{BeBr--97} for (iii) and \cite[Section VII.5, Exercise 3]{Brow--82} for (iv).
The following question is natural:

\emph{Does $\mathfrak G$ have an amenable cover of type FP$_2$?}

\noindent
The answer is due to Yves de Cornulier (unpublished).
We reproduce it here, with our thanks to him.

\begin{prop}[de Cornulier]
\label{GrigNotFP2}
Any cover of type FP$_2$ of the group $\mathfrak G$ is large.
\end{prop}

\begin{proof}
Let $E \twoheadrightarrow \mathfrak G$ be a cover, with $E$ of type FP$_2$.
By (iv) above, there exists a finitely presented group $F$
and a perfect normal subgroup $P$ of $F$ such that $F/P$ is isomorphic to $E$.
For $n$ large enough and $\mathfrak G_n$ as in Theorem \ref{improved}, 
there exists by Proposition \ref{main}
a normal subgroup $K_n$ of $F$ such that $F/K_n$ 
is isomorphic to $\mathfrak G_n$.
Observe that the group
\begin{equation*}
F / PK_n \, \simeq \, \left( F / K_n \right) / \left(P / P \cap K_n \right)
\, \simeq \, \mathfrak G_n / \left(P / P \cap K_n \right)
\end{equation*}
is a quotient of $\mathfrak G_n$ by a perfect normal subgroup.
\par
Since $\mathfrak G_n$ has a subgroup of index a power of $2$
which is a direct product of free groups 
(Theorem \ref{improved}, or \cite{BaCo--06}), 
$\mathfrak G_n$ is residually soluble.
It follows that the only perfect subgroup of $\mathfrak G_n$ is $\{e\}$,
so that $F / PK_n \simeq \mathfrak G_n$.
Hence $F/PK_n$ is large,
and so is its cover $F/P \simeq E$.
\end{proof}

\appendix

\section{\textbf{On soluble groups, metabelian groups, and finite presentations}}
\label{AppendixAMetabelian}

The existence of groups that are finitely generated and infinitely presented
was established by B.H.\ Neumann in 1937.
More precisely, he constructed \emph{uncountably many $2$-generator groups}
\cite[Theorem 14]{Neum--37};
at most countably many of them are finitely presented.
Later it was checked that none of them is finitely presented 
(see the last proof of Appendix \ref{AppendixCBNS}, as well as \cite[Theorem C]{BaMi--09}).
\par

Infinitely generated \emph{soluble} groups are equally abundant,
as we recall below after having fixed some notation.
\par

The groups of the derived series of a group $G$
are defined inductively by $D^0G=G$ and $D^{\ell+1}G=[D^\ell G,D^\ell G]$.
The \textbf{free soluble group of rank $k$ and solubility class $\ell$}
is the quotient $\operatorname{FSol}(k,\ell) = F_k / D^\ell F_k$,
where $F_k$ stands for the free group of rank $k$.
Any $k$-generated soluble group of solubility class at most $\ell$
is a quotient of $\operatorname{FSol}(k,\ell)$.
A group $G$ is \textbf{metabelian} if $D^2G = 1$,
namely if it is a cover of an abelian group with abelian kernel.
The group $\operatorname{FSol}(k,2)$ 
is the \textbf{free metabelian group of rank $k$}.
\par

Philip Hall established the existence of 
\emph{uncountably many finitely generated soluble groups.}
His result is much more precise \cite[Theorem 6]{Hall--54}:
given any countable abelian group $A \ne 1$, there exist
uncountably many groups $G$ such that
\begin{equation*}
d(G) \, = \, 2, \hskip.5cm
Z(G) \, \simeq \, A, \hskip.5cm
[G , D^2G] \, = \, 1 .
\end{equation*}
Here $d(G)$ stands for the minimal number of generators of $G$,
and $Z(G)$ for its centre.
The condition $[G , D^2G] = 1$ can be translated in words:
$G$ is a \emph{centre-by-metabelian group}.
It is moreover known that
there are uncountably many finitely generated soluble groups
which are not quasi-isometric to each other \cite[Corollary 1.8]{CoTe}.
\par

On the contrary, 
there are only countably many finitely generated metabelian groups
(this is repeated as Proposition \ref{HallUncountablySolGroups} below),
and more generally\footnote{Let $\mathcal P$ and $\mathcal Q$ be group properties.
A group $G$ is $\mathcal P$-by-$\mathcal Q$ if $G$ has a normal subgroup $N$
with Property $\mathcal P$ such that $G/N$ has Property $\mathcal Q$.}
abelian-by-polycyclic groups
(\cite[Corollary 2 to Theorem 3]{Hall--54}, 
see also \cite[Corollary 4.2.5]{LeRo--04}).
Before comparing soluble groups in general with metabelian groups in particular,
we collect some well-known facts in the following lemma.

Recall that a group $G$ satisfies 
\textbf{Max-n}, \textbf{the maximal condition for normal subgroups},
if any increasing sequence of normal subgroups of $G$ is ultimately stationary,
or equivalently if any normal subgroup of $G$ 
is finitely generated \emph{as normal subgroup.}

\begin{lem}  
\label{consequenceofHall'slemma}
Let $G$ be a finitely generated group, $N$ a normal subgroup,
and $Z$ a central subgroup.

(i) If $G/N$ is finitely presented, 
there exists a finite subset $S \subset N$ such that
$N$ is the smallest normal subgroup of $G$ containing $S$.

(ii) If $G/Z$ is finitely presented,
then $Z$ is finitely generated.

(iii) If $G$ has uncountably many normal subgroups,
then $G$ has uncountably many pairwise non-isomorphic quotients.

(iv) Suppose that $G$ is finitely presented and satisfies Max-n.
Then $G/N$ is finitely presented.
\end{lem}

\begin{proof}
Claim (i) is  \cite[Lemma 14.1.3]{Robi--96}.
It is a simple consequence of the following
fundamental observation of B.H. Neumann:
let $S,S'$ be two finite generating sets of a group $G$;
assume that $G$ has a finite presentation 
$\langle S \mid R \rangle$
involving $S$ and a finite set $R$ of relators;
then there exists a finite set $R'$ of relators in the letters of $S'$
such that $\langle S' \mid R' \rangle$
is also a finite presentation of $G$
\cite[Lemma 8]{Neum--37}.
\par

Claim (ii) is the special case of (i) for a central subgroup.
\par

For Claim (iii), consider
an uncountable family $(N_\alpha)_{\alpha \in A}$
of distinct normal subgroups of $G$.
Fix $\alpha \in A$.
Let $B$ be a subset of $A$ such that,
for each $\beta \in B$, there exists an isomorphism
$\phi_\beta : G/N_\beta \longrightarrow G/N_\alpha$.
It suffices to show that $B$ is countable.
\par

For $\beta \in B$, let $\pi_\beta$ denote the composition
of the canonical projection $G \longrightarrow G/N_\beta$ with $\phi_\beta$.
Since $G$ is finitely generated and $G/N_\alpha$ countable,
there are only countably many homomorphisms from $G$ to $G/N_\alpha$.
As $N_\beta = \ker (\pi_\beta)$, the set $B$ is countable.
\par

For Claim (iv), consider a finite presentation of $G$,
namely a free group $F$ on a finite set $S$
and a normal subgroup $M$ of $F$ generated as normal subgroup
by a finite subset $R$ of $F$, such that $G = F/M$.
Since $G$ satisfies Max-n, there exists a finite subset $R'$ of $F$
of which the image in $G$ generates $N$ as a normal subgroup.
Then $\langle S \mid R \cup R' \rangle$
is a finite presentation of $G/N$.
\par

Note that Claim (ii) is a special case of \cite[Lemma 2]{Hall--54}.
Our argument for Claim (iii) can be found in \cite[Page 433]{Hall--54},
and that for Claim (iv) is ``a well-known principle'' cited in \cite[Page 420]{Hall--54}.
\end{proof}

Finitely generated metabelian groups are ``well-behaved'' in many ways:

\begin{prop}[Hall, Baumslag, Remeslennikov]
\label{MetabelianWellBehaved}
Let $G$ be a finitely \emph{generated} metabelian group.
\begin{itemize}
\item[(i)]
$G$ satisfies Max-n.
In particular, the centre of $G$ is finitely generated.
\item[(ii)]
If $G$ is finitely presented,
so is any quotient of $G$.
\item[(iii)]
$G$ is residually finite.
\item[(iv)]
$G$ has a soluble word problem.
\item[(v)]
$G$ can be embedded into a finitely presented metabelian group.
\item[(vi)]
$G$ is recursively presented.
\end{itemize}
\end{prop}

\begin{proof}[References]
Claim (i) is \cite[Theorem 3]{Hall--54},
Claim (ii) follows by Lemma \ref{consequenceofHall'slemma}.iv,
and Claim (iii) is  \cite[Theorem 1]{Hall--59}.
\par
For the particular case of free metabelian groups,
Claim (iii) follows from linearity:
it is known that $\operatorname{FSol}(k,2)$ 
is a subgroup of $\operatorname{GL}_2(\C)$.
This is a form of the ``Magnus embedding theorem'' ;
see \cite{Magn--39}, and also \cite[Theorem 2.11]{Wehr--73}.
(On the contrary, $\operatorname{FSol}(k,\ell)$ is not linear when $\ell \ge 3$;
see Remark \ref{linearityofFSol}.)
\par

Claim (iv) can be found in \cite{Baum--74}.
It is also a consequence of (a particular case of)
a result of Wehrfritz: any finitely generated metabelian group is quasi-linear,
namely is a subgroup of a group of the form $\prod_{i=1}^r \operatorname{GL}_n(F_i)$,
where $F_1, \hdots, F_r$ are fields \cite{Wehr--80}.
More generally, several algorithmic problems are known to be soluble
in finitely generated metabelian groups \cite{BaCR--94}.
\par

Claim (v) was proved by Baumslag \cite{Baum--73} 
and Remeslennikov \cite{Reme--73}, independently.
See also \cite[Proposition 11.3.2]{LeRo--04}.
\par

Claim (vi) is  \cite[Corollary A1]{Baum--74};
it also follows from Claim~(v).
Note that Claim (vi) is contained in Claim (iv),
but we add it for comparison with Proposition \ref{HallAbelsEtc}.
\end{proof}

\begin{prop}[P.\ Hall]
\label{HallUncountablySolGroups}
There are countably many finitely generated metabelian groups.
\end{prop}

\begin{rem!}
\label{RemarkOnHall}
(a)
Let $G$ be a finitely generated metabelian group $G$;
let $k \ge 0$ be such that $G$ can be generated by $k$ elements,
so that $G$ is a quotient of the free metabelian group $\operatorname{FSol}(k,2)$.
Though it need not be finitely presented
(examples are shown below),
the group $G$ is finitely presented \emph{as a metabelian group},
because $\operatorname{FSol}(k,2)$ satisfies Max-n;
in other terms, we can write
\begin{equation*}
G \, = \, \operatorname{FSol}(k,2)/ \langle\langle r_1, \hdots, r_n \rangle\rangle ,
\end{equation*}
where the notation $\langle\langle \cdots \rangle\rangle$ indicates 
a normal subgroup generated \emph{as such} 
by elements $r_1, \hdots, r_n$ 
in $\operatorname{FSol}(k,2)$.
Proposition \ref{HallUncountablySolGroups} follows.
\par
Note that Proposition \ref{HallUncountablySolGroups}
is also a straightforward consequence of Claim (vi) in the previous proposition.
\par

(b) Some of the claims in Proposition \ref{MetabelianWellBehaved}
can be improved. For example, 
(i) holds for finitely generated abelian-by-polycyclic groups,
and (iii) holds for finitely generated abelian-by-nilpotent groups (Hall).
Moreover (iii) holds for abelian-by-polycyclic groups,
as shown by Roseblade and Jategaonkar in 1973 and 1974
(see \cite{Sega--75}, or
Chapter 7 and in particular Theorem 7.2.1 in \cite{LeRo--04}).

(c)
Until the early 70's, 
there were rather few known examples of finitely presented metabelian groups.
The 3-generator 3-relator group
\begin{equation*}
H \, = \, 
\langle a,s,t \mid a^t = a a^s, \hskip.1cm [s,t] = 1 = [a, a^s] \rangle ,
\end{equation*}
appeared independently in papers by 
Baumslag \cite{Baum--72} and Remes\-lennikov \cite{Reme--73}; 
see also \cite[Theorem A]{Stre--84}.
It is metabelian, its derived group is free abelian of infinite rank,
and it contains the wreath product $\Z \wr \Z$ as a subgroup
\cite[Pages 72--73]{Baum--74}.
It was quite a surprise at this time \cite[first lines]{Baum--72}
to find a finitely presented group
containing a normal abelian subgroup of infinite rank.
More recently,
the quotient group $H / (a^2 = 1)$ was the main character in \cite{GLSZ--00}.
%
\end{rem!}

For groups of higher solubility degrees, 
the picture is substantially different,
even under the stronger hypothesis of finite presentability.
Each of the claims of the next proposition is meant
to be compared with the corresponding claim
of Proposition \ref{MetabelianWellBehaved}.

\begin{prop}
\label{HallAbelsEtc}
Let $G$ be a finitely \emph{presented} soluble group.
\begin{itemize}
\item[(i)]
$G$ need not satisfy Max-n.
Indeed, the centre of $G$ need not be finitely generated.
\item[(ii)]
$G$ may have uncountably many quotients,
and in particular infinitely presented quotients.
\item[(ii')]
Any metabelian quotient of $G$ is finitely presented.
\item[(iii)]
$G$ need not be residually finite.
\item[(iv)]
$G$ need not have a soluble word problem.
\end{itemize}
Let $G$ be now a finitely \emph{generated} soluble group.
\begin{itemize}
\item[(v)]
$G$ need not be recursively presented.
\item[(vi)]
$G$ need not embed into any finitely presented group.
\end{itemize}
\end{prop}

\begin{proof}[On the proof]
Let $p$ be a prime.
For $n \ge 2$, consider the group $A_n$ 
of $n$-by-$n$ triangular matrices of the form
\begin{equation*}
\left(
\begin{array}{ccccc}
1 & * & \cdots & *  & *  \\
0 & * & \cdots & *  & *  \\
\vdots & \vdots & \ddots  & \vdots & \vdots \\
0 & 0 &  \cdots &  * & *  \\
0 & 0 & \cdots & 0 & 1
\end{array}
\right )
\end{equation*}
with upper-triangular entries $a_{i,j}$, $1 \le i < j \le n$, in $\Z [\frac{1}{p}]$ and 
diagonal entries $a_{2,2}, \hdots, a_{n-1,n-1}$ in $p^\Z$.
This group is soluble.
Its center $Z(A_n)$ is isomorphic to $\Z [\frac{1}{p}]$, that is to $A_2$,
and therefore is not finitely generated; 
it follows that $A_n$ does not satisfy Max-n.
It is easy to check that $A_n$ is finitely generated when $n \ge 3$.
\par

For $n \ge 4$, the group $A_n$ is finitely presented 
(\cite{Abel--77} for $n=4$ and \cite{AbBr--87} for $n \ge 4$).
This justifies Claim (i).
The existence of a finitely presented soluble group without Max-n
solves a problem of P.\ Hall; 
Remeslennikov had an earlier claim for this \cite{Reme--72}
which was apparently unjustified \cite{Thom--77}.
\par

Note that $A_3$ is infinitely presented.
This was most likely known to P.\ Hall,  and can be found in \cite{AbBr--87}.
But it is also a consequence of Bieri-Strebel Theorem \ref{ThBieriStrebel};
indeed, since  $Z(A_3)$ is not finitely generated,
$A_3/Z(A_3)$ is infinitely presented (Lemma \ref{consequenceofHall'slemma}),
and the soluble group $A_3$ cannot be a finitely presented cover
of the metabelian group $A_3/Z(A_3)$.
\par

Since  $Z(A_n) \simeq \Z [\frac{1}{p}]$ is not finitely generated,
the quotient $A_n / Z(A_n)$, with $n \ge 3$, 
is finitely generated non-finitely presented,
by Lemma \ref{consequenceofHall'slemma}.ii.
When $n \ge 4$, this justifies the second part of Claim (ii).
\par 

Claim (ii') follows from Theorem \ref{ThBieriStrebel}.
\par

For $n \ge 3$, the quotient of $A_n$ by the central subgroup
\begin{equation*}
\left\{ (z_{i,j})_{1 \le i,j \le n} \in A_n 
\hskip.2cm \Big\vert \hskip.2cm
\aligned
& z_{i,j} = \delta_{i,j} \hskip.2cm \text{for} \hskip.2cm (i,j) \ne (1,n)
\\
& z_{1,n} \in \Z 
\endaligned
\right\}    
\, \simeq \, \Z
\end{equation*}
(where $(\delta_{i,j})_{1 \le i,j \le n}$ denotes the unit matrix)
is finitely generated non-Hopfian\footnote{
A group $G$ is \textbf{non-Hopfian} if there exists
a surjective endomorphism of $G$ onto itself that is not injective.
A residually finite finitely generated group is Hopfian \cite{Mal'c--40}.}
(the argument of \cite{Hall--61} for $n=3$ carries over to all $n \ge 3$),
and therefore non-residually finite.
When $n \ge 4$, this justify Claim (iii).
\par

Still for $n \ge 4$, it is known that
the quotient $A_n /Z(A_n)$ does satisfy Max-n
and does not have any minimal presentation
\cite[Lemma 3.2 and Corollary 3.6]{BCGS}.
The last statement means that any presentation of $A_n/Z(A_n)$
contains redundant relators; 
in particular, the finitely related group $A_n$ has a quotient 
that is not finitely related.
The group $A_n$ itself has only countably many quotients;
see \cite[Theorem 1]{Lyul--84} and \cite[Corollary 3.4]{BCGS}.
\par

Concerning Claim (iv),
finitely presented soluble groups with unsoluble word problems
have been constructed by Kharlampovich in 
\cite{Khar--81}
and by Baumslag, Gildenhuys and Strebel in \cite{BaGS--86}.
Groups in \cite{BaGS--86} have centers that are not finitely generated,
and therefore have infinitely presented quotients 
(see again Claim (ii)).
\par

Earlier, Meskin had constructed a finitely generated 
recursively presented residually finite soluble group
with unsoluble word problem \cite{Mesk--74}.
\par

For any prime $p$, there exists 
a finitely presented soluble groups $G$ 
with centre $(C_p)^{(\infty)}$, an infinite direct sum
of cyclic groups of order $p$
(see \cite{Khar--90}, as well as \cite[Lemma 4.14]{KhMS}).
Hence $G$ has uncountably many quotients,
by Lemma \ref{consequenceofHall'slemma}.iii.
This justifies the first part of Claim (ii).
\par

Claim (v) follows from the existence 
of uncountably many finitely generated soluble groups,
because there are only countably many recursively presented groups.
\par

(Digression: Let $G$ be a finitely generated soluble group.
Assume that $G$ is of finite Pr\"ufer rank, namely that there exists an integer $d \ge 1$
such that any finitely generated subgroup of $G$ can be generated by $d$ elements.
Then $G$ is recursively presented if and only if
$G$ has a soluble word problem.
For this, and for examples of $G$ with and without soluble word problem,
see \cite{CaRo--84}.)
\par

Claim (vi) follows from Claim (v)
because a finitely generated subgroup
of a finitely presented group is recursively presented.
(This is straightforward; see the first page of \cite{Higm--61},
where Higman establishes the famous non-trivial converse;
alternatively, see \cite[Lemma 2.1]{Mill--89}.)
\end{proof}

\begin{rem!} 
\label{linearityofFSol}
As noted parenthetically just after Proposition \ref{MetabelianWellBehaved}, 
free solvable groups
$\operatorname{FSol}(k,\ell)$ are not linear when $k \ge 2$ and $\ell \ge 3$
\cite{Smir--64, Smir--65}.
Here is a proof.
\par

For $k \ge 3$, here is first a short reduction to a more standard result.
Let $\mathfrak C$ denote the class of (nilpotent-by-abelian)-by-finite groups;
observe that quotients of groups in $\mathfrak C$ are in $\mathfrak C$.
The Lie-Kolchin-Mal'cev theorem 
(see for example \cite[Section 15.1]{Robi--96}) establishes
that soluble linear groups are in $\mathfrak C$.
If $\operatorname{FSol}(k,\ell)$ were linear, hence in $\mathfrak C$,
so would be any quotient,  in particular the iterated wreath product
$(\Z \wr \Z) \wr \Z$; but this is not \cite[15.1.5]{Robi--96}.
\par

The following argument, shown to us by Ralph Strebel,
holds for any $k \ge 2$.
Denote by $x_1, \hdots, x_k$ a set of free generators of $\operatorname{FSol}(k, \ell)$.
Let $H$ be a subgroup of finite index in $\operatorname{FSol}(k, \ell)$;
there exists an integer $m \ge 1$ such that $x_1^m, x_2^m \in H$;
let $U$ denote the subgroup of $\operatorname{FSol}(k, \ell)$ 
generated by $x_1^m$ and $x_2^m$,
so that 
\begin{equation*}
U \,  \subset \,  H \, \subset_{\text{finite index}} \, \operatorname{FSol}(k, \ell) .
\end{equation*}
The group $U$ is isomorphic to $\operatorname{FSol}(k, \ell)$,
by a result due independently to Gilbert Baumslag \cite[Theorem 2]{Baum--63}
and Shmel'kin \cite[Theorem 5.2]{Shme--64}.
\par

If $\operatorname{FSol}(k, \ell)$ was linear, 
it would have a nilpotent-by-abelian subgroup of finite index, say $H$.
By the lines above, 
$\operatorname{FSol}(2, \ell)$ would be nilpotent-by-abelian.
Hence any $2$-generated soluble group of solubility class at least $3$
would be nilpotent-by-abelian. But this is not true. 
\par

Indeed, consider the symmetric group $S_4$ on four letters.
It is generated by two elements, a transposition and a $4$-cycle.
It is soluble of class $3$, with $D^1S_4 = A_4$ (the alternating group on four letters),
$D^2S_4 = V$ (the Klein Vierergruppe), and $D^3S_4 = \{1\}$. 
It is not nilpotent-by-abelian, namely $A_4$ is not nilpotent: $[A_4,V] = V$.
\par
The argument of Smirnov is different. It relies on the
bi-orderability of the groups $\operatorname{FSol}(k, \ell)$.
\end{rem!}

\section{\textbf{On wreath products and lamplighter groups}}
\label{AppendixBWreath}

Permutational wreath products have been defined 
in the beginning of Section \ref{sectionSelfsimilar}.
The \textbf{standard wreath product} $G \wr_H H$
refers to the action of $H$ on itself by left multiplications.

\begin{prop}
\label{finitenesswreath}
Consider two groups $G,H$, a non-empty $H$-set $X$, and the
permutational wreath product $G \wr_X H$.
We assume that $G \ne \{1\}$.
\par

(i)
$G \wr_X H$ is finitely generated if and only if 
$G,H$ are finitely generated and $H$ has finitely many orbits on $X$.
\par

(ii)
$G \wr_X H$ is finitely presented if and only if 
$G,H$ are finitely presented,
the $H_x$'s ($x \in X$) are finitely generated,
and there are finitely many orbits in $X \times X$ 
for the diagonal action of $H$
(where $H_x = \{h \in H \mid h(x) = x \}$).

(iii)
In particular, as soon as $H$ is infinite,
the standard wreath product $G \wr_H H$ is \emph{not} finitely presented.
\end{prop}

\begin{proof}[References]
Claim (i) is standard, and easy; 
if necessary, see \cite[Proposition 2.1]{Corn--06}.
For Claim (ii), see \cite[Theorem 1.1]{Corn--06}.
Claim (iii) is the main result of \cite{Baum--61}.
\end{proof}

As a particular case of Proposition \ref{finitenesswreath},
if $G$ is finitely-generated abelian and $G \ne 1$, 
the group $G \wr \Z$ is
metabelian, finitely generated, and infinitely presented.
When $G$ is finite abelian
and $H \simeq \Z$ infinite cyclic,
we will call  $G \wr \Z$ the \textbf{lamplighter group} for $G$.
(For this terminology,  precise assumptions on $G$ and $H$ 
vary from one author to the other; some ask that $G = \Z / 2\Z$.)
\par

\begin{prop}
\label{AwrZ}
Consider two finitely presented groups $G,H$, 
an $H$-set $X$ such that $H$ has finitely many orbits on $X$
and infinitely many orbits on $X \times X$, 
and the permutational wreath product $G \wr_X H$.
We assume that $G \ne \{1\}$.
\par
(i) $G \wr_X H$ is finitely generated and is infinitely presented.
\par
(ii) For any finitely presented cover
$\pi : E \twoheadrightarrow G \wr_X H$,
the group $E$ has non-abelian free subgroups.
\end{prop}

\begin{proof}
Claim (i) is a particular case of Proposition \ref{finitenesswreath}.
In Proposition 2.10 of \cite{Corn--06}, it is shown that the kernel of $\pi$
contains non-abelian free subgroups.
\end{proof}

Note that, in the particular case of two abelian groups $G$ and $H$,
Claim (ii) is also a consequence of Theorem \ref{ThBieriStrebel}.

\begin{proof}[On other proofs of Proposition \ref{AwrZ} in the case of $W := \Z \wr_{\Z} \Z$]
(We have $G = H = \Z$ and $X = \Z$.)
We have a presentation
\begin{equation*}
W \,  = \, 
\langle s,t \mid [s^{t^i}, s^{t^j}] \hskip.2cm \forall i,j \in \Z \rangle ;
\end{equation*}
indeed, any element in the right-hand side can be written as
\begin{equation*}
\aligned
& 
s^{m_1}t^{n_1}s^{m_2}t^{n_2}s^{m_3}t^{n_3} \cdots s^{m_\ell}t^{n_\ell} 
\\
& \hskip.5cm \, = \,
s^{m_1} 
\big( s^{t^{-n_1}} \big)^{m_2} 
\big( s^{t^{-n_1 - n_2}} \big)^{m_3}    \cdots 
\big( s^{t^{-n_1 - \cdots - n_{\ell-1}}} \big)^{m_\ell}
t^{n_1 + \cdots + n_\ell}
\endaligned
\end{equation*}
for some $m_1, n_1, \hdots, m_\ell, n_\ell \in \Z$, and therefore as
\begin{equation*}
\big(s^{t^{j_1}}\big)^{i_1} \big( s^{t^{j_2}} \big)^{i_2} \cdots \big( s^{t^{j_k}} \big)^{i_k} t^N
\end{equation*}
for appropriate $i_1, j_1, \hdots, i_k, j_k, N \in \Z$ with $i_1 < i_2 < \cdots < i_k$.
It follows that the natural homomorphism 
\begin{equation*}
\langle s,t \mid [s^{t^i}, s^{t^j}] \hskip.2cm \forall i,j \in \Z \rangle
\longrightarrow W
\end{equation*}
is an isomorphism.
\par

Since $t^i [s, s^{t^k}] t^{-i} = [s^{t^{i}}, s^{t^{i+k}}]$,
we have a second presentation
\begin{equation*}
W \,  = \, 
\langle s,t \mid [s, s^{t^i}] \hskip.2cm \forall i \in \N \rangle .
\end{equation*}
For a positive integer $n$, define
\begin{equation*}
W_n \, = \, \langle s,t \mid [s, s^{t^i}], \hskip.1cm i = 0, \hdots, n \rangle . 
\end{equation*}
Note that $\lim_{n \to \infty} W_n= W$ in $\mathcal M_2$.
We have a third presentation
\begin{equation*}
W_n \, = \, 
\left\langle 
s_0, \hdots, s_n, t \hskip.1cm \bigg\vert \hskip.1cm
\aligned
&[s_i, s_j] , \hskip.1cm 0 \le i,j \le n , \hskip.1cm 
\\
&s_k^t = s_{k+1}, \hskip.1cm 0 \le k \le n-1 
\endaligned
\right\rangle .
\end{equation*}
Indeed, it can be checked that the assignments
\begin{equation*}
\aligned
&\varphi_1 \, : \, s \longmapsto s_0, \hskip.2cm t \longmapsto t
\\
&\varphi_2 \, : \, s_i \longmapsto s^{t^i}, \hskip.2cm t \longmapsto t
\hskip1cm (0 \le i \le n)
\endaligned
\end{equation*}
define, between the groups of the two previous presentations,
isomorphisms that are inverse to each other.
\par

Let $H_n$ be the free abelian subgroup of $W_n$
generated by $s_0, \hdots, s_n$.
Denote by $K_n$ the subgroup of $H_n$ generated by $s_0, \hdots, s_{n-1}$,
and by $L_n$ that generated by $s_1, \hdots, s_n$;
observe that $K_n \simeq L_n \simeq \Z^n$.
Let $\psi_n : K_n \longrightarrow L_n$ be the isomorphism
defined by $\psi(s_{i-1}) = s_i$ for $i = 1, \hdots, n$. 
Then $W_n$ is clearly the HNN-extension
corresponding to the data $(H_n, \psi_n : K_n \overset{\simeq}{\rightarrow} L_n)$.
By Britton's lemma, $W_n$
contains non-abelian free subgroups.
\par
It follows from Corollary \ref{main2} that any finitely presented cover of $W$
contains non-abelian free groups.
\par

   Let us finally allude to another argument showing that $W$ is infinitely presented.
Let $H$ denote the subgroup of $\operatorname{GL}_3(\C)$
generated by the three matrices
\begin{equation*}
\left(
\begin{array}{ccc}
1 & 0 & 0  \\
0 & t & 0  \\
0 & 0 & 1
\end{array}
\right ) ,
\hskip.2cm
\left(
\begin{array}{ccc}
1 & 1 & 0  \\
0 & 1 & 0  \\
0 & 0 & 1
\end{array}
\right ) ,
\hskip.2cm \text{and} \hskip.2cm
\left(
\begin{array}{ccc}
1 & 0 & 0  \\
0 & 1 & 1  \\
0 & 0 & 1
\end{array} 
\right ) ,
\end{equation*}
where $t$ is some transcendental number.
The centre $Z(H)$ of $H$ is free abelian of infinite rank,
and the quotient $H/Z(H)$ is isomorphic to $M := \Z[t,t^{-1}]^2 \rtimes \Z$,
where $\rtimes$ refers to the action of the generator $1 \in \Z$
by $\left( \begin{array}{cc} t & 0 \\ 0 & t^{-1} \end{array} \right)$.
There is an isomorphism of $M$ onto a subgroup of index $2$ 
in $\Z[t,t^{-1}] \rtimes_t \Z \simeq \Z \wr \Z$,
given by $(P(t), Q(t), n) \longmapsto (P(t^2) + tQ(t^{-2}), 2n)$.
It follows from Lemma \ref{consequenceofHall'slemma}.ii
that $\Z \wr \Z$ is not finitely presented.
We are grateful to Adrien Le Boudec for correcting an earlier version
of our argument at this point.
\par

The isomorphism of $H/Z(H)$ with a subgroup of index $2$ in $\Z \wr \Z$
is essentially due to P.\ Hall.
See \cite[Theorem 7]{Hall--54} and \cite[Lemma 3.1]{CaCo}.
\end{proof}

Concerning Proposition \ref{AwrZ}, let us add one more remark
about the particular case $W = (\Z / h\Z) \wr \Z$, with $h \ge 2$:
it is known that any finitely presented cover of $W$ is large
\cite[Section IV.3, Theorem~7]{Baum--93}.

\vskip.2cm

Recall that the free soluble group 
$\operatorname{FSol}(k,\ell) = F_k / D^\ell F_k$
of rank $k$ and solubility class $\ell$
has been defined in Appendix \ref{AppendixAMetabelian}.

\begin{cor}
\label{freeSol}
For $k,\ell \ge 2$, 
the group $\operatorname{FSol}(k,\ell)$ is infinitely presented,
and any finitely presented cover of it contains non-abelian free subgroups.
\end{cor}

\noindent
\emph{Note.}
(i)
That $\operatorname{FSol}(k,\ell)$ is infinitely presented
is a result due to \cite{Shme--65}.
See also \cite[Proposition 2.10 and Corollary 2.14]{Corn--06}.
\par
(ii)
Corollary \ref{freeSol} and our proof carry over to 
free polynilpotent groups 
\begin{equation*}
\operatorname{FPolynilp}(k, \ell_1, \hdots, \ell_n) \, := \,  
F_k/ C^{\ell_k}( C^{\ell_{k-1}}( \cdots C^{\ell_1}(F_k) \cdots ))
\end{equation*}
for any $k \ge 2$, $n \ge 2$, and $\ell_1, \hdots, \ell_n \ge 2$
(where $C^\ell G$ denotes the $\ell$th group of the lower central series of a group $G$,
defined by $C^1G = G$, and $C^{\ell + 1}G  = [G, C^\ell G]$).
These groups are the subject of \cite{Shme--64}.

\begin{proof}
Since $\Z \wr \Z$ is a two-generator metabelian group,
we have an epimorphism $\operatorname{FSol}(2,2) \twoheadrightarrow \Z \wr \Z$.
Indeed, we have a sequence of natural epimorphisms
\begin{equation*}
\operatorname{FSol}(k,\ell) \twoheadrightarrow 
\operatorname{FSol}(k,2)  \twoheadrightarrow
\operatorname{FSol}(2,2)  \twoheadrightarrow
\Z \wr \Z .
\end{equation*}
Hence any finitely presented cover of $\operatorname{FSol}(k,\ell)$ 
is also one of $\Z \wr \Z$.
If $\operatorname{FSol}(k,\ell)$ was finitely presented,
it would contain non-abelian free subgroup by Proposition \ref{AwrZ},
but this cannot be in a soluble group.
\end{proof}

In the situation of Proposition \ref{AwrZ}, suppose moreover that
$H$ is infinite residually finite,
$G$ has at least one non-trivial finite quotient,
and consider the standard wreath product ($X = H$).
The following strengthening of Claim (ii) is shown in \cite[Theorem 1.5]{CoKa--11}:
any finitely presented cover of $G \wr H$ is \emph{large}.
In particular:
\begin{center}
\emph{any finitely presented cover of $\Z \wr \Z$ is large.}
\end{center}
By the proof of Corollary \ref{freeSol}, it follows that,
\begin{center}
\emph{for $k,\ell \ge 2$, any finitely presented cover of $\operatorname{FSol}(k,\ell)$
is large.}
\end{center}

\vskip.2cm

The following notion provides interesting examples of metabelian groups,
as we will illustrate with Baumslag-Solitar groups.

\begin{defn!}
The \textbf{metabelianization} of a group $G$ is the 
metabelian quotient group $G / D^2G$.
\end{defn!}

\begin{defn!}[\cite{BaSo--62}]
\label{BSmet}
For $\ell,m \in \Z \smallsetminus \{0\}$, 
the \textbf{Baumslag-Solitar group}  is defined by
the two-generators one-relator presentation
\begin{equation*}
\operatorname{BS}(\ell,m) \, = \, 
\langle s,t \mid t^{-1}s^\ell t = s^m \rangle .
\end{equation*}
\end{defn!}

We collect three well-known properties of these groups as follows.

\begin{prop}
\label{BSbasique}
Let $\ell,m \in \Z \smallsetminus \{0\}$ and $\operatorname{BS}(\ell,m)$ be as above.
\begin{itemize}
\item[(i)]
$\operatorname{BS}(\ell,m)$ is abelian 
if and only if $\operatorname{BS}(\ell,m)$ is nilpotent,
if and only if $\ell = m = \pm 1$.
\item[(ii)]
$\operatorname{BS}(\ell,m)$ is metabelian 
if and only if $\operatorname{BS}(\ell,m)$ is soluble, 
if and only if $\operatorname{BS}(\ell,m)$ does not contain non-abelian free subgroups,
if and only if $\vert \ell \vert = 1$ or $\vert m \vert = 1$.
\item[(iii)]
If $\ell,m$ satisfy $\ell, m \ge 2$ and are \emph{coprime},
then $\operatorname{BS}(\ell,m)$ is non-Hopfian.
\item[(iv)]
For $\ell,m$ as in (iii), the group $\operatorname{BS}(\ell,m)$
contains non-abelian free groups, but is not large.
\end{itemize}
\end{prop}

\begin{proof}[On the proof]
It is easy to check that the four groups
$\operatorname{BS}(\ell,m)$, $\operatorname{BS}(m,\ell)$,
$\operatorname{BS}(-\ell,-m)$, $\operatorname{BS}(-m,-\ell)$
are isomorphic. For simplicity, let us assume from now on
that $\ell$ and $m$ are positive.
(For the general case, with all details, we refer to \cite{Souc--01}).
\par
It is an exercise to check that $\operatorname{BS}(1,m) \simeq 
\Z \left[ \frac{1}{m} \right] \rtimes_m \Z$ for any $m \ge 1$.
It follows that $\operatorname{BS}(\ell,m)$ is metabelian 
if $\ell = 1$ or $m = 1$,
and abelian if and only if $\ell = m = 1$.
If $m \ge 2$, note that $\operatorname{BS}(1,m)$ is not nilpotent,
because its subgroup $\Z \left[ \frac{1}{m} \right]$
is not finitely generated.
If $\ell \ge 2$ and  $m \ge 2$,
the subgroup generated by $s^{-1}ts$ and $t$ is free of rank $2$,
by Britton's Lemma.

Claim (iii) is the main reason for the celebrity of these groups. 
It is straightforward to check that the assignments
$\varphi (s) = s^\ell$ and $\varphi (t) = t$
define an endomorphism $\varphi$ of $\operatorname{BS}(\ell,m)$.
The image of $\varphi$ contains $t$ and $s^\ell$,
hence $t^{-1}s^\ell t^{-1} = s^m$, and therefore $s$;
hence $\varphi$ is onto.
On the one hand, 
$\varphi( [t^{-1} s t, s]) = [t^{-1} s^\ell t, s^\ell] = [s^m, s^\ell] = 1$;
on the other hand, 
$[t^{-1} s t, s] = t^{-1} s t \hskip.1cm s \hskip.1cm t^{-1} s^{-1} t \hskip.1cm s^{-1} \ne 1$,
where the last inequality holds by Britton's Lemma;
hence $\varphi$ is not one-to-one.

More generally, we know necessary and sufficient conditions on $\ell,m$ 
for $\operatorname{BS}(\ell,m)$ to be non-Hopfian; 
see \cite{BaSo--62, Coll--78, CoLe--83}.

The first part of Claim (iv) is standard
(it also follows from Theorem \ref{ThBieriStrebel},
see \ref{BSmetab}.ii below). 
For the second part, see Example 3.2 and Theorem 6 in \cite{EdPr--84}.
\end{proof}

\begin{defn!}
\label{defmet(l,m)}
For two coprime positive integers $\ell,m$, not both $1$,
let
\begin{equation*}
\operatorname{Met}(\ell,m) \,= \, 
\left(
\begin{array}{cc}
  \big(\frac{\ell}{m} \big)^{\Z} & \Z \left[ \frac{1}{\ell m} \right] \\
  0   & 1 
\end{array} 
\right )  \, \simeq \,
\Z \left[ \frac{1}{\ell m} \right] \rtimes_{\ell / m} \Z
\end{equation*}
be the group of triangular matrices generated by
\begin{equation*}
\left(
\begin{array}{cc}
  1 & 1 \\
  0   & 1 
\end{array} 
\right ) 
\hskip.2cm \text{and} \hskip.2cm
\left(
\begin{array}{cc}
  \frac{\ell}{m} & 0 \\
  0   & 1 
\end{array} 
\right ) .
\end{equation*}
Let
\begin{equation*}
\mu_{\ell,m} \, : \, 
\operatorname{BS}(\ell,m) \, \twoheadrightarrow  \operatorname{Met}(\ell,m)
\end{equation*}
denote the epimorphism defined by
\begin{equation*}
\mu_{\ell,m}(s) =
\left(
\begin{array}{cc}
  1 & 1 \\
  0   & 1 
\end{array} 
\right ) 
\hskip.2cm \text{and} \hskip.2cm
\mu_{\ell,m}(t) = 
\left(
\begin{array}{cc}
  \frac{\ell}{m} & 0 \\
  0   & 1 
\end{array} 
\right ) .
\end{equation*}
\end{defn!}

The following proposition collects
facts on $\operatorname{BS}(\ell,m)$ and $\operatorname{Met}(\ell,m)$.
Claims (iii) to (v) constitute a digression from our theme.

\begin{prop}  
\label{BSmetab}
Let the notation be as just above, and $\ell,m$ 
be two \emph{coprime} positive integers, not both $1$.

(i) 
$\mu_{\ell,m}$ is an isomorphism if and only if $\ell = 1$ or $m = 1$.
\par

\noindent
We assume furthermore that  $\min \{\ell,m\} \ge 2$.
\par

(ii)
$\operatorname{Met}(\ell,m)$ is infinitely presented.
\par

(iii)
The multiplicator group $H_2(\operatorname{Met}(\ell,m), \Z)$ is trivial.
\par

(iv)
$\operatorname{Met}(\ell,m)$
is of cohomological dimension~$3$.
\par

(v)
For $x \in \C$ transcendental,
the matrices
$
\left(
\begin{array}{cc}
  1 & 1 \\
  0   & 1 
\end{array} 
\right ) 
$
and
$
\left(
\begin{array}{cc}
  x  & 0 \\
  0   & 1 
\end{array} 
\right ) 
$
generate a group isomorphic to $\Z \wr \Z$.
\end{prop}

\begin{proof}[On proofs.]
Claim (i) has already been given as an exercise,
in the proof of Proposition \ref{BSbasique}.
Claim (ii) is a consequence of a particular case 
of the main result of \cite{BaSt--76},
or a consequence of \cite[Theorem C]{BiSt--78};
see also \cite[Proposition 11.4.3]{LeRo--04}. 
Claim (iv) is \cite[Theorem 4]{Gild--79}.
Claim (v) is \cite[Proposition 3.1.4]{LeRo--04}.
\par

For (iii), see \cite[No 1.8]{BaSt--76}.
Recall that, if a group $G$ is finitely presented,
then its multiplicator group $H_2(G, \Z)$ is finitely generated;
this is a simple consequence of the so-called \emph{Schur-Hopf Formula},
for a group $G = F/R$ presented as a quotient of a free group $F$,
which reads
``$H_2(G, \Z) = (R \cap [F,F]) / [R,F]$''.
Claim (iii) is one of the standard examples showing
that the converse \emph{does not} hold.
\end{proof}

Let $\ell,m$ be coprime positive integers, with $\ell,m \ge 2$.
We denote by $p_{\ell,m} : F_2 \twoheadrightarrow \operatorname{BS}(\ell,m)$ 
the defining cover of the corresponding Baumslag-Solitar group,
namely the cover mapping a basis of the free group of rank $2$ onto $\{s,t\}$.
Let  $\varphi : \operatorname{BS}(\ell,m) \twoheadrightarrow \operatorname{BS}(\ell,m)$ 
be the usual non-injective surjective endomorphism,
as in Proposition \ref{BSbasique}.
For $n \ge 1$, set
\begin{equation*}
M_n \, = \, \ker( \varphi^n ), \hskip.2cm
N_n \, = \, \ker( \varphi^n p_{m,l} ), \hskip.2cm
G_n \, = \, \operatorname{BS}(\ell,m) / M_n \, = \, F_2 / N_n .
\end{equation*}
Observe that the sequence $(M_n)_{n \ge 1}$ is strictly increasing,
yet $G_n$ is isomorphic to $\operatorname{BS}(\ell,m)$ for each $n \ge 1$.

\begin{prop}
\label{BS+GrMa}  
Let $\ell,m$ be coprime integers, with $\ell,m \ge 2$,
and $\operatorname{Met}(\ell,m)$ as in Definition \ref{defmet(l,m)}.
Let $(G_n)_{n \ge 1}$ be as above.
\par

(i) $\operatorname{Met}(\ell,m)$ is isomorphic to 
$\operatorname{BS}(\ell,m)_{\operatorname{metab}}$.
\par

(ii) With the notation above, we have an isomorphism
\begin{equation*}
\operatorname{Met}(\ell,m) \, \simeq \, F_2 / \Big( \bigcup_{n \ge 1} N_n \Big) ,
\end{equation*}
so that $\lim_{n \to \infty}G_n =  \operatorname{Met}(\ell,m)$
in $\mathcal M_2$.
\par

(iii) Any finitely presented cover of $\operatorname{Met}(\ell,m)$
contains non-abelian free subgroups.
\end{prop}

\begin{proof}
Claim (i) is part of \cite[Theorem G]{Baum--74}.
\par

Claim (ii) is  \cite[see 1.8]{BaSt--76} or \cite[Theorem 3]{GrMa--97}, 
there for $(\ell,m) = (2,3)$, but the argument carries over to the case stated here.
(When working on \cite{GrMa--97}, the authors were not aware of \cite{BaSt--76}.)

\par

Since $\operatorname{BS}(\ell,m)$ contains non-abelian free subgroups 
by Proposition \ref{BSbasique}.ii, 
and since $\lim_{n \to \infty} G_n = \operatorname{Met}(\ell,m)$, 
Claim (iii) follows by Corollary \ref{main2}.
\end{proof}

\section{\textbf{On Bieri-Neumann-Strebel invariants}}
\label{AppendixCBNS}

\begin{exe!}
\label{motivationBS}
Consider the two metabelian groups
\begin{equation*}
\aligned
\operatorname{Met}(1,6) \, &= \,  \operatorname{BS}(1,6)
\, = \, \Z \left[1/6\right] \rtimes_{1/6} \Z
\\
\operatorname{Met}(2,3) \, &= \, \operatorname{BS}(2,3)_{\operatorname{metab}}
\, = \, \Z \left[1/6\right] \rtimes_{2/3} \Z .
\endaligned 
\end{equation*}
The first group is finitely presented.
By contrast, the second group is infinitely presented
(Propositions \ref{BSmetab}).
This spectacularly different behavior of two superficially similar-looking groups
was the initial motivation of Bieri and Strebel (later joined by Neumann and others) 
for the work that lead to the $\Sigma$-invariants; 
see \cite{Bier--79}, as well as \cite[Problem 1]{Baum--74}.
We come back to these two examples in Example \ref{examplesBS}.
\end{exe!}

Let $G$ be a finitely generated group, $S$ a finite generating set,
and $\operatorname{Cay}(G,S)$ the corresponding \textbf{Cayley graph},
with vertex set $G$ and edge set
$\{ \{g,h\} \mid g^{-1}h \in S \cup S^{-1} \}$.
The set of \textbf{characters}, namely of group homomorphisms from $G$ to $\R$, 
is a real vector space $\operatorname{Hom}(G,\R)$ isomorphic to $\R^n$,
where $n$ is the torsion-free rank of the abelian group $G/D^1G$.
Let $S(G)$ denote the \textbf{character sphere} of $G$, quotient of
the set of non-zero characters by the natural action of the group of positive reals;
we have $S(G) \approx \mathbf{S}^{n-1}$, 
and we write $[\chi]$ the class in $S(G)$
of a character $\chi \ne 0$ in $\operatorname{Hom}(G,\R)$. 
(The sign $\approx$ indicates a homeomorphism;
$\mathbf{S}^k$ denotes the $k$-sphere, and $\mathbf{S}^{-1} = \emptyset$.)
For such a $[\chi]$, set
\begin{equation*}
G_{[\chi]} \, = \, \left\{g \in G \mid \chi(g) \ge 0 \right\},
\end{equation*}
and let $\operatorname{Cay}_{[\chi]}(G,S)$ be the subgraph of $\operatorname{Cay}(G,S)$
spanned by $G_{[\chi]}$.
The \textbf{Bieri-Neumann-Strebel invariant}, 
or shortly \textbf{BNS-invariant}, of $G$ is defined to be
\begin{equation*}
\Sigma^1(G) \, = \,
\left\{
[\chi] \in S(G) \mid \operatorname{Cay}_{[\chi]}(G,S) \hskip.2cm \text{is connected}
\right\} \, \subset \, S(G) .
\end{equation*}
The superscript indicates that $\Sigma^1$ is just one out of many ``geometric invariants''
(see \cite{BiRe--88} and \cite{Stre}).
\par

Observe that there is an antipodal map $[\chi] \longmapsto -[\chi] := [-\chi]$ defined on $S(G)$.
Let $G_0$ be a finitely generated group and $p : G_0 \twoheadrightarrow G$
an epimorphism;
then $p$ induces a map $p^* : S(G) \longrightarrow S(G_0)$
which a homeomorphism onto its image
and which intertwines the antipodal maps.
\par

On the one hand, the invariant of \cite{BiNS--87} involves one more group $A$ 
on which $D^1G$ acts in an appropriate way; here, we particularize to $A = D^1 G$.
On the other hand, the definition above is a reformulation of the original definition,
see for example \cite{Stre}.

\begin{thm}
\label{BNSgeneral}
Let $G$ be a finitely generated group, and $\Sigma^1(G) \subset S(G)$ as above.
\par

(i)
$\Sigma^1(G)$ is independent of the choice of $S$.

(ii) 
$\Sigma^1(G)$ is open in $S(G)$.

(iii)
$\Sigma^1(G) = S(G)$ if and only if
the derived group $D^1G$ is finitely generated.

(iv)
Suppose that $G$ is metabelian.
Then $G$ is finitely presented if and only if
$\Sigma^1(G) \cup (-\Sigma^1(G)) = S(G)$,
and $G$ is polycyclic if and only if 
$\Sigma^1(G) = S(G)$.

(v) 
If $G$ is finitely presented and has no non-abelian free subgroup,
then $\Sigma^1(G) \cup (-\Sigma^1(G)) = S(G)$.

(vi)
If $G_0$ is a finitely generated group and $p : G_0 \twoheadrightarrow G$
an epimorphism,
then  $\Sigma^1(G_0) \cap p^*(S(G)) \subset p^*(\Sigma^1(G))$.
In particular:
\begin{equation*}
\Sigma^1(G) \cup (-\Sigma^1(G)) \subsetneqq S(G)
\hskip.2cm \Longrightarrow \hskip.2cm
\Sigma^1(G_0) \cup (-\Sigma^1(G_0)) \subsetneqq S(G_0) .
\end{equation*}
\end{thm}

\begin{proof}[On the proof]
For (i) see \cite[Theorem A2.3]{Stre}.
\par
For (ii), see \cite[Theorem A]{BiNS--87} and \cite[Theorem A3.3]{Stre}.
\par
(iii) If $D^1G$ is finitely generated, 
then $\Sigma^1(G) = S(G)$, \cite[Proposition A2.6]{Stre}.
For the ``iff'', see \cite[Theorem B1]{BiNS--87}. 
\par
For (iv), see \cite[Theorem A]{BiSt--80} and \cite[Subsection B3.2c]{Stre}.
\par
For (v), see \cite[Theorem C]{BiNS--87} and \cite[Theorem A5.1]{Stre}.
\par
Claim (vi) is rather straightforward from the definitions
\cite[Proposition A4.5]{Stre}.
\end{proof}

\begin{exe!}
\label{examplesBS}
(i)
$\Sigma^1(G) = S(G) \approx \mathbf S^{n-1}$
for $G$ a finitely generated abelian group of torsion-free rank $n$.
\par

(ii)
$\Sigma^1(F_n) = \emptyset \subset S(F_n) \approx \mathbf S^{n-1}$ 
for the non-abelian free group $F_n$ of rank $n \ge 2$; 
see \cite[Item A2.1a]{Stre}.

(iii)
For the soluble Baumslag-Solitar group $\operatorname{BS}(1,n)$,
the invariant $\Sigma^1(\operatorname{BS}(1,n))$ 
is one of the two points of the sphere $S(\operatorname{BS}(1,n)) \approx \mathbf S^0$;
the argument of \cite[Item A2.1a]{Stre} for $n=2$ carries over to $n \ge 2$.
\par

(iv)
For two coprime integers $\ell,m \ge 2$, we have
$\Sigma^1( \operatorname{Met}(\ell,m)) = \emptyset \subset \mathbf S^0$ 
 \cite[Item A3.6]{Stre}.

\par

(v) If $G$ is a semi-direct product $H \rtimes \Z$ of an infinite locally finite group $H$ 
by an infinite cyclic group, then  $\Sigma^1(G) = \emptyset \subset \mathbf S^0$ 
\cite[Lemma B3.1]{Stre}.

\par

(vi) For any $n \ge 1$ and any rational polyhedral subset $P$ of $\mathbf{S}^{n-1}$,
there exists a finitely presented group $G$ with $S(G) \approx \mathbf{S}^{n-1}$
and a homeomorphism $p^* : \mathbf {S}^{n-1} \longrightarrow S(G)$
such that $\Sigma^1(G) = p^* (S(G) \smallsetminus P)$.
\end{exe!}

Some comment is in order for (vi),
cited here to show that the invariant $\Sigma^1(G)$ can be more complicated
than those of Examples (i) to (v).
For a finitely generated group $G$, a non-zero character $\chi \in \operatorname{Hom}(G,\R)$
is \emph{rational} if $\chi(G)$ is an infinite cyclic subgroup of $\R$.
Denote by $\operatorname{Hom}^*_{\Q}(G,\R)$ the set of non-zero rational characters on $G$,
and by $S_{\Q}(G)$ its image in $S(G)$;
then $S_{\Q}(G)$ is a dense subset of $S(G)$ \cite[Lemma B3.3]{Stre}.
\par

Consider a positive integer $n$ and the sphere $\mathbf S^{n-1} = S(\Z^n)$.
A \emph{rational hemisphere} of $\mathbf S^{n-1}$ 
is the closure of the image of the half-space
$\{ \chi \in \operatorname{Hom}^*_{\Q}(G,\R) \mid \chi(z) \ge 0 \}$, 
for some $z \in \Z^n \smallsetminus \{0\}$.
A \emph{rational convex polyhedral subset} of $\mathbf S^{n-1}$
is a finite intersection of rational hemispheres.
A \emph{rational  polyhedral subset} of $\mathbf S^{n-1}$
is a finite union of rational convex polyhedral subsets.

Given an integer $n \ge 1$ and a rational polyhedral subset $P \subset \mathbf{S}^{n-1}$,
there is a finitely presented group $G$ 
and an epimorphism $p : G \twoheadrightarrow \Z^n$
such that $p^* : S(\Z^n) \longrightarrow S(G)$
is a homeomorphism, and $\Sigma^1(G) = p^*(S(\Z^n) \smallsetminus P)$.
See \cite[Corollary 7.6]{BiNS--87} and \cite[Chapter IV, Section 1.1]{BiSt}.

\begin{cor}
\label{BNScorollary}
Let $G$ be a finitely generated group and $E$ a finitely presented cover of $G$.
If $\Sigma^1(G) \cup (-\Sigma^1(G)) \subsetneqq S(G)$, 
then $E$ contains non-abelian free subgroups.
This holds in particular:
\par

when $G$ is metabelian and infinitely presented,

when $G = H \rtimes \Z$ with $H$ infinite locally finite, as in Example \ref{examplesBS}.v.
\end{cor}

\begin{proof}
This is straightforward from Claims (iii) to (v) of Theorem \ref{BNSgeneral}
and from Example \ref{examplesBS}.v. 
\end{proof}

In Corollary \ref{BNScorollary},
the claim concerning metabelian groups is precisely Theorem \ref{ThBieriStrebel}.

\begin{exe!}[B.H.\ Neumann]
\label{exNeumann}
Let $\mathcal V = \{2 \le v_1 < v_2 < v_3 < \cdots \}$ 
be an infinite increasing sequence of integers.
Set 
\begin{equation*}
X_{\mathcal V} \, = \,  \{ (i,j) \in \Z^2 \mid i \ge 1, -v_i \le j \le v_i \} .
\end{equation*}
For each $i \ge 1$, denote by $X_{\mathcal V,i}$ the subset
$\{(i,j) \in X_{\mathcal V} \mid -v_i \le j \le v_i \}$, of cardinal $2v_i + 1$.
Define two permutations $\alpha_{\mathcal V}, \beta_{\mathcal V}$
of the set $X_{\mathcal V}$ as follows: for each $i \ge 1$, 
they preserve $X_{\mathcal V,i}$, and
\begin{itemize}
\item[($\alpha_{\mathcal V}$)]
the restriction of $\alpha_{\mathcal V}$ to $X_{\mathcal V,i}$
is the $(2v_i+1)$-cycle 
\item[]
$((i,-v_i), (i,-v_i+1), \hdots, (i, v_i))$,
\item[($\beta_{\mathcal V}$)]
the restriction of $\beta_{\mathcal V}$ to $X_{\mathcal V,i}$
is the $3$-cycle 
\item[]
$((i, -1), (i, 0), (i, 1))$.
\end{itemize}
The \textbf{Neumann group} corresponding to $\mathcal V$
is the group $G_{\mathcal V}$ of permutations of $X_{\mathcal V}$
generated by $\alpha_{\mathcal V}$ and $\beta_{\mathcal V}$.
\par

Let $\operatorname{Alt}_f(\Z)$ denote 
the group of permutations of finite supports of $\Z$ that are even on their support.
Let $\operatorname{Alt}_f(\Z) \rtimes_{\operatorname{shift}} \Z$
denote its semi-direct product with $\Z$, where $\Z$ acts on itself by shifts,
the generator $1$ acting by $\alpha : j \longmapsto j+1$.
Observe that $\operatorname{Alt}_f(\Z) \rtimes_{\operatorname{shift}} \Z$
is generated by $\alpha$ and the $3$-cycle $\beta = (-1, 0, 1)$.
It is easy to check that the assignment
$\alpha_{\mathcal V} \longmapsto \alpha, \hskip.1cm \beta_{\mathcal V} \longmapsto \beta$
extends to an epimorphism $\pi_{\mathcal V} : G_{\mathcal V} \twoheadrightarrow
\operatorname{Alt}_f(\Z) \rtimes_{\operatorname{shift}} \Z$.
Neumann has shown that the kernel of $\pi_{\mathcal V}$
is the restricted product ${\prod'}_{i = 1}^{\infty} \operatorname{Alt}(2v_i+1)$;
see \cite{Neum--37}, as well as \cite[Complement III.35]{Harp--00}.
\par

This has two straightforward consequences.
On the one hand, 
any minimal finite normal subgroup of $G_{\mathcal V}$
is one of the $\operatorname{Alt}(2v_i+1)$.
Thus, for two distinct sequences $\mathcal V$ and $\mathcal V'$,
the groups $G_{\mathcal V}$ and $G_{\mathcal V'}$ are not isomorphic.
In particular,
there are uncountably many pairwise non-isomorphic $2$-generator groups;
hence these are infinitely presented, 
except possibly for a countable number of them (but see below).
On the other hand, the groups $G_{\mathcal V}$ are elementary amenable.
\end{exe!}

\begin{cor}
\label{BNSrem}
Let $\mathcal V$ be a sequence of integers as in Example \ref{exNeumann}.
Any finitely presented cover of $G_{\mathcal V}$ contains non-abelian free subgroups.
\par
In particular $G_{\mathcal V}$ is not finitely presented.
\end{cor}

\begin{proof}
The claim of Corollary \ref{BNScorollary} concerning $H \rtimes \Z$ 
applies to the group
$\operatorname{Alt}_f(\Z) \rtimes_{\operatorname{shift}} \Z$.
Since any cover of $G_{\mathcal V}$ is a cover of
$\operatorname{Alt}_f(\Z) \rtimes_{\operatorname{shift}} \Z$,
Corollary \ref{BNSrem} follows from Corollary \ref{BNScorollary}.
\end{proof}

Let us also indicate how the last statement of
Corollary \ref{BNSrem} is a straightforward consequence of
the first paper \cite{BiSt--78} on BNS-invariants.

\begin{proof}[Proof that $G_{\mathcal V}$ is not finitely presented]
It follows from the definition of $G_{\mathcal V}$ 
that the kernel $N_{\mathcal V}$ of the composition
$G_{\mathcal V}
\twoheadrightarrow \operatorname{Alt}_f(\Z) \rtimes_{\operatorname{shift}} \Z
\twoheadrightarrow \Z$
is locally finite.
Suppose (ab absurdo) that $G_{\mathcal V}$ is finitely presented.
By \cite[Theorem A]{BiSt--78}, we have
$G_{\mathcal V} = \operatorname{HNN}(H, \varphi : K \overset{\simeq}{\rightarrow} L)$
for a finitely generated subgroup $H$ of $N_{\mathcal V}$
and an isomorphism $\varphi$ between two subgroups $K,L$ of $H$.
Since non-ascending extensions contain non-abelian free subgroups (by Britton's lemma),
$G_{\mathcal V}$ is an ascending HNN-extension; we may assume that $K=H$.
Since $N_{\mathcal V}$ is locally finite, the finitely generated subgroup $H$ of $N_{\mathcal V}$
is finite. It follows that $K=H=L$. Hence $G_{\mathcal V} = H \rtimes_\varphi \Z$
and the kernel of $G_{\mathcal V} \twoheadrightarrow \Z$ is finite.
This is preposterous, and the proof is complete.
\end{proof}

\section{\textbf{On growth and amenability}}
\label{AppendixDGrowthAmenability}

Let $G$ be a group generated by a finite set $S$.
For an integer $n \ge 0$, let $B_S^G(n)$ denote the 
``ball of radius $n$ around the origin'',
namely the set of those elements $g \in G$ 
that can be written as words
$g=s_1 \cdots s_n$, with $s_1, \hdots, s_n \in S \cup S^{-1} \cup \{1\}$.
Let $\gamma_S^G(n)$ denote 
the cardinality of $B_S^G(n)$.
Then $G$ is said to be
\begin{itemize}
\item[(pol)] of \textbf{polynomial growth} if there exist constants $a,d > 0$
such that $\gamma_S^G(n) \le a n^d$ for all $n \ge 1$,
\item[(exp)]
of \textbf{exponential growth} if there exist a constant $c > 1$
such that $\gamma_S^G(n) \ge c^n$ for all $n \ge 0$,
\item[(int)]
of \textbf{intermediate growth} in other cases.
\end{itemize}
It is easy to check that this trichotomy depends 
only on $G$, not on the finite generating set $S$.
For information on the growth of groups, 
we refer to the books \cite{Harp--00} and \cite{Mann--11}.
\par

A group $G$ is \textbf{amenable}
if there exists a left-invariant finitely additive probability measure 
defined on all subsets of $G$ (there are many other equivalent definitions). 
Two basic results are important here:
(i) amenability of groups is preserved by the four operations of 
taking subgroups, quotients, directs limits,
and extensions with amenable kernels
(already in \cite{vNeu--29}),
and (ii) groups of intermediate growth are amenable
(this goes back to \cite{AdSr--57},
and is also a straightforward consequence of F\o lner's Criterion \cite{Foln--55}).
These results make it natural to define three classes of groups:
\begin{itemize}
\item[$\mathcal A \mathcal G$]
is the \textbf{class of amenable groups}, 
defined in \cite{vNeu--29}.
\item[$\mathcal E \mathcal G$]
is the \textbf{class of elementary amenable groups},
defined in \cite{Day--57};
it is the smallest class of groups containing the easiest examples,
that are finite groups and abelian groups, 
and stable by the four operations listed above.
\item[$\mathcal S \mathcal G$]
is the \textbf{class of subexponentially amenable groups}
(see below for an historical comment on this definition);
it is the smallest class of groups containing $\mathcal E \mathcal G$
and the next easiest examples, 
that are the groups of intermediate growth.
\end{itemize}
We have a partition
\begin{equation*}
\mathcal A \mathcal G \, = \, 
(\mathcal A \mathcal G \smallsetminus \mathcal S \mathcal G)
\, \sqcup \, 
(\mathcal S \mathcal G \smallsetminus \mathcal E \mathcal G)
\, \sqcup \, 
\mathcal E \mathcal G .
\end{equation*}
Let us mention a few groups in each of these three parts.

\vskip.2cm

The class $\mathcal E \mathcal G$
contains all virtually soluble groups and all locally finite groups;
other examples are cited in Subsection \ref{Motivation}.
\par

There are finitely generated groups in $\mathcal E \mathcal G$
which are not virtually soluble: for example all Neumann groups $G_{\mathcal V}$
discussed in Example \ref{exNeumann},
or an example in \cite{Hill--91}.
As already mentioned in Subsection \ref{DualQuestion1.1},
any countable elementary amenable group embeds in a
finitely generated elementary amenable group;
the same hold for ``amenable'' instead of ``elementary amenable''  \cite[Corollary 1.3]{OlOs}.
\par

Let us describe a family of
finitely presented elementary amenable groups that are not virtually soluble.

\begin{exe!}
\label{Houghton}
Consider an integer $n \ge 1$, the set
\begin{equation*}
S_n \, = \, \{ (j,k) \in \Z^2 \mid j \ge 0, \hskip.1cm 1 \le k \le n \}
\end{equation*}
of $n$ parallel half-intervals in the square lattice, and the \textbf{Houghton group} $H_n$
of all permutations $h$ of $S_n$ such that,
for each $k \in \{1, \hdots, n\}$, there exists a translation $t_k \in \Z$
such that $h(j,k) = (j+t_k, k)$ for all $j$ large enough \cite{Houg--78}.
Denote by $\operatorname{Sym}_f(S_n)$ the group of permutations of $S_n$
with finite supports, clearly a normal subgroup of $H_n$, and set
\begin{equation*}
A \, = \, \{ (t_1, \hdots, t_n) \in \Z^n \mid \sum_{k=1}^n t_k = 0 \} .
\end{equation*}
We have a short exact sequence
\begin{equation*}
1 \, \longrightarrow \, \operatorname{Sym}_f(S_n) \, \longrightarrow \, H_n 
\, \overset{\pi}{\longrightarrow} \,  A \, \longrightarrow \, 1
\end{equation*}
where, with the notation above, $\pi(h) = (t_1, \hdots, t_n)$.
Since $\operatorname{Sym}_f(S_n)$ is locally finite and $A$ abelian,
$H_n$ is elementary amenable. 
\par

It is known that the group $H_n$ 
is of type FP$_{n-1}$ but not FP$_n$ \cite[Theorem 5.1]{Brow--87}.
In particular, for $n \ge 3$, the group $H_n$ is finitely presented.
\end{exe!}

Finitely generated groups 
in the class $\mathcal E \mathcal G$
are either of polynomial growth or of exponential growth
\cite{Chou--80};
this has been sharpened:
a finitely generated group 
in the class $\mathcal E \mathcal G$
has either polynomial growth 
or \emph{uniform} exponential growth \cite{Osin--04}
(see also \cite{Breu--07}).
By a famous theorem of Gromov \cite{Grom--81}, 
a finitely generated group of polynomial growth is virtually nilpotent,
and in particular finitely presented.

\vskip.2cm

The class 
$\mathcal S \mathcal G \smallsetminus \mathcal E \mathcal G$
contains the class of finitely generated groups of intermediate growth.
Historically, the group $\mathfrak G$ of 
Theorem \ref{GrHa+i} and Example \ref{Ex1stGrig}
was the first group shown to be of intermediate growth \cite{Grig--83}.
This class also contains finitely presented groups,
such as the group with $5$ generators and $11$ relators of \cite{Grig--98},
later shown to have another presentation with $2$ generators and $4$ relators
(due to Bartholdi, see \cite[Number 12]{CeGH--99}).
\par

The history of early papers on the class $\mathcal S \mathcal G$ is worth a few lines.
This class was first implicitly introduced in a paper on 4-manifold topology,
more precisely on 4-manifold surgery and 5-dimensional s-cobordism theorems,
\cite{FrTe--95} (see also \cite{KrQu--00}),
and then explicitly in \cite{Grig--98}.
Freedman and Teichner introduce a class of groups  that they call ``good'',
defined as the groups for which the ``$\pi_1$-Null Disk Lemma'' holds;
this lemma establishes the existence of $2$-discs bounding some closed curves
in 4-manifolds of a certain kind.
Good groups include finitely generated groups in the class $\mathcal S \mathcal G$
\cite[Theorem 0.1 and Lemma 1.2]{FrTe--95}.

\vskip.2cm

The class
$\mathcal A \mathcal G \smallsetminus \mathcal S \mathcal G$
contains the \textbf{Basilica group} $\mathfrak B$ of Example \ref{ExBasilica},
which was first shown to be not in $\mathcal S \mathcal G$ \cite{GrZ--02a},
and later shown to be amenable \cite{BaVi--05}. 
The method of Bartholdi and Virag 
was streamlined and generalized by Kaimanovich in \cite{Kaim--05},
in terms of entropy and of the legendary ``M\"unchhausen's trick''.
This and later papers show the amenability of $\mathfrak B$
and of many other non elementary amenable groups
(see \cite{BaKN--10, AmAV}, building among other things on \cite{Sidk--00}).
The class
$\mathcal A \mathcal G \smallsetminus \mathcal S \mathcal G$
contains also the finitely generated amenable simple groups
that appear in \cite{JuMo}, and in Problem \ref{ProblemOnJuMo}.

\vskip.2cm

Non-amenable groups include non-abelian free groups, 
more generally groups containing non-abelian free subgroups \cite{vNeu--29}.

\vskip.2cm

We conclude this report by Question \ref{lastquestion},
due to Tullio Ceccherini-Silberstein.
Before this, we recall some background.
\par

A group $G$ has a \textbf{paradoxical decomposition} if there exist
integers $p,q \ge 2$, subsets $X_1, \hdots, X_p$, $Y_1, \hdots, Y_q$ of $G$,
and elements 
$g_1, \hdots, g_p$, $h_1, \hdots, h_q $ in $G$ such that
\begin{equation*}
\aligned
G \, &= \, X_1 \sqcup \cdots \sqcup X_p \sqcup Y_1 \sqcup \cdots Y_q
\\
\, &= \, g_1X_1 \sqcup \cdots \sqcup g_pX_p 
\, = \, h_1Y_1 \sqcup \cdots \sqcup h_qY_q ,
\endaligned
\end{equation*}
where $\sqcup$ denotes disjoint union.
Tarski has shown that $G$ is non-amenable if and only if 
$G$ has a paradoxical decomposition \cite{Tars--38};
see also \cite{Wago--85} and  \cite{HaSk--86}.
For $G$ non-amenable, the \textbf{Tarski number} $\mathcal T (G)$ of $G$
is the minimum of the sum $p+q$,
over all paradoxical decompositions of $G$.
When $G$ has non-abelian free subgroups, it is easy to show that $\mathcal T (G) = 4$;
in particular, if $G$ is a non-elementary Gromov-hyperbolic group, then $\mathcal T (G) = 4$.
As a student of Tarski in the 1940's, 
Jonsson has shown that, conversely, $\mathcal T (G) = 4$ 
implies that $G$ has non-abelian free subgroups.
It is also easy to show that $\mathcal T (G) \ge 6$ for a non-amenable torsion group $G$.
See \cite{Wago--85} or \cite[Propositions 20 and 21]{CeGH--99}.
We do not know any example of a non-amenable group $G$ 
without non-abelian free subgroups
for which the exact value of $\mathcal T (G)$ has been computed.
\par

Let $\mathcal M_m^{\text{na}}$ be the subspace of $\mathcal M_m$ of those pairs $(G,S)$
with $G$ non-amenable.
For $m \ge 2$, the Tarski number function 
\begin{equation*}
\mathcal M_m^{\text{na}} \longrightarrow \N, \hskip.2cm (G,S) \longmapsto \mathcal T (G)
\end{equation*}
is not continuous.
Indeed, there are sequences of non-elementary Gromov-hyperbolic groups
(with Tarski number $4$)
converging in $\mathcal M_m$ to non-amenable torsion groups 
(with Tarski number at least $6$).
See \cite[Th\'eor\`eme 1.3]{Cham--00}, or one of the  following classes of examples.
\par

(1)
For integers $m,n \ge 2$, let $\operatorname{B}(m,n)$ denote the free Burnside group
of rank $m$ and exponent $n$, that is the quotient of the free group $F_m$
by the set of relators $(w^n = 1)_{w \in F_m}$.
For $n$ odd and $n \ge 665$, 
Adyan has shown that $\operatorname{B}(m,n)$ is non-amenable \cite{Adya--82};
moreover, it is known that $6 \le \mathcal T (\operatorname{B}(m,n)) \le 14$ \cite[Theorem 61]{CeGH--99}.
\par

Such groups are limits in $\mathcal M_m$ of non-elementary Gromov-hyperbolic groups.
More precisely, for $n$ odd and large enough, 
the group $\operatorname{B}(m,n)$ is a limit in $\mathcal M_m$ 
of a sequence $\big( \operatorname{B}(m,n,i) \big)_{i \ge 1}$
of non-elementary Gromov-hyperbolic groups.
A similar fact is shown by Ivanov in \cite[see Lemma 21.1]{Ivan--94},
in the much more difficult case of $n$ even, 
with $n \ge 2^{48}$ and $n$ divisible by $2^9$;
his proof adapts to the case needed here, with important simplifications
(compare with \cite{GrIv--09}, in particular Theorems 1.10 and 1.7,
and recall that a finitely presented group with a subquadratic Dehn function is hyperbolic).
In particular, for $m \ge 2$ and $n$ odd large enough,
the free Burnside group $\operatorname{B}(m,n)$, 
of Tarski number between $6$ and $14$,
is a limit in $\mathcal M_m$ (indeed in $\mathcal M_m^{\text{na}}$)
of groups of Tarski number~$4$.
\par

(2)
Ol'shanskii has worked out several ``Tarski monster groups''.
In particular, for any prime $p$ large enough,
he has constructed
a $2$-generated non-amenable torsion group $\operatorname{TM}(p)$
in which any proper non-trivial subgroup is of order $p$.
Moreover $\operatorname{TM}(p)$, of Tarski number $\ge 6$, 
is a limit in $\mathcal M_2$ of non-elementary Gromov hyperbolic groups,
of Tarski number $4$ \cite[Lemma 10.7a]{Ol'sh--80}.
\par

In several other papers by Ol'shanskii and co-authors,
there are several other classes of examples of such Burnside type limits of hyperbolic groups.
Let us only quote one paper \cite{OlOS--09}, and the book \cite{Ol'sh--91}.

\begin{que}
\label{lastquestion}
For an integer $k \ge 4$ and a finitely generated non-amenable group $G$
with $\mathcal T (G) = k$,
does there exist a finitely presented cover $E$ of $G$ with $\mathcal T (E) = k$~?
\end{que}

The answer is clearly positive when $\mathcal  T(G) = 4$ (with $E$ free).
If we define $\mathcal T (G) = \infty$ when $G$ is amenable,
the previous question for $k=\infty$ coincides with Question \ref{initialquestion},
and the answer is negative.

\vskip.5cm
\begin{center}
Post-scriptum
\end{center}

\section*{A construction of large groups, 
\\
due to Yves de Cornulier}

\vskip.2cm
\begin{center}
By Ralph Strebel
\end{center}
\vskip.3cm

The following proposition implies that any finitely presented cover 
of one of the groups $G_{\mathcal V}$ of Corollary \ref{BNSrem} is large.
The proof uses the basic idea of the proof of Theorem A in \cite{BiSt--78},
refined for Theorem 6.1 in \cite{BCGS}.

\vskip.2cm

\noindent
\textbf{Proposition.}
\emph{
Let $G$ be a finitely generated group
that is an extension of an infinite, locally finite group by an infinite cyclic group.}
\par
\emph{Any finitely presented cover of $G$ is large.}

\vskip.2cm

\begin{proof}
The beginning of our proof holds for any finitely generated group $G$
given with a projection $\chi : G \twoheadrightarrow \mathbf Z$;
let $N$ denote the kernel of $\chi$.
\par
We can choose a finite subset $A$ of $N$
and an element $t \in G$ such that $\chi(t) = 1$ and $A \cup \{t\}$ generates $G$.
Let $\widetilde A$ be a finite set given with a bijection $\tilde a \leftrightarrow a$ with $A$;
we denote by $F$ the free group on a set $\widetilde A \sqcup \{\tilde t \}$.
Let
\begin{equation*}
\rho \, : \,  F \longrightarrow G
\end{equation*}
be the epimorphism defined by $\rho(\tilde a) = a$
for all $\tilde a \in \widetilde A$, and $\rho(\tilde t) = t$.
\par
We construct as follows a sequence $(G_m)_{m \ge 1}$ of finitely generated groups.
The kernel $N$ is the union of subgroups
\begin{equation*}
B_m \, := \, \operatorname{gp} \{ t^\ell a t^{-\ell} \mid -m \le \ell \le m
\hskip.2cm \text{and} \hskip.2cm a \in A  \}, \hskip.5cm m \ge 1;
\end{equation*}
the notation indicates that $B_m$ is the subgroup $\operatorname{gp} \{ \hdots \}$ of $G$
generated by $\{ \hdots \}$.
For $m \ge 1$, consider the two subgroups  
\begin{equation*}
\aligned
H_m \, &:= \, \operatorname{gp} \{ t^\ell a t^{-\ell} \mid -m \le \ell \le m-1
\hskip.2cm \text{and} \hskip.2cm a \in A  \} 
\\
K_m \, &:= \, \operatorname{gp} \{ t^\ell a t^{-\ell} \mid -m+1 \le \ell \le m
\hskip.2cm \text{and} \hskip.2cm a \in A  \} 
\endaligned
\end{equation*}
of $B_m$, the isomorphism
\begin{equation*}
\tau_m \, : \,  H_m \overset{\simeq}{\longrightarrow} K_m, \hskip.2cm
h \longmapsto t h t^{-1},
\end{equation*}
and the HNN-extension
\begin{equation*}
G_m \, := \, \langle B_m, u_m \mid u_m h u_m^{-1} = \tau_m (h)
\hskip.2cm \text{for all} \hskip.2cm h \in H_m \rangle .
\end{equation*}
By Britton's lemma, we can and do view $B_m$ as a subgroup of $G_m$.
Observe that $G_m$ is actually generated by $A \subset H_m$ and $u_m$;
indeed, the defining relations in $G_m$ imply that
\begin{equation*}
u_m^\ell a u_m^{-\ell} \, = \, (\tau_m)^\ell (a) \, = \,
t^{\ell} a t^{-\ell}
\end{equation*}
for all 
$\ell$ with $-m \le \ell \le m$ and $a \in A$.
\par
Let
\begin{equation*}
\lambda_{m,m+1} \, : \, G_m \longrightarrow G_{m+1}
\end{equation*}
be the homomorphism of which the restriction to $B_m$
is the inclusion $B_m \subset B_{m+1} \subset G_{m+1}$, 
and such that $\lambda_{m,m+1}(u_m) = u_{m+1}$.
Let
\begin{equation*}
\lambda_m \, : \, G_m \longrightarrow G
\end{equation*}
be the homomorphism of which the restriction to $B_m$
is the inclusion $B_m \subset G$, 
and such that $\lambda_m(u_m) =t$.
Observe that $\lambda_{m,m+1}$ and $\lambda_m$ are epimorphisms,
and that $\lambda_m = \lambda_{m+1} \lambda_{m, m+1}$.
Let 
\begin{equation*}
\rho_m \, : \,F \longrightarrow G_m
\end{equation*}
be the homomorphism defined by
$\rho_m(\tilde a) = a$ for all $\tilde a \in \widetilde A$ and $\rho_m(\tilde t) = u_m$.
Observe that $\rho_m$ is an epimorphism,
and that $\lambda_m \rho_m = \rho$.
\par
We have clearly
\begin{equation*}
\ker(\rho_1) \, \subset \, \ker(\rho_2) \, \subset \,  \cdots
\subset \ker(\rho_m) \, \subset \,  \cdots \, \subset \ker(\rho) .
\end{equation*}
We claim that, moreover, $\bigcup_{m=1}^\infty \ker(\rho_m) = \ker(\rho)$.
Indeed, let $w \in \ker(\rho)$.
Since $\chi(\rho(w)) = 0 \in \mathbf Z$, the element $w$
is a product of conjugates by powers of $\tilde t$ of elements in 
$\widetilde A \cup {\widetilde A}^{-1}$.
Let $m$ be large enough for $w$ to be a product of elements of the form
${\tilde t}^\ell \tilde a {\tilde t}^{-\ell}$, with $\vert \ell \vert \le m$ 
and $\tilde a \in \widetilde A \cup {\widetilde A}^{-1}$.
Then $\rho(w) = \lambda_m (\rho_m (w)) = 0$,
with $\rho_m (w) \in B_m \subset G_m$.
Since the restriction of $\lambda_m$ to $B_m$ is injective,
this implies $w \in \ker (\rho_m)$, and the claim holds.

\vskip.2cm

We use now one more hypothesis: $N$ is locally finite. 
Since $B_m$ is a finitely generated subgroup in the locally finite group $N$, 
it is a finite group.
It follows from Proposition 11 and its corollary in \cite[Pages 160--161]{Serr--77}
that the HNN-extension $G_m$ 
has a finitely generated free subgroup $F_m$ of finite index.
We use now the last hypothesis: $N$ is infinite.
It implies that $F_m$ cannot be cyclic, 
i.e. that $G_m$ has  a finite index subgroup $F_m$ 
that is non-abelian and free.
In other words, $G_m$ is large.

\vskip.2cm

Consider finally a finitely presented group $E$ and an epimorphism
$\pi : E \twoheadrightarrow G$.
Upon replacing $A$ at the beginning of the proof by a larger finite subset of $N$,
we can assume that there exists an epimorphism
$\sigma : F \twoheadrightarrow E$
of which the kernel is generated by a finite set of relators,
and that $\rho$ factors as the composition 
$F \overset{\sigma}{\twoheadrightarrow} E  \overset{\pi}{\twoheadrightarrow} G$.
Hence $\ker(\sigma) \subset \ker(\rho)$.
Since $\ker(\sigma)$ is the normal closure of a finite set,
there exists $m \ge 1$ with $\ker(\sigma) \subset \ker(\rho_m)$.
It follows that $E = F/K$ has $G_m = F/\ker(\rho_m)$ as a quotient group.
Since $G_m$ is large, we conclude that $E$ is large.
\end{proof}

As a consequence of the proof, we have:

\vskip.2cm

\noindent
\textbf{Corollary.}
\emph{
Let $G$ be a (locally finite)-by-(infinite cyclic)  group, as in the previous proposition.
Assume that $G$ is finitely generated, say that $G$ has a generating set $S$ of some size $k \ge 1$.
}
\par
\emph{
Then, in the space of marked group $\mathcal M_k$, see Section \ref{sectionMark+Chab},
the pair $(G,S)$ is the limit of finitely generated virtually free groups.
More precisely, there exists in $\mathcal M_k$
a sequence $(G_m, S_m)_{m \ge 1}$
of virtually free marked groups converging to $(G,S)$,
and for each $m \ge 1$ an epimorphism $G_m \twoheadrightarrow G$
mapping $S_m$ onto $S$.
}

\end{document}